\documentclass[11pt]{amsart}
\usepackage{amsmath,amsfonts,amsthm,amssymb,url,graphicx}

\DeclareFontFamily{OT1}{rsfs}{}
\DeclareFontShape{OT1}{rsfs}{n}{it}{<-> rsfs10}{}
\DeclareMathAlphabet{\mathscr}{OT1}{rsfs}{n}{it}
\DeclareMathOperator{\sgn}{sgn}

\DeclareMathOperator{\Res}{Res}

\DeclareMathOperator{\mo}{\,mod}

\DeclareMathOperator{\hyp}{hyp}

\DeclareMathOperator{\err}{err}

\DeclareMathOperator{\Si}{Si}

\newtheorem{prop}{Proposition}[section]
\newtheorem{thm}[prop]{Theorem}

\newtheorem{cor}[prop]{Corollary}
\newtheorem{lem}[prop]{Lemma}

\newtheorem*{defn*}{Definition}

\numberwithin{equation}{section}
\title{Major arcs for Goldbach's problem}
\author{H. A. Helfgott}
\address{Harald Helfgott, 
\'Ecole Normale Sup\'erieure, D\'epartement de Math\'ematiques, 45 rue d'Ulm, F-75230 Paris, France}
\email{harald.helfgott@ens.fr}
\begin{document}
\begin{abstract}
The ternary Goldbach conjecture states that every odd number $n\geq 7$ 
is the sum of three primes.
The estimation of the Fourier series $\sum_{p\leq x} e(\alpha p)$ and
related sums has been central to the study of the problem since Hardy and 
Littlewood (1923). 

Here we show how to estimate such Fourier series for $\alpha$ in the so-called
major arcs, i.e., for $\alpha$ close to a rational of small denominator.
This is part of the author's proof of the ternary Goldbach conjecture.

In contrast to most previous work on the subject, we will rely on a 
finite verification of the Generalized Riemann Hypothesis up to a 
bounded conductor and bounded height, rather than on zero-free regions. We 
apply a rigorous verification due to D. Platt; the results we obtain are both
rigorous and unconditional.

The main point of the paper 
will be the development of estimates on parabolic cylinder
functions that make it possible to use smoothing functions based on the
Gaussian. The generality of our explicit formulas will allow us to work with a 
wide variety of such functions.
\end{abstract}
\maketitle
\tableofcontents
\section{Introduction}

The ternary Goldbach conjecture (or {\em three-prime problem}) 
states that every odd number $n$ greater than $5$ can be written as
the sum of three primes. Hardy and Littlewood (1923) were the first to treat
the problem by means of the {\em circle method}, i.e., Fourier analysis over
$\mathbb{Z}$. (Fourier transforms of functions on $\mathbb{Z}$ live on
the {\em circle} $\mathbb{R}/\mathbb{Z}$.)

I. M. Vinogradov \cite{Vin} showed in 1937 that the ternary Goldbach conjecture
 is true for all $n$ above a large constant $C$. The main innovation in his work
consisted in the estimation of sums of the form
\begin{equation}\label{eq:joko}\sum_{n\leq N} \Lambda(n) e(\alpha n)\end{equation}
for $\alpha$ outside the so-called ``major arcs'' -- these being a union
of short intervals in $\mathbb{R}/\mathbb{Z}$
around the rationals with small denominator. (Here $\Lambda(n)$ is the von Mangoldt 
function, defined as $\Lambda(n) = \log p$ for $n$ a power of a prime $p$ and
$\Lambda(n) = 0$ for $n$ having at least two prime factors, whereas
$e(t) = e^{2\pi i t}$.)

The estimation of such sums for $\alpha$ {\em in} the major arcs is also 
important, and goes back to Hardy and Littlewood \cite{MR1555183}. In some ways, their work
is rather modern -- in particular, it studies a version of (\ref{eq:joko})
with smooth truncation:
\[S_{\eta}(\alpha,x) = \sum_n \Lambda(n) e(\alpha n) \eta(n/x),\]
where $\eta:\mathbb{R}^+\to \mathbb{C}$ is a smooth function; in \cite{MR1555183},
 $\eta(t) = e^{-t}$.

We will show how to estimate sums such as $S_{\eta}(\alpha,x)$ for
$\alpha$ in the major arcs. We will
see how we can obtain good estimates by using smooth functions $\eta$
based on the Gaussian $e^{-t^2/2}$. 
This will involve proving new, fully explicit bounds
for the Mellin transform\footnote{See \S \ref{subs:milly} for definitions.} 
of the twisted Gaussian, or, what is the same,
 parabolic cylindrical functions in certain ranges. It will also require
explicit formulae that are general and strong enough, even for moderate
values of $x$.

Any estimate on $S_\eta(\alpha,x)$ for $\alpha$ in the major arcs relies on
the properties of $L$-functions $L(s,\chi) = \sum_n \chi(n) n^{-s}$, where
$\chi:(\mathbb{Z}/q\mathbb{Z})^*\to \mathbb{C}$ is a multiplicative character.
In particular, what is key is the location of the zeroes of $L(s,\chi)$
in the critical strip $0\leq \Re(s)\leq 1$ (a region in which $L(s,\chi)$
can be defined by analytic continuation). In contrast to most previous work,
we will not use zero-free regions, which are too narrow for our purposes.
Rather, we use a verification of the Generalized Riemann Hypothesis
up to bounded height
for all conductors $q\leq 300000$ 
(due to D. Platt \cite{Plattfresh}).

The bounds we will obtain have shown themselves to
be strong enough to prove the ternary Goldbach conjecture. See \cite{Helf}.
A key feature of the present work is that it allows one to mimic a wide 
variety of smoothing functions by means of estimates on the Mellin transform
of a single smoothing function -- here, the Gaussian $e^{-t^2/2}$.
\subsection{Results}\label{subs:results}
Write $\eta_\heartsuit(t) = e^{-t^2/2}$.
Let us first give a bound for exponential sums on the primes
using $\eta_\heartsuit$ as the smooth weight.

\begin{thm}\label{thm:gowo1}
Let $x$ be a real number $\geq 10^8$. 
Let $\chi$ be a primitive character mod $q$, $1\leq q\leq r$, where
$r=300000$.

Then, for any $\delta \in \mathbb{R}$ with $|\delta|\leq 4r/q$,
\[\sum_{n=1}^\infty \Lambda(n) \chi(n) e\left(\frac{\delta}{x} n\right) 
e^{-\frac{(n/x)^2}{2}}
= I_{q=1}\cdot \widehat{\eta_\heartsuit}(-\delta) \cdot x
+ E \cdot x,\]
where $I_{q=1}=1$ if $q=1$, $I_{q=1}=0$ if $q\ne 1$,
and
\[|E|\leq
5.281 \cdot 10^{-22} + \frac{1}{\sqrt{x}}
\left( \frac{650400}{\sqrt{q}} + 112\right).\]
\end{thm}
We normalize the Fourier transform $\widehat{f}$ as follows:
$\widehat{f}(t) = \int_{-\infty}^\infty e(-xt) f(x) dx$.
Of course, 
$\widehat{\eta_\heartsuit}(-\delta)$ is just $\sqrt{2\pi} e^{-2 \pi^2 \delta^2}$.

As it turns out, smooth weights based on the Gaussian are often better
in applications than the Gaussian $\eta_\heartsuit$ itself. Let us
give a bound based on $\eta(t) = t^2 \eta_\heartsuit(t)$.
\begin{thm}\label{thm:janar}
Let $\eta(t) = t^2 e^{-t^2/2}$.
Let $x$ be a real number $\geq 10^8$. 
Let $\chi$ be a primitive character mod $q$, $1\leq q\leq r$, where
$r=300000$.

Then, for any $\delta \in \mathbb{R}$ with $|\delta|\leq 4r/q$,
\[\sum_{n=1}^\infty \Lambda(n) \chi(n) e\left(\frac{\delta}{x} n\right) 
\eta(n/x)
= I_{q=1}\cdot \widehat{\eta}(-\delta) \cdot x
+ E \cdot x,\]
where $I_{q=1}=1$ if $q=1$, $I_{q=1}=0$ if $q\ne 1$,
and
\[|E|\leq \frac{4.269\cdot 10^{-14}}{q} + 
\frac{1}{\sqrt{x}} \left(\frac{276600}{\sqrt{q}} + 56\right).\]
\end{thm}
The advantage of $\eta(t) = t^2 \eta_\heartsuit(t)$ over $\eta_\heartsuit$ is
that it vanishes at the origin (to second order); as we shall see, this makes
it is easier to estimate exponential sums with the smoothing $\eta \ast_M g$,
where $\ast_M$ is a Mellin convolution and $g$ is nearly arbitrary.
Here is a good example that is used, crucially, in \cite{HelfTern}.

\begin{cor}\label{cor:coprar}
Let $\eta(t) = t^2 e^{-t^2/2} \ast_M \eta_2(t)$, where $\eta_2 = \eta_1 \ast_M
\eta_1$ and $\eta_1 = 2\cdot I_{\lbrack 1/2,1\rbrack}$.
Let $x$ be a real number $\geq 10^8$. 
Let $\chi$ be a primitive character mod $q$, $1\leq q\leq r$, where
$r=300000$.

Then, for any $\delta \in \mathbb{R}$ with $|\delta|\leq 4r/q$,
\[\sum_{n=1}^\infty \Lambda(n) \chi(n) e\left(\frac{\delta}{x} n\right) 
\eta(n/x)
= I_{q=1}\cdot \widehat{\eta}(-\delta) \cdot x
+ E \cdot x,\]
where $I_{q=1}=1$ if $q=1$, $I_{q=1}=0$ if $q\ne 1$,
and
\[|E|\leq \frac{4.269\cdot 10^{-14}}{q} + 
\frac{1}{\sqrt{x}} \left(\frac{380600}{\sqrt{q}} + 76\right).\]
\end{cor}

Let us now look at a different kind of modification of the Gaussian smoothing.
Say we would like a weight of a specific shape; for example, for the purposes
of \cite{HelfTern}, we would like an approximation to the function
\begin{equation}\label{eq:cleo2}
\eta_\circ:t\mapsto \begin{cases}
t^3 (2-t)^3 e^{-(t-1)^2/2} &\text{for $t\in \lbrack 0,2\rbrack$,}\\
0 &\text{otherwise.}\end{cases}\end{equation}
At the same time, what we have is an estimate for the Mellin transform of
the Gaussian $e^{-t^2/2}$, centered at $t=0$. 

The route taken here is to work with an approximation $\eta_+$ to $\eta_\circ$.
We let
\begin{equation}\label{eq:patra2}\eta_+(t) = h_H(t) \cdot 
t e^{-t^2/2},\end{equation} 
where $h_H$ is a band-limited approximation to 
\begin{equation}\label{eq:hortor}
h(t) = \begin{cases}
t^2 (2-t)^3 e^{t-1/2} &\text{if $t\in \lbrack 0,2\rbrack$,}\\ 
0 &\text{otherwise.}\end{cases}\end{equation}
By {\em band-limited} we mean that the restriction of the Mellin transform
of $h_H$ to the imaginary axis is of compact support. (We could, alternatively,
let $h_H$ be a function whose Fourier transform is of compact support; this
would be technically easier in some ways, but it would also lead to using GRH
verifications less efficiently.) 

To be precise: we define
\begin{equation}\label{eq:dirich2}\begin{aligned}
F_H(t) &= \frac{\sin(H \log y)}{\pi \log y},\\
h_H(t) &= (h \ast_M F_H)(y) = \int_0^\infty h(t y^{-1}) F_H(y)
\frac{dy}{y}
\end{aligned}\end{equation}
and $H$ is a positive constant. It is easy to check that $M F_H(i\tau) = 1$ for 
$-H<\tau<H$ and $M F_H(i \tau) =0$ for $\tau>H$ or $\tau<-H$ (unsurprisingly,
since $F_H$ is a Dirichlet kernel under a change of variables). Since,
in general, the Mellin transform of a multiplicative convolution $f\ast_M g$
equals $M f \cdot M g$, we see that the Mellin transform of $h_H$,
on the imaginary axis, equals the truncation of the Mellin transform of $h$
to $\lbrack - i H, i H\rbrack$. Thus, $h_H$ is a band-limited approximation to
$h$, as we desired.

The distinction between the odd and the even case in the statement that follows
simply reflects the two different points up to which computations where carried
out in \cite{Plattfresh}; these computations, in turn, were tailored to the
needs of \cite{HelfMaj} (as was the shape of $\eta_+$ itself).
\begin{thm}\label{thm:malpor}
Let $\eta(t) = \eta_+(t) = h_H(t) t e^{-t^2/2}$,
where $h_H$ is as in (\ref{eq:dirich2}) and $H=200$.
Let $x$ be a real number $\geq 10^{12}$. 
Let $\chi$ be a primitive character mod $q$, where
$1\leq q\leq 150000$ if $q$ is odd, and
$1\leq q\leq 300000$ if $q$ is even.

Then, for any $\delta \in \mathbb{R}$ with 
$|\delta|\leq 600000 \cdot \gcd(q,2)/q$,
\[\sum_{n=1}^\infty \Lambda(n) \chi(n) e\left(\frac{\delta}{x} n\right) 
\eta(n/x)
= I_{q=1}\cdot \widehat{\eta}(-\delta) \cdot x
+ E \cdot x,\]
where $I_{q=1}=1$ if $q=1$, $I_{q=1}=0$ if $q\ne 1$,
and
\[|E|\leq
\frac{6.18\cdot 10^{-12}}{\sqrt{q}} +
\frac{1.14\cdot 10^{-10}}{q} 
+ \frac{1}{\sqrt{x}}  \left(\frac{499100}{\sqrt{q}} + 52\right).\]

If $q=1$, we have the sharper bound
\[|E|\leq  3.34 \cdot 10^{-11} +  
\frac{251100}{\sqrt{x}}.\]
\end{thm}
This is a paradigmatic example, in that, following the proof given in 
\S \ref{subs:astardo}, we can bound exponential sums with weights of the form
$h_H(t) e^{-t^2/2}$, where $h_H$ is a band-limited approximation to just
about any continuous function of our choosing.


Lastly, we will need an explicit estimate of the $\ell_2$ norm corresponding to
the sum in Thm.~\ref{thm:malpor}, for the trivial character.
\begin{prop}\label{prop:malheur}
Let $\eta(t) = \eta_+(t) = h_H(t) t e^{-t^2/2}$,
where $h_H$ is as in (\ref{eq:dirich2}) and $H=200$.
Let $x$ be a real number $\geq 10^{12}$. 

Then
\[\begin{aligned}
\sum_{n=1}^\infty \Lambda(n) (\log n)
\eta^2(n/x)
&= x\cdot \int_0^\infty \eta_+^2(t) \log x t\; dt + E_1 \cdot x \log x\\
&= 0.640206 x \log x - 0.021095 x + E_2 \cdot x \log x,\end{aligned}\]
where
\[|E_1|\leq 1.536\cdot 10^{-15} + \frac{310.84}{\sqrt{x}}\;\;\;\;\;\; |E_2|\leq
2\cdot 10^{-6} + \frac{310.84}{\sqrt{x}}.\]
\end{prop}
\subsection{Main ideas}
We will be working with smoothed sums
\begin{equation}\label{eq:mork}
S_{\eta}(\alpha,x) = \sum_{n=1}^\infty \Lambda(n) \chi(n) e(\delta n/x)
\eta(n/x).\end{equation}
Our integral will actually be of the form
\begin{equation}\label{eq:kamon}\int_{\mathfrak{M}} S_{\eta_+}(\alpha,x)^2 S_{\eta_*}(\alpha,x) e(-N\alpha) d\alpha,\end{equation}
where $\eta_+$ and $\eta_*$ are two different smoothing functions to be 
discussed soon.

Estimating the sums (\ref{eq:mork}) on $\mathfrak{M}$ reduces to estimating
the sums 
\begin{equation}\label{eq:mindy}
S_{\eta}(\delta/x,x) = \sum_{n=1}^\infty \Lambda(n) \chi(n) e(\delta n/x)
\eta(n/x)\end{equation}
for $\chi$ varying among
all Dirichlet characters modulo $q\leq r_0$ and
for $|\delta|\leq c r_0/q$, i.e., $|\delta|$ small.
Sums such as (\ref{eq:mindy}) are estimated using Dirichlet $L$-functions
$L(s,\chi)$ (see \S \ref{subs:durian}).
An {\em explicit formula} gives an expression
\begin{equation}\label{eq:manon}
S_{\eta,\chi}(\delta/x,x) = I_{q=1} \widehat{\eta}(-\delta) x - 
\sum_{\rho} F_{\delta}(\rho) x^\rho + \text{small error},\end{equation}
where $I_{q=1}=1$ if $q=1$ and $I_{q=1}=0$ otherwise. Here $\rho$ runs
over the complex numbers $\rho$ with $L(\rho,\chi)=0$ and 
$0<\Re(\rho)<1$ (``non-trivial zeros''). The function $F_\delta$ is the
Mellin transform of $e(\delta t) \eta(t)$ (see \S \ref{subs:milly}).

The questions are then: where are the non-trivial zeros $\rho$ of $L(s,\chi)$?
How fast does $F_\delta(\rho)$ decay as $\Im(\rho)\to \pm \infty$?

Write $\sigma = \Re(s)$, $\tau = \Im(s)$.
The belief is, of course, that $\sigma=1/2$ for every non-trivial zero
(Generalized Riemann Hypothesis), but this is far from proven.
Most work to date has used zero-free regions of the form 
$\sigma \leq 1 - 1/C \log q |\tau|$, $C$ a constant. This is a classical
zero-free region, going back, qualitatively, to de la Vall\'ee-Poussin (1899).
The best values of $C$ known are due to McCurley \cite{MR726004} and Kadiri
 \cite{MR2140161}.

These regions seem too narrow to yield a proof of the three-primes theorem.
What we will use instead is a finite verification of GRH ``up to $T_q$'', 
i.e., a computation 
showing that, for every Dirichlet character of conductor $q\leq r_0$
($r_0$ a constant, as above), every non-trivial zero $\rho=\sigma+i\tau$
with $|\tau|\leq T_q$ satisfies $\Re(\sigma)=1/2$. Such
verifications go back to Riemann; modern computer-based
methods are descended in part from a paper by Turing \cite{MR0055785}.
(See the historical article \cite{MR2263990}.)
In his thesis \cite{Platt}, D. Platt gave a rigorous verification for
$r_0 = 10^5$, $T_q = 10^8/q$. In coordination with the present work, he has
extended this to
\begin{itemize}
\item all odd $q\leq 3\cdot 10^5$, with $T_q = 10^8/q$,
\item all even $q\leq 4\cdot 10^5$, with 
$T_q = \max(10^8/q,200 + 7.5\cdot 10^7/q)$.
\end{itemize}  
This was a major computational effort, involving, in particular, a fast
implementation of interval arithmetic (used for the sake of rigor). 

What remains to discuss, then, is how to choose $\eta$ in such a way
$F_\delta(\rho)$ decreases fast enough as $|\tau|$ increases, so
that (\ref{eq:manon}) gives a good estimate. We cannot hope for 
$F_\delta(\rho)$ to start decreasing consistently before $|\tau|$ is at least
as large as a multiple of $2 \pi |\delta|$. Since $\delta$ varies within
$(-c r_0/q, c r_0/q)$, this explains why $T_q$ is taken inversely proportional
to $q$ in the above. As we will work with $r_0\geq 150000$, we also see
that we have little margin for maneuver: we want $F_\delta(\rho)$ to be
extremely small already for, say, $|\tau|\geq 80 |\delta|$.
 We also have a Scylla-and-Charybdis situation, courtesy
of the uncertainty principle: roughly speaking, $F_\delta(\rho)$ cannot decrease
faster than exponentially on $|\tau|/|\delta|$ both for $|\delta|\leq 1$ and for
$\delta$ large.

The most delicate case is that of $\delta$ large, since then $|\tau|/|\delta|$
is small. It turns out we can manage to get decay that is much faster
than exponential for $\delta$ large, while no slower than exponential
 for $\delta$ small. This we will achieve
by working with smoothing functions based on the (one-sided) 
Gaussian $\eta_\heartsuit(t) = e^{-t^2/2}$.

The Mellin transform of the twisted Gaussian $e(\delta t) e^{-t^2/2}$ is
a parabolic cylinder function $U(a,z)$ with $z$ purely imaginary.
Since fully explicit estimates for $U(a,z)$, $z$ imaginary, have not been
worked in the literature, we will have to derive them ourselves. 

Once we have fully explicit estimates for the Mellin transform of the twisted
Gaussian, we are able to use essentially
any smoothing function based on the Gaussian $\eta_\heartsuit(t) = e^{-t^2/2}$. 
As we already saw, we
can and will consider smoothing functions obtained by convolving the twisted
Gaussian with another function and also functions obtained by multiplying the twisted Gaussian
with another function. All we need to do is use an explicit formula of the 
right kind -- that is, a formula that does not assume too much about the smoothing
function or the region of holomorphy of its Mellin transform, but still gives
very good error terms, with simple expressions. 

All results here will be based on a single, general
 explicit formula (Lem.~\ref{lem:agamon}) 
valid for all our purposes. The contribution of the zeros in the critical trip
can be handled in a unified way (Lemmas \ref{lem:garmola} and 
\ref{lem:hausierer}). All that has to be done for each smoothing function
is to bound a simple integral (in (\ref{eq:jotok})). We then apply a finite
verification of GRH and are done.

\subsection{Acknowledgments}
The author is very thankful to D. Platt, who, working in close coordination
with him, provided GRH verifications in the necessary ranges, and
also helped him with the usage of interval arithmetic. 
Warm thanks are due to A. C\'ordoba and J. Cilleruelo, 
for discussions on the method of stationary phase, and to V. Blomer and
N. Temme, for further help with parabolic cylinder functions.
Gratitude is also felt towards
A. Booker, B. Green, R. Heath-Brown, H. Kadiri, O. Ramar\'e, 
T. Tao and M. Watkins,
for discussions on Goldbach's problem and related issues.
Additional
 references were graciously provided by R. Bryant, S. Huntsman
and I. Rezvyakova.

Travel and other expenses were funded in part by 
the Adams Prize and the Philip Leverhulme Prize.
The author's work on the problem started at the
Universit\'e de Montr\'eal (CRM) in 2006; he is grateful to both the Universit\'e
de Montr\'eal and the \'Ecole Normale Sup\'erieure for providing pleasant
working environments.

The present work would most likely not have been possible without free and 
publicly available
software: PARI, Maxima, Gnuplot, VNODE-LP, PROFIL / BIAS, SAGE, and, of course, \LaTeX, Emacs,
the gcc compiler and GNU/Linux in general. Some exploratory work was done
in SAGE and Mathematica. Rigorous calculations used either D. Platt's
interval-arithmetic package (based in part on Crlibm) or the PROFIL/BIAS interval arithmetic package
underlying VNODE-LP.

The calculations contained in this paper used a nearly trivial amount of
resources; they were all carried out on the author's desktop computers at
home and work. However, D. Platt's computations \cite{Plattfresh}
used a significant amount of resources, kindly donated to D. Platt and the
author by several institutions. This crucial help was provided by MesoPSL
(affiliated with the Observatoire de Paris and Paris Sciences et Lettres),
Universit\'e de Paris VI/VII (UPMC - DSI - P\^{o}le Calcul), University of Warwick (thanks
to Bill Hart), University of Bristol, France Grilles (French National Grid
Infrastructure, DIRAC instance),
Universit\'e de Lyon 1 and Universit\'e de Bordeaux 1. Both D. Platt and the
author would like to thank the donating organizations, their technical
staff, and all academics who helped to make these resources available to them.

\section{Preliminaries}
\subsection{Notation}
As is usual, we write $\mu$ for the Moebius function, $\Lambda$ for
the von Mangoldt function. We let $\tau(n)$ be the number of divisors of an
integer $n$ and $\omega(n)$ the number of prime divisors.
For $p$ prime, $n$ a non-zero integer,
we define $v_p(n)$ to be the largest non-negative integer $\alpha$ such
that $p^\alpha|n$.

We write $(a,b)$ for the greatest common divisor of $a$ and $b$. If there
is any risk of confusion with the pair $(a,b)$, we write $\gcd(a,b)$.
Denote by $(a,b^\infty)$ the divisor $\prod_{p|b} p^{v_p(a)}$ of $a$.
(Thus, $a/(a,b^\infty)$ is coprime to $b$, and is in fact the maximal
divisor of $a$ with this property.)

As is customary, we write $e(x)$ for $e^{2\pi i x}$.
We write $|f|_r$ for the $L_r$ norm of a function $f$.

We write $O^*(R)$ to mean a quantity at most $R$ in absolute value. 

\subsection{Dirichlet characters and $L$ functions}\label{subs:durian}
A {\em Dirichlet character} $\chi:\mathbb{Z}\to \mathbb{C}$ of modulus $q$
is a character $\chi$ of $(\mathbb{Z}/q \mathbb{Z})^*$ lifted to 
$\mathbb{Z}$ with the convention that $\chi(n)=0$ when
$(n,q)\ne 1$. Again by convention, there is a Dirichlet character of modulus
$q=1$, namely, the {\em trivial character} $\chi_T:\mathbb{Z}\to \mathbb{C}$
defined by $\chi_T(n)=1$ for every $n\in \mathbb{Z}$. 

If $\chi$ is a character modulo $q$ and $\chi'$ is a
character modulo $q'|q$ such that $\chi(n)=\chi'(n)$ for all $n$ coprime to $q$,
we say that $\chi'$ {\em induces} $\chi$. A character is 
{\em primitive} if it is not induced by any character of smaller modulus.
Given a character $\chi$, we write $\chi^*$ for the (uniquely defined)
primitive character inducing $\chi$. If a character $\chi$ mod $q$ is induced
by the trivial character $\chi_T$, we say that $\chi$ is {\em principal}
and write $\chi_0$ for $\chi$ (provided the modulus $q$ is clear from the 
context). In other words, $\chi_0(n)=1$ when $(n,q)=1$ and $\chi_0(n)=0$ when
$(n,q)=0$.

A Dirichlet $L$-function $L(s,\chi)$ ($\chi$ a Dirichlet character) is
defined as the analytic continuation of $\sum_n \chi(n) n^{-s}$ to the
entire complex plane; there is a pole at $s=1$ if $\chi$ is principal. 

A non-trivial zero of $L(s,\chi)$ is any $s\in \mathbb{C}$
such that $L(s,\chi)=0$ and $0 < \Re(s) < 1$. (In particular, a zero
at $s=0$ is called ``trivial'', even though its contribution can be a little
tricky to work out. The same would go for the other zeros with $\Re(s)=0$
occuring for $\chi$ non-primitive, though we will avoid this issue by
working mainly with $\chi$ primitive.) The zeros that occur at (some) negative 
integers are called {\em trivial zeros}.

The {\em critical line} is the line $\Re(s)=1/2$ in the complex
plane. Thus, the generalized Riemann hypothesis for Dirichlet $L$-functions
reads: for every Dirichlet character $\chi$,
all non-trivial zeros of $L(s,\chi)$ lie on the critical line. 
Verifiable finite versions of the generalized Riemann hypothesis generally 
read: for every Dirichlet character $\chi$ of modulus $q\leq Q$,
all non-trivial zeros of $L(s,\chi)$ with $|\Im(s)|\leq f(q)$ lie
on the critical line (where $f:\mathbb{Z}\to \mathbb{R}^+$ is some given 
function).

\subsection{Mellin transforms}\label{subs:milly}

The {\em Mellin transform} of a function $\phi:(0,\infty)\to \mathbb{C}$ is
\begin{equation}\label{eq:souv}
M \phi(s) := \int_0^{\infty} \phi(x) x^{s-1} dx .\end{equation}
If $\phi(x) x^{\sigma-1}$ is in $\ell_1$ with respect to $dt$ 
(i.e., $\int_0^\infty |\phi(x)| x^{\sigma-1} dx < \infty$), then the Mellin transform
is defined on the line $\sigma + i \mathbb{R}$. Moreover, if
$\phi(x) x^{\sigma-1}$ is in $\ell_1$ for $\sigma=\sigma_1$ and for
$\sigma = \sigma_2$, where $\sigma_2>\sigma_1$, then it is easy to see that
it is also in $\ell_1$ for all $\sigma \in (\sigma_1,\sigma_2)$, 
and that, moreover,
the Mellin transform is holomorphic on $\{s: \sigma_1 < \Re(s) < \sigma_2\}$. We
then say that $\{s: \sigma_1 < \Re(s) < \sigma_2\}$ is a {\em strip of
holomorphy} for the Mellin transform.

The Mellin transform becomes a Fourier transform (of $\eta(e^{-2\pi v})
e^{-2\pi v \sigma}$) by means of the change of variables $x = e^{-2\pi v}$.
We thus obtain, for example, that 
the Mellin transform is an isometry, in the sense that
\begin{equation}\label{eq:victi}
\int_0^\infty |f(x)|^2 x^{2\sigma} \frac{dx}{x} = \frac{1}{2\pi} \int_{-\infty}^\infty
|Mf(\sigma+it)|^2 dt.\end{equation}
Recall that, in the case of the Fourier transform,
for $|\widehat{f}|_2 = |f|_2$ to hold, it is enough that $f$
be in $\ell_1 \cap \ell_2$. This gives us that, for (\ref{eq:victi}) to hold,
it is enough that $f(x) x^{\sigma-1}$ be in $\ell_1$ and $f(x) x^{\sigma-1/2}$
be in $\ell_2$ (again, with respect to $dt$, in both cases).

We write $f\ast_M g$ for the multiplicative, or Mellin, convolution of $f$
and $g$:
\begin{equation}\label{eq:solida}
(f\ast_M g)(x) = \int_0^{\infty} f(w) g\left(\frac{x}{w}\right) \frac{dw}{w}.
\end{equation}
In general, \begin{equation}\label{eq:zorbag}
M (f\ast_M g) = Mf \cdot Mg\end{equation} and 
\begin{equation}\label{eq:mouv}
M(f\cdot g)(s) = \frac{1}{2\pi i}\int_{\sigma-i\infty}^{\sigma+i\infty}
Mf(z) Mg(s-z) dz\;\;\;\;\;\;\;\;\text{\cite[\S 17.32]{MR1773820}}\end{equation}
provided that $z$ and $s-z$ are within the strips on which $Mf$ and $Mg$
(respectively) are well-defined.

We also have several useful transformation rules, just as for the Fourier
transform. For example, 
\begin{equation}\label{eq:harva}\begin{aligned}
M(f'(t))(s) &= - (s-1)\cdot Mf(s-1),\\
M(t f'(t))(s) &= - s\cdot  Mf(s),\\
M((\log t) f(t))(s) &= (Mf)'(s)\end{aligned}\end{equation}
(as in, e.g., \cite[Table 1.11]{Mellin}).

 Since (see, e.g., \cite[Table 11.3]{Mellin}
or \cite[\S 16.43]{MR1773820})
\[(M I_{\lbrack a,b\rbrack})(s) = \frac{b^s - a^s}{s},\]
we see that
\begin{equation}\label{eq:envy}
M\eta_2(s) = \left(\frac{1 - 2^{-s}}{s}\right)^2,\;\;\;\;\;
M\eta_4(s) = \left(\frac{1 - 2^{-s}}{s}\right)^4 .\end{equation}

Let $f_z = e^{-zt}$, where $\Re(z)>0$. Then
\[\begin{aligned}
(Mf)(s) &= \int_0^\infty e^{-zt} t^{s-1} dt = \frac{1}{z^s} \int_0^\infty
e^{-t} dt\\&= \frac{1}{z^s} \int_0^{z\infty} e^{-u} u^{s-1} du = 
\frac{1}{z^s} \int_0^\infty e^{-t} t^{s-1} dt = \frac{\Gamma(s)}{z^s},
\end{aligned}\]
where the next-to-last step holds by contour integration,
and the last step holds by the definition of the Gamma function $\Gamma(s)$.

\section{The Mellin transform of the twisted Gaussian}\label{sec:gauss}
Our aim in this section is 
to give fully explicit, yet relatively simple bounds for 
the Mellin transform $F_\delta(\rho)$ of $e(\delta t) \eta_\heartsuit(t)$, where
 $\eta_\heartsuit(t) = e^{-t^2/2}$ and $\delta$ is arbitrary.
The rapid decay that results will establish that the Gaussian 
$\eta_\heartsuit$ is a very good choice for a smoothing, particularly when
the smoothing has to be twisted by an additive character $e(\delta t)$.

Gaussian smoothing has been used before in number theory; see, notably,
Heath-Brown's well-known paper on the fourth power moment of the
Riemann zeta function \cite{MR532980}. 
What is new here is that we will derive fully explicit
bounds on the Mellin transform of the twisted Gaussian. This means that
the Gaussian smoothing will be a real option in explicit work on exponential 
sums in number theory and elsewhere from now on.

(There has also been work using the Gaussian after a logarithmic change
of variables; see, in particular, \cite{MR0202686}. In that case, the Mellin
transform is simply a Gaussian (as in, e.g.,
\cite[Ex. XII.2.9]{MR2378655}). However, for $\delta$ non-zero,
the Mellin transform of a twist $e(\delta t) e^{-(\log t)^2/2}$ decays
very slowly, and thus would not be in general useful.)

\begin{thm}\label{thm:princo}
Let $f_\delta(t) = e^{-t^2/2} e(\delta t)$, $\delta\in \mathbb{R}$. 
Let $F_\delta$ be the Mellin transform of $f_\delta$.
 Let $s = \sigma+i \tau$, $\sigma\geq 0$,
$\tau \ne 0$. Let $\ell = - 2 \pi \delta$. Then, if $\sgn(\delta)\ne \sgn(\tau)$,
\begin{equation}\label{eq:wilen}\begin{aligned}|F_{\delta}(s)| &\leq
C_{0,\tau,\ell} \cdot e^{- E\left(\frac{|\tau|}{(\pi \delta)^2}\right) \cdot |\tau|}
\\ &+ C_{1,\tau}\cdot
e^{-0.4798 |\tau|}
+ C_{2,\tau,\ell}\cdot
 e^{-  \min\left(\frac{1}{8} \left(\frac{\tau}{\pi \delta}\right)^2, \frac{25}{32}
|\tau|\right)},
\end{aligned}\end{equation}
where 
\begin{equation}\label{eq:cormo}\begin{aligned}
E(\rho) &= \frac{1}{2} 
\left(\arccos \frac{1}{\upsilon(\rho)} -
\frac{2 (\upsilon(\rho)- 1)}{\rho}
\right),\\
C_{0,\tau,\ell} &= \min\left(2, \frac{\sqrt{|\tau|}}{\frac{2 \pi |\delta|}{3.3}}\right)
\left(1 + \max(7.83^{1-\sigma},1.63^{\sigma-1})\right) \left(\frac{3/2}{\min\left(
\frac{|\tau|}{2\pi |\delta|},\sqrt{|\tau|}\right)}
\right)^{1-\sigma},\\
C_{1,\tau} &= \left(1 + \frac{1+\sqrt{2}}{|\tau|}\right) 
 |\tau|^{\frac{\sigma}{2}}
,\;\;\;\;\;\;
\upsilon(\rho) = \sqrt{\frac{1 + \sqrt{1+\rho^2}}{2}},
\\
C_{2,\tau,\ell} &= 
P_\sigma\left(\min\left(\frac{|\tau|}{2 \pi |\delta|}, \frac{5}{4}
\sqrt{|\tau|}\right)\right),
\end{aligned}
\end{equation}
where $P_\sigma(x) = x^{\sigma-2}$ if $\sigma\in \lbrack 0,2\rbrack$,
$P_\sigma(x) = x^{\sigma-2} + (\sigma-2) x^{\sigma-4}$ if $\sigma\in (
2,4\rbrack$ and $P_\sigma(x) = x^{\sigma-2} + (\sigma-2) x^{\sigma-4} + 
\dotsc + (\sigma-2k) x^{\sigma-2(k+1)}$ if $\sigma \in (2k,2(k+1)\rbrack$.

If $\sgn(\delta)=\sgn(\tau)$ (or $\delta=0$) and $|\tau|\geq 2$,
\begin{equation}\label{eq:pieses}
|F_{\delta}(s)| \leq C'_{\tau,\ell} e^{-\frac{\pi}{4} |\tau|}
,\end{equation}
where \[C'_{\tau,\ell}\leq
\frac{e^{\pi/2} |\tau|^{\sigma/2}}{2}
\cdot \begin{cases}
1 + \frac{2 \pi^{3/2} |\delta|}{\sqrt{|\tau|}}
&\text{for $\sigma\in \lbrack 0,1\rbrack$,}\\
\Gamma(\sigma/2) &\text{for $\sigma>0$ arbitrary.}
\end{cases}\]
\end{thm}

As we shall see, the choice of $\eta(t) = e^{-t^2/2}$ can be easily motivated
by the method of stationary phase, but the problem is actually solved by
the saddle-point method. One of the challenges here is to keep all expressions
explicit and practical. This turns out to require the use of validated numerics
(Appendix \ref{subs:exbi}); in particular, the bisection method (implemented
using interval arithmetic) gets combined with a closer examination at infinity
and near extrema.
 
The expressions in Thm.~\ref{thm:princo} can be easily simplified further 
in applications with some mild constraints, especially if one is ready to 
make some sacrifices in the main term. 

\begin{cor}\label{cor:amanita1}
Let $f_\delta(t) = e^{-t^2/2} e(\delta t)$, $\delta\in \mathbb{R}$. Let
$F_\delta$ be the Mellin transform of $f_\delta$. Let 
$s = \sigma+i\tau$, where $\sigma\in \lbrack 0,1\rbrack$ and
$|\tau|\geq \max(100,4\pi^2 |\delta|)$. Then, for $0\leq k\leq 2$,
\begin{equation}
|F_\delta(s+k)| + |F_\delta(k+1-s)|\leq 
c_k \cdot
\begin{cases}
\left(\frac{|\tau|}{2\pi |\delta|}\right)^k
e^{-0.1065 \left(\frac{\tau}{\pi \delta}\right)^2} &
\text{if $|\tau| < \frac{3}{2} (\pi \delta)^2$,}\\ 
|\tau|^{k/2}
e^{-0.1598 |\tau|} & \text{if $|\tau| \geq \frac{3}{2} (\pi \delta)^2$,}
\end{cases}
\end{equation}
where $c_0 = 4.226$, $c_1 = 3.516$, $c_2 = 3.262$.
\end{cor}
It is natural to look at 
$|F_\delta(s+k)| + |F_\delta(k+1-s)|$ with $s$ in the critical strip
$\Re(s) \in \lbrack 0,1\rbrack$, 
since such expressions are key to the study of exponential sums with a
smoothing function equal to or based on $t^k e^{-t^2/2}$.



Let us end by a remark that may be relevant to applications outside number
theory. By (\ref{eq:mosot}), 
Thm.~\ref{thm:princo} gives us bounds on the parabolic cylinder
function $U(a,z)$ for $z$ purely imaginary and $|\Re(a)|\leq 1/2$.
The bounds are useful when $|\Im(a)|$ is at least somewhat larger than
$|\Im(z)|$ (i.e., when $|\tau|$ is large compared to $\ell$). 
While the Thm.~\ref{thm:princo} is stated for $\sigma \geq 0$ (i.e., for $\Re(a)\geq 0$),
 extending the result to larger half-planes for $a$ is not
too hard -- integration by parts can be used to push $a$ to the right.

As we shall see in \S \ref{subs:melltwist}, the literature on parabolic cylinder
functions is rich and varied. However, it stopped short of giving fully explicit
expressions for $a$ general and $z$ imaginary. That precluded, for 
instance, the use of the Gaussian in explicit work on exponential sums in
number theory. Such work will now be possible.

\subsection{How to choose a smoothing function?}
The method of {\em stationary phase} (\cite[\S 4.11]{MR0435697}, 
\cite[\S II.3]{MR1851050}))
 suggests that the main contribution to the integral
\begin{equation}\label{eq:madejap}
F_\delta(t) = \int_0^\infty e(\delta t) \eta(t) t^s \frac{dt}{t}
\end{equation}
should come when the phase has derivative $0$. The
phase part of (\ref{eq:madejap}) is
\[e(\delta t) = t^{\Im(s)} = e^{2\pi i \delta t + \tau \log t}\]
(where we write $s=\sigma+i\tau$); clearly,
\[(2\pi \delta t + \tau \log t)' = 2\pi \delta + \frac{\tau}{t} = 0\]
when $t = -\tau/2\pi \delta$. This is meaningful when $t\geq 0$, i.e., 
$\sgn(\tau) \ne \sgn(\delta)$. The contribution of
$t = -\tau/2\pi \delta$ to (\ref{eq:madejap}) is then
\begin{equation}\label{eq:rateda}
\eta(t) e(\delta t) t^{s-1} = \eta\left(\frac{-\tau}{2\pi \delta}\right)
e^{-i \tau} \left(\frac{-\tau}{2\pi \delta}\right)^{\sigma+i\tau-1}
\end{equation}
multiplied by a ``width'' approximately equal to a constant divided by
\[\sqrt{|(2\pi i \delta t + \tau \log t)''|} = \sqrt{|-\tau/t^2|} = \frac{2\pi |\delta|}{\sqrt{|\tau|}}.\]
The absolute value of (\ref{eq:rateda}) is
\begin{equation}\label{eq:maloko}
\eta\left(-\frac{\tau}{2\pi \delta}\right) \cdot 
\left|\frac{- \tau}{2\pi \delta}\right|^{\sigma-1}.
\end{equation}

In other words, if $\sgn(\tau)\ne \sgn(\delta)$ and $\delta$ is not too small,
asking that $F_\delta(\sigma + i \tau)$ decay rapidly as $|\tau|\to \infty$
amounts to asking that $\eta(t)$ decay rapidly as $t\to 0$. Thus, if 
we ask for $F_\delta(\sigma + i \tau)$ to decay rapidly as $|\tau|\to \infty$
for all moderate $\delta$, we are requesting that
\begin{enumerate}
\item $\eta(t)$ decay rapidly as $t\to \infty$,
\item\label{it:reko} the Mellin transform $F_0(\sigma+i \tau)$ decay rapidly as $\tau
\to \pm \infty$.
\end{enumerate}
Requirement (\ref{it:reko}) is there because we also need to consider
$F_\delta(\sigma+it)$ for $\delta$ very small, and, in particular, for
$\delta=0$.

There is clearly an uncertainty-principle issue here; one cannot do arbitrarily
well in both aspects at the same time. Once we are conscious of this, the
choice $\eta(t)=e^{-t}$ in Hardy-Littlewood actually looks fairly good:
obviously, $\eta(t) = e^{-t}$ decays exponentially, and its Mellin transform
$\Gamma(s+i\tau)$ also decays exponentially as $\tau\to \pm \infty$.
Moreover, for this choice of $\eta$, the Mellin transform $F_\delta(s)$ can
be written explicitly:
$F_\delta(s) = \Gamma(s)/(1-2\pi i\delta)^s$.

It is not hard to work out an explicit formula\footnote{There may be a minor gap in the
  literature in this respect. The explicit formula given
in \cite[Lemma 4]{MR1555183} does not make all constants explicit.
The constants and trivial-zero terms were fully worked out for $q=1$ 
by \cite{Wigert}  (cited in \cite[Exercise 12.1.1.8(c)]{MR2378655}; 
the sign of $\hyp_{\kappa,q}(z)$ there seems to be off).
As was pointed out by Landau (see \cite[p. 628]{MR0201267}),
\cite[Lemma 4]{MR1555183} actually has mistaken terms for $\chi$
non-primitive. (The author thanks R. C. Vaughan for this information
and the references.)} for $\eta(t) = e^{-t}$.
However, it is not hard to see that, for
$F_\delta(s)$ as above, $F_\delta(1/2+it)$ decays like $e^{-t/2 \pi
  |\delta|}$,
just as we expected from (\ref{eq:maloko}).
This is a little too slow for our purposes: we will often have to work
with relatively large $\delta$, and we would like to have to check
the zeroes of $L$ functions only up to relatively low heights $t$.
We will settle for a different choice of $\eta$: the Gaussian. 

The decay of
the Gaussian smoothing function $\eta(t) = e^{-t^2/2}$ is much faster than
exponential. Its Mellin transform is $\Gamma(s/2)$, which
decays exponentially as $\Im(s) \to \pm \infty$. Moreover,
the Mellin transform
$F_\delta(s)$ ($\delta \ne 0$), while not an elementary or 
very commonly
occurring function, equals (after a change of variables) a relatively 
well-studied special function, namely, a parabolic cylinder function
$U(a,z)$ (or, in Whittaker's \cite{Whi} notation, $D_{-a-1/2}(z)$).


For $\delta$ not too small, the main term will indeed work
out to be proportional to $e^{-(\tau/2\pi \delta)^2/2}$, as the method of stationary
phase indicated. This is, of course, much better than $e^{-\tau/2\pi |\delta|}$.
The ``cost'' is that the Mellin transform $\Gamma(s/2)$ for $\delta=0$
now decays like $e^{- (\pi/4) |\tau|}$ rather than $e^{- (\pi/2) |\tau|}$. This we can
certainly afford.


\subsection{The Mellin transform of the twisted Gaussian}\label{subs:melltwist}

We wish to approximate the Mellin transform
\[F_{\delta}(s) = \int_0^\infty e^{-t^2/2} e(\delta t) t^s
\frac{dt}{t},\]
where $\delta\in \mathbb{R}$. 
The parabolic cylinder function $U:\mathbb{C}^2\to \mathbb{C}$ is given by
\[U(a,z) = \frac{e^{-z^2/4}}{\Gamma\left(\frac{1}{2} + a\right)}
\int_0^\infty t^{a - \frac{1}{2}} e^{-\frac{1}{2} t^2 - z t} dt\]
for $\Re(a)>-1/2$; this can be extended to all $a,z\in \mathbb{C}$
either by analytic continuation or by other integral representations
(\cite[\S 19.5]{MR0167642}, \cite[\S 12.5(i)]{MR2655352}). Hence
\begin{equation}\label{eq:mosot}F_\delta(s) = e^{\left(\pi i \delta
\right)^2} \Gamma(s) 
U\left(s-\frac{1}{2},- 2\pi i \delta
\right).\end{equation}
The second argument of $U$ is purely imaginary; it would be otherwise if 
a Gaussian of non-zero mean were chosen.

Let us briefly discuss the state of knowledge up to date 
on Mellin transforms of ``twisted''
Gaussian smoothings, that is, $e^{-t^2/2}$ multiplied 
by an additive character $e(\delta t)$. 
As we have just seen, these Mellin transforms are precisely the parabolic
cylinder functions $U(a,z)$.

The function $U(a,z)$ has been well-studied for $a$ and $z$ real; see, e.g., 
\cite{MR2655352}. Less attention has been paid to the more general case of
$a$ and $z$ complex. The most notable exception is by far the work of
Olver \cite{MR0094496}, 
\cite{MR0109898}, \cite{MR0131580}, \cite{MR0185350};
 he gave asymptotic series 
for $U(a,z)$, $a,z\in \mathbb{C}$. These were asymptotic series in the 
sense of Poincar\'e, and thus not in general convergent; they would solve
our problem if and only if they came with error term bounds. Unfortunately, 
it would seem that all fully explicit
error terms in the literature are either for $a$ and $z$
real, or for $a$ and $z$ outside our range of interest (see both Olver's work
and \cite{MR1993339}.) The bounds in \cite{MR0131580} involve non-explicit 
constants.
 Thus, we will have to find expressions with explicit error bounds
ourselves. Our case is that of
$a$ in the critical strip, $z$ purely imaginary.


\subsection{General approach and situation} We will use the
{\em saddle-point method} (see, e.g., \cite[\S 5]{MR671583}, 
\cite[\S 4.7]{MR0435697}, \cite[\S II.4]{MR1851050}) 
to obtain bounds with an optimal
leading-order term and small error terms. (We used the stationary-phase
method solely as an exploratory tool.)

What do we expect to obtain? Both
the asymptotic expressions in \cite{MR0109898} and the bounds in
\cite{MR0131580} make clear that, if the sign of $\tau = \Im(s)$ is
different from that of $\delta$,
there will a change in behavior 
when $\tau$ gets to be of size about $(2 \pi \delta)^2$. 
This is unsurprising,
given our discussion using stationary phase: for $|\Im(a)|$ smaller
than a constant times
$|\Im(z)|^2$, the term proportional to $e^{-(\pi/4) |\tau|} = e^{-|\Im(a)|/2}$
should be dominant, whereas for $|\Im(a)|$ much larger than a constant
times $|\Im(z)|^2$, the term proportional to $e^{- \frac{1}{2}
\left(\frac{\tau}{2\pi \delta}\right)^2}$ should be dominant.


\subsection{Setup}

We write 
\begin{equation}\label{eq:fribar}
\phi(u) = \frac{u^2}{2} - (2 \pi i \delta) u - i \tau \log u 
\end{equation}
for $u$ real or complex, so that
\[F_{\delta}(s) = \int_0^\infty e^{-\phi(u)} u^{\sigma} \frac{du}{u}.\]
We will be able to shift the contour of integration
as we wish, provided that it starts at $0$
and ends at a point at infinity while keeping within the sector 
$\arg(u)\in (-\pi/4,\pi/4)$.

We wish to find a saddle point. At a saddle point, $\phi'(u) = 0$.
This means that 
\begin{equation}\label{eq:arbre}
u - 2 \pi i \delta - \frac{i \tau}{u} = 0,\;\;\;\;\;\;
\text{i.e.,}\;\;\;\;\;\;
u^2 + i \ell u - i \tau = 0,\end{equation}
where $\ell = - 2\pi \delta$. The solutions to
$\phi'(u)=0$ are thus
\begin{equation}\label{eq:rooto}
u_0 = \frac{- i \ell \pm \sqrt{-\ell^2 + 4 i \tau}}{2} .\end{equation}
The second derivative at $u_0$ is
\begin{equation}\label{eq:polan}
\phi''(u_0) = \frac{1}{u_0^2} \left(u_0^2 + i \tau\right) =
\frac{1}{u_0^2} (-i \ell u_0 + 2 i \tau).\end{equation}

Assign the names $u_{0,+}$, $u_{0,-}$ to the roots in (\ref{eq:rooto})
according to the sign in front of
the square-root (where the square-root is defined so as to have argument in 
$(-\pi/2,\pi/2\rbrack$).

We assume without loss of generality that $\tau\geq 0$.
We shall also assume at first that $\ell\geq 0$ (i.e., $\delta\leq 0$), as the
case $\ell<0$ is much easier. 

\subsection{The saddle point}

Let us start by estimating
\begin{equation}\label{eq:petpan}\left|u_{0,+}^s e^{y/2}\right| = 
|u_{0,+}|^\sigma e^{- \arg(u_{0,+}) \tau} e^{y/2} 
,\end{equation}
where $y = \Re(- \ell i u_0)$. (This is the main part of the
contribution of the saddle point, without the factor that depends on the
contour.) We have
\begin{equation}\label{eq:russ}
y = \Re\left(- \frac{i \ell}{2} \left( -i \ell + \sqrt{-\ell^2 + 4 i
\tau}\right)\right) = \Re\left( - \frac{\ell^2}{2}
 - \frac{i \ell^2}{2}
\sqrt{-1 + \frac{4 i \tau}{\ell^2}}\right).\end{equation}
Solving a quadratic equation, we get that
\begin{equation}\label{eq:masha}
\sqrt{-1 + \frac{4 i \tau}{\ell^2}} = 
\sqrt{\frac{j(\rho)-1}{2}} + i \sqrt{\frac{j(\rho)+1}{2}},\end{equation}
where $j(\rho) = (1+\rho^2)^{1/2}$ and $\rho = 4 \tau/\ell^2$. Thus
\[y =  \frac{\ell^2}{2} \left(\sqrt{\frac{j(\rho)+1}{2}}
- 1\right).\]

Let us now compute the argument of $u_{0,+}$:
\begin{equation}\label{eq:phal}\begin{aligned}
\arg(u_{0,+}) &= \arg\left(- i \ell + \sqrt{-\ell^2 + 4 i \tau}\right) = 
\arg\left(- i + \sqrt{- 1 + i \rho}\right) \\ &= 
\arg\left(-i + \sqrt{\frac{-1 + j(\rho)}{2}} + i 
\sqrt{\frac{1 + j(\rho)}{2}}\right)\\
&= \arcsin\left(\frac{
\sqrt{\frac{1 + j(\rho)}{2}}-1}{
\sqrt{2 \sqrt{\frac{1+ j(\rho)}{2}} 
\left(\sqrt{\frac{1+ j(\rho)}{2}} -1\right)}}\right)\\
&=  
\arcsin\left(\sqrt{\frac{1}{2} \left(1 - \sqrt{\frac{2}{1+j(\rho)}}\right)}\right)
= \frac{1}{2} \arccos \sqrt{\frac{2}{1+ j(\rho)}}
\end{aligned}\end{equation}
(by $\cos 2\theta = 1 - 2\sin^2 \theta$). Thus
\begin{equation}\label{eq:caglio}\begin{aligned}
- \arg(u_{0,+}) \tau + \frac{y}{2} &= 
- \left(\arccos \sqrt{\frac{2}{1+j(\rho)}} - 
\frac{\ell^2}{2 \tau} \left(\sqrt{\frac{j(\rho)+1}{2}} - 1\right)
\right) \frac{\tau}{2}
\\
&= - \left(\arccos \frac{1}{\upsilon(\rho)} - \frac{2 (\upsilon(\rho)-1)}{\rho}
\right) \frac{\tau}{2} 
,\end{aligned}\end{equation}
where $\upsilon(\rho) = \sqrt{(1+j(\rho))/2}$.

It is clear that
\begin{equation}\label{eq:kla1}\lim_{\rho\to \infty}
\left(\arccos \frac{1}{\upsilon(\rho)} - \frac{2 (\upsilon(\rho)-1)}{\rho}\right)
= \frac{\pi}{2}\end{equation}
whereas
\begin{equation}\label{eq:kla2}
\arccos \frac{1}{\upsilon(\rho)} - \frac{2 (\upsilon(\rho)-1)}{\rho}
\sim \frac{\rho}{2} - \frac{\rho}{4} = \frac{\rho}{4} \end{equation}
as $\rho\to 0^+$.

We are still missing a factor of $|u_{0,+}|^\sigma$ (from 
(\ref{eq:petpan})), a factor of $|u_{0,+}|^{-1}$ (from the invariant
differential $du/u$) and a factor of $1/\sqrt{|\phi''(u_{0,+})|}$
(from the passage by the saddle-point along a path of steepest descent).
By (\ref{eq:polan}), this is
\[\frac{|u_{0,+}|^{\sigma-1}}{\sqrt{|\phi''(u_{0,+})|}} = 
\frac{|u_{0,+}|^{\sigma-1}}{\frac{1}{|u_{0,+}|} 
\sqrt{\left|- i \ell u_{0,+} + 2 i \tau\right|}} = 
\frac{|u_{0,+}|^\sigma}{\sqrt{\left|- i \ell  u_{0,+} + 2 i \tau\right|}} .
\]

By (\ref{eq:rooto}) and (\ref{eq:masha}),
\begin{equation}\label{eq:dada}\begin{aligned}
|u_{0,+}| &= \left|\frac{- i \ell + \sqrt{-\ell^2 + 4 i \tau}}{2}\right|
= \frac{\ell}{2}\cdot \left|\sqrt{\frac{-1+j(\rho)}{2}} +
\left(\sqrt{\frac{1+j(\rho)}{2}} -1\right) i\right|\\
&= \frac{\ell}{2} \sqrt{\frac{-1+j(\rho)}{2} + \frac{1+j(\rho)}{2} + 1
- 2 \sqrt{\frac{1+j(\rho)}{2}}} \\ &= \frac{\ell}{2} \sqrt{1+j(\rho)-2\sqrt{\frac{1+j(\rho)}{2}
}} =  \frac{\ell}{\sqrt{2}}\sqrt{\upsilon(\rho)^2- \upsilon(\rho)}.
\end{aligned}\end{equation}
Proceeding as in (\ref{eq:russ}), we obtain that
\begin{equation}\label{eq:muxomor}\begin{aligned}
\left|- i \ell u_{0,+} + 2 i \tau\right| &=
\left|- \frac{i \ell}{2} \left( -i \ell + \ell \sqrt{-1 + \frac{4 i
\tau}{\ell^2}}\right) + 2 i \tau\right|\\
&= \left| -\frac{\ell^2}{2} + 2 i \tau + 
\frac{\ell^2}{2} \sqrt{\frac{j(\rho)+1}{2}} - \frac{i \ell^2}{2}
\sqrt{\frac{j(\rho)-1}{2}}
\right|\\ &= \frac{\ell^2}{2}
\sqrt{\left(-1+ \sqrt{\frac{j(\rho)+1}{2}}\right)^2
 + \left(\rho - \sqrt{\frac{j(\rho)-1}{2}}\right)^2}\\
&= \frac{\ell^2}{2} \sqrt{
j(\rho) + \rho^2 + 1 - 2 \sqrt{\frac{j(\rho)+1}{2}} - 2 \rho
\sqrt{\frac{j(\rho)-1}{2}}} .
\end{aligned}\end{equation}
Since $\sqrt{j(\rho)-1} = \rho/\sqrt{j(\rho)+1}$,
this means that
\begin{equation}\label{eq:tintin}\begin{aligned}
\left|- i \ell u_{0,+} + 2 i \tau\right| &=
\frac{\ell^2}{2} \sqrt{j(\rho) + \rho^2 + 1 - \sqrt{\frac{2}{j(\rho)+1}} 
(j(\rho) + 1 + \rho^2)}\\
&=
\frac{\ell^2}{2} \sqrt{j(\rho) + j(\rho)^2 - (\upsilon(\rho))^{-1} (j(\rho) +
j(\rho)^2)}\\
&= \frac{\ell^2}{2} \sqrt{2 \upsilon(\rho)^2 j(\rho) (1 - (\upsilon(\rho))^{-1})} 
=\frac{\ell^2 \sqrt{j(\rho)}}{\sqrt{2}} \sqrt{\upsilon(\rho)^2-\upsilon(\rho)}
.
\end{aligned}\end{equation}

Hence \[\begin{aligned}
\frac{|u_{0,+}|^\sigma}{\sqrt{\left|- i \ell u_{0,+} + 2 i \tau\right|}} &= 
\frac{\left( \frac{\ell}{\sqrt{2}}\sqrt{\upsilon(\rho)^2- \upsilon(\rho)}
\right)^\sigma}{
\frac{\ell (j(\rho))^{1/4}}{2^{1/4}} (\upsilon(\rho)^2-\upsilon(\rho))^{1/4}
} = \frac{\ell^{\sigma-1}}{2^{\frac{\sigma}{2} - \frac{1}{4}} j(\rho)^{\frac{1}{4}}}
(\upsilon(\rho)^2 - \upsilon(\rho))^{\frac{\sigma}{2} - \frac{1}{4}}\\
&= \frac{2^{\frac{\sigma}{2} - \frac{3}{4}} 
}{
\rho^{\frac{\sigma-1}{2}} j(\rho)^{\frac{1}{4}}}
(\upsilon(\rho)^2 - \upsilon(\rho))^{\frac{\sigma}{2} - \frac{1}{4}}
\cdot \tau^{\frac{\sigma-1}{2}}.\end{aligned}\]

It remains to determine the direction of steepest descent at the
saddle-point $u_{0,+}$. Let $v\in \mathbb{C}$
point in that direction.
Then, by definition, $v^2 \phi''(u_{0,+})$ is real and positive, where
$\phi$ is as in \ref{eq:fribar}. Thus $\arg(v) = -\arg(\phi''(u_{0,+}))/2$.
By (\ref{eq:polan}),
\[\arg(\phi''(u_{0,+})) = \arg(- i \ell u_{0,+} + 2 i \tau) - 2 \arg(u_{0,+}).\]
Starting as in (\ref{eq:muxomor}), we obtain that
\[
\arg(- i \ell u_{0,+} + 2 i \tau) = \arctan\left(\frac{\rho - \sqrt{\frac{j-1}{2}}}{
-1 + \sqrt{\frac{j+1}{2}}}\right),\]
and
\begin{equation}\label{eq:jomo}\begin{aligned}
\frac{\rho - \sqrt{\frac{j-1}{2}}}{
-1 + \sqrt{\frac{j+1}{2}}} &=
\frac{\left(\rho - \sqrt{\frac{j-1}{2}}\right) \left(1 + 
\sqrt{\frac{j+1}{2}}\right)
}{-1 + \frac{j+1}{2}}
= \frac{\rho - \sqrt{2 (j-1)} +
\rho \sqrt{2 (j+1)}}{j-1}\\
&= \frac{\rho + \sqrt{\frac{2}{j+1}} \left(-\sqrt{j^2-1} + \rho
    \cdot (j+1)\right)}{j-1} = \frac{\rho + \frac{1}{\upsilon} (-\rho + \rho
  \cdot (j+1))}{j-1}\\
&= \frac{\rho (1+j/\upsilon)}{j-1} = \frac{(j+1) (1+j/\upsilon)}{\rho}
= \frac{2 \upsilon (\upsilon+j)}{\rho}.\end{aligned}\end{equation}
Hence, by (\ref{eq:phal}),
\[\arg(\phi''(u_{0,+})) 
= \arctan \frac{2 \upsilon (\upsilon+j)}{\rho}
- \arccos \upsilon(\rho)^{-1}
.\]
Therefore, the direction of steepest descent is
\begin{equation}\label{eq:heraus}\begin{aligned}
\arg(v) &= -\frac{\arg(\phi''(u_{0,+}))}{2} =
 \arg(u_{0,+}) - \frac{1}{2} \arctan \frac{2 v (v+j)}{\rho} \\
&= \arg(u_{0,+}) - \arctan \Upsilon,\end{aligned}\end{equation}
where
\begin{equation}\label{eq:horem}
\Upsilon = \tan \frac{1}{2} \arctan \frac{2 v (v+j)}{\rho}.\end{equation}
Since
\[\tan \frac{\alpha}{2} = \frac{1}{\sin \alpha} - \frac{1}{\tan \alpha} = 
\sqrt{1 + \frac{1}{\tan^2 \alpha}} - \frac{1}{\tan \alpha},\]
we see that
\begin{equation}\label{eq:brahms}\Upsilon = 
\left(\sqrt{1 + \frac{\rho^2}{4 \upsilon^2 (\upsilon+j)^2}} -
\frac{\rho}{2 \upsilon (\upsilon+j)}\right).\end{equation}
Recall as well that
\[
\cos \frac{\alpha}{2} = \sqrt{\frac{1 + \cos \alpha}{2}},\;\;\;\;\;\;\;\;
\sin \frac{\alpha}{2} = \sqrt{\frac{1 - \cos \alpha}{2}}.\]
Hence, if we let 
\begin{equation}\label{eq:tover}
\theta_0 = \arg(u_{0,+}) = \frac{1}{2} \arccos \frac{1}{\upsilon(\rho)},\end{equation}
we get that
\begin{equation}\label{eq:lokrat}\begin{aligned}
\cos \theta_0 &= \cos\left(
 \frac{1}{2} \arccos \frac{1}{\upsilon(\rho)} \right) = 
\sqrt{\frac{1}{2} + \frac{1}{2\upsilon(\rho)}},\\
\sin \theta_0 &= \sin\left(
 \frac{1}{2} \arccos \frac{1}{\upsilon(\rho)} \right) 
=\sqrt{\frac{1}{2} - \frac{1}{2\upsilon(\rho)}} .\end{aligned}\end{equation}
We will prove now the useful inequality
\begin{equation}\label{eq:alberg}
\arctan \Upsilon > \theta_0,
\end{equation}
i.e.,  $\arg(v)<0$.
By (\ref{eq:heraus}), (\ref{eq:horem}) and (\ref{eq:tover}),
this is equivalent to $\arccos(1/\upsilon)\leq
\arctan 2\upsilon (\upsilon+j)/\rho$. Since 
$\tan \alpha = \sqrt{1/\cos^2 \alpha
- 1}$, we know that
$\arccos(1/\upsilon) = \arctan \sqrt{\upsilon^2-1}$; thus, in order
to prove (\ref{eq:alberg}), it is enough to check that 
\[\sqrt{\upsilon^2 - 1} \leq \frac{2 \upsilon (\upsilon+j)}{\rho}.\] This is easy,
since $j>\rho$ and $\sqrt{\upsilon^2-1} < \upsilon < 2 \upsilon$.

\subsection{The contour}
We must now choose the contour of integration. First, let us
discuss our options. By (\ref{eq:amigusur}), $\Upsilon\geq 0.79837$; 
moreover, it is
easy to show that $\Upsilon$ tends to $1$ when either
$\rho\to 0^+$ or
$\rho\to \infty$. This means that neither the simplest contour (a straight
ray from the origin within the first quadrant) nor what is arguably the
second simplest contour (leaving the origin on a ray within the first quadrant,
then sliding down a circle centered at the origin, passing past the saddle
point until you reach the $x$-axis, which you then follow to infinity) are
acceptable: either contour passes through the saddle point on a direction
close to $45$ degrees ($=\arctan(1)$) off from the direction of steepest
descent. (The saddle-point method allows in principle for any direction
less than $45$ degrees
off from the direction of steepest descent, but the bounds can degrade rapidly
-- by more than a constant factor -- 
when $45$ degrees are approached.)

It is thus best to use a curve that goes through the saddle point $u_{+,0}$
in the direction of steepest descent. We thus should use an element of a 
two-parameter family of curves.
The curve should also have a simple description in terms of polar coordinates.


We decide that our contour $C$ will be a {\em lima\c{c}on of Pascal}. 
(Excentric circles would have been another reasonable choice.)
Let $C$ be parameterized by
\begin{equation}\label{eq:hapday}
y = \left(- \frac{c_1}{\ell} r + c_0\right) r,\;\;\;\;\; x = \sqrt{r^2 - y^2}
\end{equation}
for $r\in \lbrack (c_0-1)\ell /c_1,c_0 \ell/c_1\rbrack$, where $c_0$
and $c_1$ are parameters to be set later.
 The curve goes from $(0,(c_0-1) \ell /c_1)$ to $(c_0 \ell
/c_1,0)$, and stays within the first quadrant.\footnote{Because $c_0\geq
  1$,
by (\ref{eq:pekka}).} 
In order for the curve to go through the point $u_{0,+}$, we must have
\begin{equation}\label{eq:sadday}
- \frac{c_1 r_0}{\ell} + c_0 = \sin \theta_0,\end{equation}
where
\begin{equation}\label{eq:mordo}\begin{aligned}
r_0 &= |u_{0,+}| = \frac{\ell}{\sqrt{2}} \sqrt{\upsilon(\rho)^2 - 
\upsilon(\rho)},
\end{aligned}\end{equation}
and $\theta_0$ and $\sin \theta_0$ are as in (\ref{eq:tover}) and (\ref{eq:lokrat}).
We must also check that the curve $C$ goes through $u_{0,+}$ in the
direction of steepest descent. The argument of the point $(x,y)$ is
\[\theta = \arcsin \frac{y}{r} =  \arcsin\left(-\frac{c_1 r}{\ell}
 + c_0 \right).\]
Hence
\[r \frac{d\theta}{dr} = r \frac{d \arcsin\left(
-\frac{c_1 r}{\ell} + c_0\right)
}{d r}
= r\cdot \frac{-c_1/\ell}{\cos \arcsin\left(
-\frac{c_1 r}{\ell} + c_0\right)} = 
\frac{-c_1 r}{\ell \cos \theta}.
\]
This means that, if $v$ is tangent to $C$ at the point $u_{0,+}$,
\[\tan(\arg(v) - \arg(u_{0,+})) = r \frac{d\theta}{dr} 
= \frac{-c_1 r_0}{\ell \cos \theta_0},\]
and so, by (\ref{eq:heraus}),
\begin{equation}\label{eq:andso}c_1 = \frac{\ell \cos \theta_0}{r_0} 
\Upsilon,\end{equation}
where $\Upsilon$ is as in (\ref{eq:horem}).
In consequence,
\[c_0 = \frac{c_1 r_0}{\ell} + \sin \theta_0 = 
(\cos \theta_0) \cdot  \Upsilon + \sin \theta_0,\]
and so, by (\ref{eq:lokrat}),
\begin{equation}\label{eq:hoxha}
c_1 =  \sqrt{\frac{1+1/\upsilon}{\upsilon^2-\upsilon}} \cdot \Upsilon,\;\;\;\;\;\;\;
c_0 = \sqrt{\frac{1}{2} + \frac{1}{2 \upsilon}} \cdot \Upsilon + 
\sqrt{\frac{1}{2} - \frac{1}{2 \upsilon}}
.\end{equation}

Incidentally, we have also shown that the arc-length infinitesimal is
\begin{equation}\label{eq:sprel}
|du| = \sqrt{1 + \left(r \frac{d\theta}{dr}\right)^2} 
dr = \sqrt{1 + \frac{(c_1 r/\ell)^2}{\cos^2 \theta}} dr = 
\sqrt{1 + \frac{r^2}{\frac{\ell^2}{c_1^2} - \left(\frac{c_0}{c_1}\ell 
- r\right)^2}} dr.\end{equation}
%


The contour will be as follows: first we go out of the origin
along a straight radial segment $C_1$;
 then we meet the curve $C$, and we follow it clockwise 
for a segment $C_2$, with the saddle-point roughly at its midpoint;
then we follow another radial ray $C_3$ up to infinity. For $\rho$ small,
$C_3$ will just be the $x$-axis. Both $C_1$ and $C_3$ will be contained
within the first quadrant; we will specify them later.

\subsection{The integral over the main contour segment $C_2$}



We recall that
\begin{equation}\label{eq:koromo}
\phi(u) = \frac{u^2}{2} + \ell i u - i \tau \log u.\end{equation}
Our aim is now to bound the integral
\[\int_{C_2} e^{-\Re(\phi(u))} u^{\sigma-1} du\]
over the main contour segment $C_2$. We will proceed as follows. First,
we will parameterize the integral using a real variable $\nu$,
with the value $\nu=0$ corresponding to the saddle point $u=u_{0,+}$.
We will bound $\Re(\phi(u))$ from below by
an expression of the form $\Re(\phi(u_{0,+})) + \eta \nu^2$. We then bound
$|u|^{\sigma-1} |du/d\nu|$ from above by a constant. 
This procedure will give a bound that is larger than the 
true value by at most a (very moderate) constant factor.

For $u=x+iy$ (or $(r,\theta)$ in polar coordinates), (\ref{eq:koromo})
gives us
\begin{equation}\label{eq:fuju}\begin{aligned}
\Re(\phi(u)) &= \frac{x^2-y^2}{2} - \ell y + \theta \tau
= \frac{r^2-2 y^2}{2} - \ell y + \tau \arcsin \frac{y}{r}\\
&= \left( \frac{4 \tau}{\ell}\right)^2 \psi_0(\nu) = 
\ell^2 \rho^2 \psi_0(\nu), 
\end{aligned}\end{equation}
where, by (\ref{eq:hapday}), (\ref{eq:sadday}), and (\ref{eq:hoxha}),
\[\begin{aligned}
\psi_0(\nu) &= 
\frac{(\nu+\nu_0)^2}{2} (1 - 2 (\sin \theta_0 - c_1 \rho \nu)^2)\\ &- 
\frac{\nu+\nu_0}{\rho} (\sin \theta_0 - c_1 \rho \nu)
+ \frac{\arcsin(\sin \theta_0 - c_1 \rho \nu)}{4 \rho},
\end{aligned}\]
and
\begin{equation}\label{eq:oropel}
\nu = \frac{r-r_0}{\ell \rho},\;\;\;\;\;\;\;\;\;\;\;\;\;\;
\nu_0 = \frac{r_0}{\ell \rho}.
\end{equation}


By (\ref{eq:hapday}), (\ref{eq:sadday}) and (\ref{eq:oropel}),
\begin{equation}\label{eq:mowart}
\frac{y}{r} =
c_0 - c_1 \rho (\nu+\nu_0) = 
 \sin \theta_0 - c_1 \rho \nu\end{equation}
and so \begin{equation}\label{eq:hutu}
c_0 - c_1 \nu_0 \rho = \sin \theta_0.\end{equation} 
The variable $\nu$ will range within an interval 
\begin{equation}\label{eq:bonaire}
\lbrack \alpha_0,\alpha_1\rbrack \subset \left(
- \frac{1 - \sin \theta_0}{c_1 \rho} , \frac{\sin \theta_0}{c_1 \rho}
\right\rbrack .\end{equation}
(Here $\nu = - (1-\sin \theta_0)/(c_1 \rho)$ corresponds to the
intersection with the $y$-axis, and
$\nu = (\sin \theta_0)/(c_1 \rho)$ corresponds to the
intersection with the $x$-axis.)

We work out the expansions around $0$ of
\begin{equation}\label{eq:berlioz}\begin{aligned}
\frac{(\nu+\nu_0)^2}{2}  (1 - 2 (\sin \theta_0 &- c_1 \rho \nu)^2)
= \frac{\nu_0^2 \cos  2 \theta_0}{2} + (\nu_0 \cos 2 \theta_0 +
2 \nu_0^2 c_1 \rho \sin \theta_0) \nu \\
&+ \left(\frac{\cos 2 \theta_0}{2} + 4 c_1 \nu_0 \rho \sin \theta_0
- c_1^2 \rho^2 \nu_0^2 \right) \nu^2
\\ &+ 2 (- \nu_0 c_1^2 \rho^2 + c_1 \rho \sin \theta_0)  
\nu^3  - c_1^2 \rho^2 \nu^4,
\\
- \frac{\nu+\nu_0}{\rho} (\sin \theta_0 - c_1 \rho \nu)
&= - \frac{\nu_0 \sin \theta_0 }{\rho} 
+ \left( - \frac{\sin \theta_0}{\rho} + c_1 \nu_0\right) \nu + c_1 
\nu^2,\\
\frac{\arcsin(\sin \theta_0 - c_1 \rho \nu)}{4 \rho} &= 
 \frac{\theta_0}{4\rho} + \frac{1}{4\rho}
\sum_{k=1}^\infty \frac{P_k(\sin \theta_0)}{(\cos \theta_0)^{2k-1}} \frac{(-c_1 \rho)^k}{k!}
\nu^k\\
&= \frac{\theta_0}{4\rho} + \frac{1}{4\rho} 
\left(\frac{- c_1 \rho}{\cos \theta_0} \nu +
\frac{(c_1 \rho)^2 \sin \theta_0}{2 (\cos \theta_0)^3} \nu^2 + \dotsc
\right),
\end{aligned}\end{equation}
where $P_1(t) = 1$ and $P_{k+1}(t) = P_k'(t) (1-t^2) + (2k-1) t P(t)$ 
for $k\geq 1$. (This follows from $(\arcsin z)' = 1/\sqrt{1-z^2}$;
it is easy to show that $(\arcsin z)^{(k)} = P_k(z) (1-z^2)^{-(k-1/2)}$.)

We sum these three expressions and obtain a series $\psi_0(\nu) = 
\sum_k a_k \nu^k$.
We already know that
\begin{enumerate}
\item $a_0$
 equals the value of $\Re(\phi(u))/(\ell^2 \rho^2)$
at the saddle point $u_{0,+}$,
\item $a_1=0$,
\item \[a_2 = \frac{1}{2} \left(\frac{1}{\ell \rho}\right)^2 
\left(\frac{dr}{d\nu}\right)^2
|\phi''(u_{0,+})| \left|\frac{du}{dr}|_{r=r_0}\right|^2 
=  \frac{1}{2} |\phi''(u_{0,+})| \left|\frac{du}{dr}|_{r=r_0}\right|^2.\]
\end{enumerate}
Here, as we know from (\ref{eq:polan}), (\ref{eq:tintin}) and (\ref{eq:dada}),
\[|\phi''(u_{0,+})| = \frac{|-i \ell u_{0,+} + 2 i \tau|}{|u_{0,+}|^2} = 
\frac{\ell^2 \sqrt{\frac{j(\rho)}{2}} \sqrt{\upsilon^2-\upsilon}}{
\frac{\ell^2}{2} (\upsilon^2-\upsilon)} =
\sqrt{\frac{2 j(\rho)}{\upsilon^2 - \upsilon}},
\]
and, by (\ref{eq:hoxha}) and (\ref{eq:sprel}),
\begin{equation}\label{eq:jojot}\begin{aligned}
\left|\frac{du}{dr}|_{r=r_0}\right| 
= \frac{|du|}{|dr|}|_{r=r_0}
&= \sqrt{1 + \frac{r_0^2}{\frac{\ell^2}{c_1^2} - \left(\frac{c_0}{c_1}\ell 
- r_0\right)^2}} = \sqrt{1 + \frac{c_1^2}{\ell^2}
\frac{r_0^2}{1-\sin^2 \theta_0}}\\
&= \sqrt{1 + \frac{c_1^2}{\ell^2} \frac{\frac{\ell^2}{2} (\upsilon^2-\upsilon)}{
\frac{1}{2} + \frac{1}{2\upsilon}}} =
\sqrt{1 + c_1^2 \frac{\upsilon^2-\upsilon}{1+1/\upsilon}}\\
&= \sqrt{1+\Upsilon^2}.
\end{aligned}\end{equation}
Thus,
\[a_2 = \frac{1}{2} \sqrt{\frac{2 j(\rho)}{\upsilon^2 - \upsilon}}
(1+\Upsilon^2),\]
where $\Upsilon$ is as in (\ref{eq:brahms}).

Let us simplify our expression for $\psi_0(\nu)$ somewhat.
We can replace the third series in (\ref{eq:berlioz}) by a truncated Taylor series ending at $k=2$, namely,
\[\frac{\arcsin(\sin \theta_0 - c_1 \rho \nu)}{4 \rho} =
\frac{\theta_0}{4\rho} + \frac{1}{4\rho} 
\left(\frac{- c_1 \rho}{\cos \theta_0} \nu +
\frac{(c_1 \rho)^2 \sin \theta_1}{2 (\cos \theta_1)^3} \nu^2\right)\]
for some $\theta_1$ between $\theta_0$ and $\theta$.
Then $\theta_1 \in \lbrack 0,\pi/2\rbrack$, and so
\[\frac{\arcsin(\sin \theta_0 - c_1 \rho \nu)}{4 \rho} \geq
\frac{\theta_0}{4\rho} + \frac{1}{4\rho} \cdot
\frac{- c_1 \rho}{\cos \theta_0} \nu .\]

Since
\[R(\nu) = - c_1^2 \rho^2 \nu^2 + 2 (\sin \theta_0 - c_1 \rho \nu_0) c_1 \rho \nu \]
is a quadratic with negative leading coefficient, its minimum within
$\lbrack -\alpha_0,\alpha_1 \rbrack$ (see (\ref{eq:bonaire})) 
is bounded from below by $\min(R(-(1-\sin \theta_0)/(c_1 \rho)),
R((\sin \theta_0)/(c_1 \rho)))$. We compare
\[R\left(\frac{\sin \theta_0}{c_1 \rho}\right) = 
2 c_3
\sin \theta_0 - \sin^2 \theta_0,\]
where $c_3 = \sin \theta_0 - c_1 \rho \nu_0$,
and
\[\begin{aligned}
R\left(- \frac{1-\sin \theta_0}{c_1 \rho}\right) &=
- 2 c_3 (1 - \sin \theta_0) - (1-\sin \theta_0)^2\\
&= 2 c_3 \sin \theta_0 - \sin^2 \theta_0 - 2 c_3 - 1 + 2 \sin \theta_0
\end{aligned}.\]
The question is whether
\[\begin{aligned}R\left(- \frac{1-\sin \theta_0}{c_1 \rho}\right) - 
R\left(\frac{\sin \theta_0}{c_1 \rho}\right) &= 
-2 c_3 - 1 + 2 \sin \theta_0 \\ &= -2 (\sin \theta_0 - c_1 \rho \nu_0) - 1
+ 2 \sin \theta_0 \\ &= 2 c_1 \rho \nu_0 - 1\end{aligned}\]
is positive. It is:
\[c_1 \rho \nu_0 = \frac{c_1 \rho_0}{\ell} = c_1 \sqrt{\frac{\upsilon^2 - \upsilon}{2}} = \sqrt{\frac{1+1/\upsilon}{2}} \cdot \Upsilon \geq
\frac{\Upsilon}{\sqrt{2}},\]
and, as we know from (\ref{eq:amigusur}), $\Upsilon> 0.79837$ is greater
than $1/\sqrt{2}= 0.70710\dotsc$. Hence, by (\ref{eq:hutu}),
\[\begin{aligned}
R(\nu) &\geq R\left(\frac{\sin \theta_0}{c_1 \rho}\right) = 
2 c_3 \sin \theta_0 - \sin^2 \theta_0 =
\sin^2 \theta_0 - 2 c_1 \rho \nu_0 \sin \theta_0 \\
&= \sin^2 \theta_0 - 2 (c_0 - \sin \theta_0) \sin \theta_0 = 
3 \sin^2 \theta_0 - 2 c_0 \sin \theta_0 \\ &=
3 \sin^2 \theta_0 - 2 ((\cos \theta_0) \cdot \Upsilon + \sin \theta_0)
\sin \theta_0
= \sin^2 \theta_0 - (\sin 2 \theta_0) \cdot \Upsilon.\end{aligned}\]

We conclude that
\begin{equation}\label{eq:wort}
\psi_0(\nu) \geq \frac{\Re(\phi(u_{0,+}))}{\ell^2 \rho^2} + \eta \nu^2,
\end{equation}
where
\begin{equation}\label{eq:tess}\begin{aligned}
\eta &= a_2 - \frac{1}{4 \rho} \frac{(c_1 \rho)^2 \sin \theta_0}{2 (\cos
\theta_0)^3} 
+ \sin^2 \theta_0 - (\sin 2 \theta_0) \cdot \Upsilon .
\end{aligned}\end{equation}
We can simplify this further, using
\[\begin{aligned}
\frac{1}{4 \rho} \frac{(c_1 \rho)^2 \sin \theta_0}{2 (\cos
\theta_0)^3} &= \frac{\rho}{8}\cdot \frac{1 + 1/\upsilon}{\upsilon^2-\upsilon}
\cdot \Upsilon^2  \cdot \frac{\sqrt{\frac{1}{2} - \frac{1}{2 \upsilon}}}{
\left(\frac{1}{2} + \frac{1}{2 \upsilon}\right)^{3/2}}
= \frac{\rho}{4} \frac{\Upsilon^2}{\upsilon^2 - \upsilon}
\frac{\sqrt{1-1/\upsilon}}{\sqrt{1+1/\upsilon}} \\ &=
\frac{\rho}{4} \frac{\Upsilon^2}{\upsilon \sqrt{\upsilon^2 - 1}} =
\frac{\rho}{4} \frac{\Upsilon^2}{\sqrt{\frac{j+1}{2}}
\sqrt{\frac{j-1}{2}}} = \frac{\rho}{4} \frac{\Upsilon^2}{\sqrt{\rho^2/4}} = 
\frac{\Upsilon^2}{2}
\end{aligned}\]
 and (by (\ref{eq:lokrat}))
\[\begin{aligned}
\sin 2 \theta_0 &= 2 \sin \theta_0 \cos \theta_0 = 2 \sqrt{\frac{1}{4} - 
\frac{1}{4 \upsilon^2}} = \frac{\sqrt{\upsilon^2 - 1}}{\upsilon}
\\ &= \frac{\upsilon \sqrt{\upsilon^2-1}}{\upsilon^2} = 
\frac{\rho/2}{(j+1)/2} = \frac{\rho}{j+1}
.\end{aligned}\]
Therefore (again by (\ref{eq:lokrat}))
\begin{equation}\label{eq:malina}
\eta = \frac{1}{2} \sqrt{\frac{2 j}{\upsilon^2 - \upsilon}} (1 + 
\Upsilon^2) - \frac{\Upsilon^2}{2}
+ \frac{1}{2} - \frac{1}{2 \upsilon} - \frac{\rho}{j+1} \cdot \Upsilon.
\end{equation}


Now recall that our task is to bound the integral
\begin{equation}\label{eq:verdoux}\begin{aligned}
\int_{C_2} e^{-\Re(\phi(u))} &|u|^{\sigma-1} |du| =
\int_{\alpha_0}^{\alpha_1}
e^{-\ell^2 \rho^2 \psi_0(\nu)} (\ell \rho (\nu+\nu_0))^{\sigma-1} 
\left|\frac{du}{dr} \cdot \frac{dr}{d\nu}\right| d\nu\\
&\leq (\ell \rho)^{\sigma} 
e^{- \Re(\phi(u_{0,+}))} \cdot
\int_{\alpha_0}^{\alpha_1}
e^{-\eta \ell^2 \rho^2 \cdot \nu^2} (\nu+\nu_0)^{\sigma-1} 
\left|\frac{du}{dr} \right| d\nu .\end{aligned}\end{equation}
(We are using (\ref{eq:fuju}) and (\ref{eq:wort}).) 
Since $u_{0,+}$ is a solution to equation (\ref{eq:arbre}), we see from
(\ref{eq:fribar}) that
\[\begin{aligned}
\Re(\phi(u_{0,+})) &= 
\Re\left(\frac{u_{0,+}^2}{2} + \ell i u_{0,+} - i \tau \log u_{0,+}\right)
\\ &= \Re\left(\frac{\ell i u_{0,+}}{2} + \frac{i \tau}{2} 
+ \tau \arg(u_{0,+})\right) = \frac{1}{2} \Re(\ell i u_{0,+}) +
\tau \arg(u_0,+).\end{aligned}\]
We defined $y = \Re(-\ell i u_0)$ (after (\ref{eq:petpan})), and we computed
$y/2 - \arg(u_{0,+}) \tau$ in (\ref{eq:caglio}). This gives us
\[e^{-\Re(\phi(u_0,+))} =
e^{- \left(\arccos \frac{1}{\upsilon} - \frac{2 (\upsilon -1)}{\rho}
\right) \frac{\tau}{2} }.\]

If $\sigma\leq 1$, we can bound
\begin{equation}\label{eq:hobor}
(\nu+\nu_0)^{\sigma-1} \leq 
\begin{cases} \nu_0^{\sigma-1} &\text{if $\nu\geq 0$,}\\
(\alpha_0+\nu_0)^{\sigma-1} &\text{if $\nu<0$,}\end{cases}\end{equation}
provided that $\alpha_0+\nu_0>0$ (as will be the case). If $\sigma>1$, then
\[
(\nu+\nu_0)^{\sigma-1} \leq 
\begin{cases} \nu_0^{\sigma-1} &\text{if $\nu\leq 0$,}\\
(\alpha_1+\nu_0)^{\sigma-1} &\text{if $\nu>0$.} \end{cases}
\]

By (\ref{eq:sprel}),
\[\begin{aligned}
\left|\frac{du}{dr}\right| &= \sqrt{1 + \frac{(c_1 r/\ell)^2}{\cos^2 \theta}}
= \sqrt{1 + \frac{(c_1 \rho (\nu+\nu_0))^2}{1-
(\sin \theta_0 - c_1 \rho \nu)^2}}
\end{aligned}\]
(This diverges as $\theta\to \pi/2$; this is main reason why we cannot
actually follow the curve all the way to the $y$-axis.)
Since we are aiming at a bound that is tight only up to an order
of magnitude, we can be quite brutal here, as we were when
using (\ref{eq:hobor}): we
 bound $(c_1 r/\ell)^2$ from above by its value
when the curve meets the $x$-axis (i.e., when $r = c_0 \ell/c_1$).
We bound $\cos^2 \theta$ from below by its value when $\nu=\alpha_1$.
We obtain
\[\left|\frac{du}{dr}\right| =
 \sqrt{1 + \frac{c_0^2}{1 - (\sin \theta_0 - c_1 \rho\alpha_1)^2}} = 
\sqrt{1 + \frac{c_0^2}{\cos^2 \theta_-}},\]
where $\theta_-$ is the value of $\theta$ when $\nu=\alpha_1$.

Finally, we complete the integral in (\ref{eq:verdoux}), we split
it in two (depending on whether $\nu\geq 0$ or $\nu<0$) and use
\[\int_{0}^{\alpha} e^{\eta \ell^2 \rho^2\cdot \nu^2} d\nu
\leq \frac{1}{\ell \rho \sqrt{\eta}} \int_{0}^\infty
e^{-\nu^2} d\nu  = \frac{\sqrt{\pi}/2}{\ell \rho \sqrt{\eta}} .\]

Therefore,
\begin{equation}\label{eq:munster}\begin{aligned}
\int_{C_2} &e^{-\Re(\phi(u))} |u|^{\sigma-1} |du| \\&=
(\ell \rho)^{\sigma} 
e^{- \left(\arccos \frac{1}{\upsilon} - \frac{2 (\upsilon -1)}{\rho}
\right)
 \frac{\tau}{2}} \sqrt{1 + \frac{c_0^2}{\cos^2 \theta_-}} \cdot
 \frac{\sqrt{\pi}/2}{\ell \rho \sqrt{\eta}} \cdot \left(\nu_0^{\sigma-1} + 
(\alpha_{j_\sigma} + \nu_0)^{\sigma-1}\right)\\
&= \frac{\sqrt{\pi}}{2} r_0^{\sigma-1} \left(
\left(1 + \left(1 + \frac{\alpha_{j_\sigma}}{\nu_0}\right)^{\sigma-1}\right)
\sqrt{1 + \frac{c_0^2}{\cos^2 \theta_-}}\right)
\frac{e^{- \left(\arccos \frac{1}{\upsilon} - \frac{2 (\upsilon -1)}{\rho}
\right) \frac{\tau}{2}}}{\sqrt{\eta}},
\end{aligned}
\end{equation}
where $j_\sigma=0$ if $\sigma\leq 1$ and $j_\sigma=1$ if $\sigma>1$.
We can set $\alpha_1
= (\sin \theta_0)/(c_1 \rho)$. We can also express $\alpha_0+\nu_0$
in terms of $\theta_-$:
\begin{equation}\label{eq:averill}\alpha_0 + \nu_0 = \frac{r_-}{\ell \rho} = 
\frac{(c_0 - \sin \theta_-) \frac{\ell}{c_1}}{\ell \rho} = 
\frac{c_0 - \sin \theta_-}{c_1 \rho}.\end{equation}
Since $\nu_0 = r_0/(\ell \rho)$ (by (\ref{eq:oropel}))
and $r_0$ is as in (\ref{eq:mordo}),
\[\nu_0 = \frac{\sqrt{\upsilon(\rho)^2 - \upsilon(\rho)}}{\sqrt{2} \rho}.\]
Definition (\ref{eq:horem}) implies immediately that $\Upsilon\leq 1$. 
Thus, by (\ref{eq:hoxha}),
\begin{equation}\label{eq:sumero}
c_1 \rho \nu_0 = \Upsilon \cdot \sqrt{2 (1+1/\upsilon)} \leq
 2 \Upsilon \leq 2,\end{equation}
while, by (\ref{eq:amigusur}),
\begin{equation}\label{eq:melanie}
c_1 \rho \nu_0 = \Upsilon \cdot \sqrt{2 (1+1/\upsilon)} \geq
0.79837\cdot \sqrt{2}
\end{equation}
By (\ref{eq:averill}) and (\ref{eq:sumero}),
\begin{equation}\label{eq:jorom}
\left(1 + \frac{\alpha_0}{\nu_0}\right)^{-1} =
\frac{\nu_0}{\alpha_0 + \nu_0} = \frac{c_1 \rho \nu_0}{c_0 - \sin \theta_-}
\leq \frac{2}{c_0 - \sin \theta_-},
\end{equation}
whereas
\[\left(1 + \frac{\alpha_1}{\nu_0}\right) = 1 + \frac{\sin \theta_0}{
c_1 \rho \nu_0} \leq 1 + \frac{1/\sqrt{2}}{0.79837\cdot \sqrt{2}} \leq
1.62628.\]
We will now use some (rigorous) numerical bounds, proven in 
Appendix \ref{subs:exbi}. First of all, by (\ref{eq:pekka}),
$c_0>1$ for all $\rho>0$;
this assures us that $c_0-\sin \theta_->0$, and so
the last expression in (\ref{eq:jorom}) is well defined. 
By (\ref{eq:averill}),
this also shows that $\alpha_0+\nu_0>0$, i.e., the curve $C$ stays within
the first quadrant for $0\leq \theta\leq \pi/2$, as we said before.

We would also like to have an upper bound for
\begin{equation}\label{eq:jompo}
\sqrt{\frac{1}{\eta} \left(1 + \frac{c_0^2}{\cos^2 \theta_-}\right)},\end{equation}
using (\ref{eq:malina}). With this in mind, we finally choose $\theta_-$:
\begin{equation}\label{eq:chothe}
\theta_- = \frac{\pi}{4}.\end{equation}
Thus, by (\ref{eq:suit}),
\[\sqrt{
\frac{1}{\eta} \left(1 + \frac{c_0^2}{\cos^2 \theta_-}\right)} \leq
\sqrt{\frac{1 + 2 c_0^2}{\eta}} \leq \sqrt{\min(5, 0.86 \rho)} \leq
\min(\sqrt{5},0.93 \sqrt{\rho}).\]
We also get
\[ \frac{2}{c_0 - \sin \theta_-} \leq \frac{2}{1 - 1/\sqrt{2}} \leq 7.82843.\]

Finally, by (\ref{eq:frais}),
\[
\sqrt{\frac{\upsilon^2-\upsilon}{2}} \geq \begin{cases}
\rho/6 &\text{if $\rho\leq 4$,}\\
\frac{\sqrt{\rho}}{2} - \frac{1}{2^{3/2}}
\leq \left(1 - \frac{1}{2^{3/2}}\right) \frac{\sqrt{\rho}}{2}  &\text{if $\rho>4$}
\end{cases}\]
and so, since
$\rho \ell = 4 \tau/\ell$, $\sqrt{\rho} \ell = 2 \sqrt{\tau}$ and
$(1-1/2^{3/2})\leq 2/3$,
(\ref{eq:mordo}) gives us
\[
r_0
\geq \begin{cases}
\frac{2}{3} \frac{\tau}{\ell} &\text{if $\tau\leq \ell^2$}\\
\frac{2}{3} \sqrt{\tau}
 &\text{if $\tau> \ell^2$}\end{cases} = \frac{2}{3} \min\left(\frac{\tau}{\ell},\sqrt{\tau}\right).
\]


We conclude that
\begin{equation}\label{eq:lutin}
\int_{C_2} e^{-\Re(\phi(u))} |u|^{\sigma-1} |du| =
C_{\tau,\ell} \cdot e^{- \left(\arccos \frac{1}{\upsilon} - \frac{2 (\upsilon -1)}{\rho}
\right) \frac{\tau}{2}},
\end{equation}
where 
\begin{equation}\label{eq:hojo}
C_{\tau,\ell} = \min\left(2, \frac{3.3 \sqrt{\tau}}{\ell}\right)
\left(1 + \max\left(7.83^{1-\sigma},1.63^{\sigma-1}\right)\right) \left(\frac{3/2}{\min(\tau/\ell,\sqrt{\tau})}
\right)^{1-\sigma}
\end{equation}
for all $\tau>0$, $\ell>0$ and all $\sigma$.
By reflection on the $x$-axis, the same bound holds for $\tau<0$,
$\ell<0$ and all $\sigma$.
Lastly, (\ref{eq:lutin}) is also valid for $\ell=0$, provided we
replace (\ref{eq:hojo}) and the exponent of (\ref{eq:lutin}) by their
 limits as $\ell\to 0^+$.

\subsection{The integral over the rest of the contour}\label{subs:rest}

It remains to complete the contour. Since we have set $\theta_- = \pi/4$,
$C_1$ will be a segment of the ray at 45 degrees
from the $x$-axis, counterclockwise (i.e., $y=x$, $x\geq 0$).
The segment will go from $(0,0)$ up to $(x,y) = (r_-/\sqrt{2},r_-/\sqrt{2})$, where,
by (\ref{eq:hapday}),
\[\frac{1}{\sqrt{2}} = \frac{y}{r_-} = 
- \frac{c_1}{\ell} r + c_0,\]
and so
\begin{equation}\label{eq:helga}
r_- = \frac{\ell}{c_1} \left(c_0 - \frac{1}{\sqrt{2}}\right).\end{equation}
Let $w=(1+i)/\sqrt{2}$. 
Looking at (\ref{eq:fribar}), we see that
\begin{equation}\label{eq:calvaire}\begin{aligned}
\left|\int_{C_1} e^{-u^2/2} e(\delta u) u^{s-1} du\right|
&= \left|\int_{C_1} e^{-\phi(u)} u^{\sigma-1} du \right|\\
&\leq \int_{C_1} e^{-\Re(\phi(u))} |u|^{\sigma-1} |du| = 
\int_0^{r_-} e^{-\Re(\phi(t w))} t^{\sigma-1} dt,
\end{aligned}\end{equation}
where $\phi(u)$ is as in (\ref{eq:koromo}).
Here
\begin{equation}\label{eq:pixie}
\Re(\phi(t w)) = \Re\left(\frac{t^2}{2} i + \ell i w t - i \tau
\left(\log t + i \frac{\pi}{4}\right)\right)\\
= - \frac{\ell t}{\sqrt{2}} + \frac{\pi}{4} \tau,\end{equation}
and, by (\ref{eq:anka}),
\[\begin{aligned}
- \frac{\ell r_-}{\sqrt{2}} + \frac{\pi}{4} \tau \geq
- 0.07639 \ell^2 \rho
 + \frac{\pi}{4} \tau =
\left(\frac{\pi}{4} - 0.30556\right) \tau > 0.4798 \tau.\end{aligned}\]

Consider first the case $\sigma\geq 1$. Then
\[\begin{aligned}
\int_0^{r_-} e^{-\Re(\phi(t w))} t^{\sigma-1} dt
\leq r_-^{\sigma-1} \int_0^{r_-}
e^{\frac{\ell t}{\sqrt{2}} - \frac{\pi}{4} \tau} dt \leq r_-^\sigma 
e^{\frac{\ell r_-}{\sqrt{2}} - \frac{\pi}{4} \tau}.
\end{aligned}\]
By (\ref{eq:helga}) and (\ref{eq:anka}),
\begin{equation}\label{eq:anne}\begin{aligned}
r_- \leq \sqrt{\rho} \ell/2 \leq \sqrt{\tau},
\end{aligned}\end{equation}
Hence, for $\sigma\geq 1$,
\begin{equation}\label{eq:traviata}
\left|\int_{C_1} e^{-u^2/2} e(\delta u) u^{s-1} du\right| \leq
\tau^{\sigma/2} e^{- 0.4798\tau}.\end{equation}
 
Assume now that $0\leq \sigma <1$, $s\ne 0$.
We can see that it is wise to start by an integration by parts, so
as to avoid convergence problems arising from the term $t^{\sigma-1}$
within the integral as $\sigma\to 0^+$. We have
\[
\int_{C_1} e^{-u^2/2} e(\delta u) u^{s-1} du =
e^{-u^2/2} e(\delta u) \frac{u^{s}}{s} |^{w r_-}_0 - 
\int_{C_1} \left(e^{-u^2/2} e(\delta u)\right)' \frac{u^s}{s} du.\]
By (\ref{eq:pixie}),
\[\left|e^{-u^2/2} e(\delta u) \frac{u^{s}}{s} |^{w r_-}_0 \right| 
= e^{-\Re(\phi(w r_-))} \cdot \frac{r_-^\sigma}{|s|}
\leq \frac{r_-^\sigma}{\tau} \cdot e^{\frac{\ell r_-}{\sqrt{2}} - \frac{\pi}{4} \tau}
\]
As for the integral,
\begin{equation}\begin{aligned}
\int_{C_1} \left(e^{-u^2/2} e(\delta u)\right)' \frac{u^s}{s} du &=
- \int_{C_1} (u + \ell i)
e^{-u^2/2 - \ell i u} \frac{u^s}{s} du \\ &=
- \frac{1}{s} \int_{C_1} e^{-u^2/2} e(\delta u) u^{s+1} du - \frac{\ell i}{s}
\int_{C_1} e^{-u^2/2} e(\delta u) u^{s} du.\end{aligned}\end{equation}
Hence, by (\ref{eq:calvaire}) and (\ref{eq:pixie}),
\[\begin{aligned}
\left|\int_{C_1} \left(e^{-u^2/2} e(\delta u)\right)' \frac{u^s}{s} du\right| 
&\leq \frac{1}{|s|} \int_0^{r_-} e^{\frac{\ell t}{\sqrt{2}} -
\frac{\pi}{4} \tau} t^{\sigma+1} dt + \frac{\ell}{|s|}
\int_0^{r_-} e^{\frac{\ell t}{\sqrt{2}} - \frac{\pi}{4} \tau} t^{\sigma} dt\\
&\leq \left(\frac{r_-^{\sigma+1}}{|s|} 
+ \frac{\ell r_-^{\sigma}}{|s|}\right)
\int_0^{r_-} e^{\frac{\ell t}{\sqrt{2}} - \frac{\pi}{4} \tau} dt\\
&\leq \frac{r_-^{\sigma+1} + \ell r_-^{\sigma}}{\tau} \cdot \min\left(
\frac{\sqrt{2}}{\ell}, r_-\right)
 \cdot e^{\frac{\ell r_-}{2} - \frac{\pi}{4} \tau}\\
&\leq \left(\frac{r_-^{\sigma+2}}{\tau} + \frac{\sqrt{2} r_-^\sigma}{\tau}
\right) 
\cdot e^{\frac{\ell r_-}{2} - \frac{\pi}{4} \tau}.
\end{aligned}\]
By (\ref{eq:anne}),
\[\begin{aligned}
\left(\frac{r_-^{\sigma+2}}{\tau} + \frac{(1+\sqrt{2}) r_-^\sigma}{\tau}
\right)\leq \frac{\tau^{\frac{\sigma}{2}+1} + (1+\sqrt{2}) 
\tau^{\frac{\sigma}{2}}}{\tau}
.\end{aligned}\]
We conclude that
\[\begin{aligned}
\left|\int_{C_1} e^{-u^2/2} e(\delta u) u^{s-1} du\right| &\leq
\frac{\tau^{\frac{\sigma}{2}+1} + (1+\sqrt{2}) 
\tau^{\frac{\sigma}{2}}}{\tau}
 \cdot e^{-\frac{\ell r_-}{2} + \frac{\pi}{4} \tau}\\
&\leq \left(1 + \frac{1+\sqrt{2}}{\tau}\right) \tau^{\frac{\sigma}{2}}
\cdot
e^{-0.4798 \tau}\end{aligned}\]
when $\sigma \in \lbrack 0,1)$; by (\ref{eq:traviata}), this is true for
$\sigma\geq 1$ as well.

Now let us examine the contribution of the last segment $C_3$ of the contour.
Since $C_2$ hits the $x$-axis at $c_0\ell/c_1$, we define $C_3$ to be
the segment of the $x$-axis going from $x=c_0\ell/c_1$ till $x=\infty$.
Then
\begin{equation}\begin{aligned}
&\left|\int_{C_3} e^{-t^2/2} e(\delta t) t^s \frac{dt}{t}\right| =
\left|\int_{\frac{c_0 \ell}{c_1}}^{\infty} e^{-x^2/2} e(\delta x) x^s \frac{dx}{x}
\right|
\leq \int_{\frac{c_0 \ell}{c_1}}^{\infty} e^{-x^2/2} x^{\sigma}
\frac{dx}{x}.
\end{aligned}\end{equation}
Now
\[\begin{aligned}
\left(-e^{-x^2/2} x^{\sigma-2}\right)' &= e^{-x^2/2} x^{\sigma-1} -
(\sigma-2) e^{-x^2/2} x^{\sigma-3}\\
\left(-e^{-x^2/2} (x^{\sigma-2} + (\sigma-2) x^{\sigma-4})\right)' 
&= e^{-x^2/2} x^{\sigma-1} -
(\sigma-2) (\sigma-4) e^{-x^2/2} x^{\sigma-5}\end{aligned}\]
and so on, implying that
\[\begin{aligned}
&\int_t^\infty e^{-x^2/2} x^\sigma \frac{dx}{x}\\ &\leq e^{-x^2/2}\cdot \begin{cases} 
x^{\sigma-2} &\text{if $0\leq \sigma \leq 2$,}\\
\left(x^{\sigma-2} + (\sigma-2) x^{\sigma-4}\right) &\text{if $2\leq \sigma \leq 4$.}\\
x^{\sigma-2} + (\sigma-2) x^{\sigma-4} + (\sigma-2) (\sigma-4) x^{\sigma-6}
&\text{if $4\leq \sigma \leq 6$,}
\end{cases}\end{aligned}\]
and so on.
By (\ref{eq:ruwo}),
\[\frac{c_0 \ell}{c_1} \geq \min\left(\frac{\tau}{\ell}, \frac{5}{4}
\sqrt{\tau}\right).\]
We conclude that
\[\left|\int_{C_3} e^{-t^2/2} e(\delta t) t^s \frac{dt}{t}\right| \leq
P_\sigma\left(\min\left(\frac{\tau}{\ell}, \frac{5}{4}
\sqrt{\tau}\right)\right)
e^{-  \min\left(\frac{1}{2} \left(\frac{\tau}{\ell}\right)^2, \frac{25}{32}
\tau\right)},
\]
where we can set $P_\sigma(x) = x^{\sigma-2}$ if $\sigma\in \lbrack 0,2\rbrack$,
$P_\sigma(x) = x^{\sigma-2} + (\sigma-2) x^{\sigma-4}$ if $\sigma\in \lbrack
2,4\rbrack$ and $P_\sigma(x) = x^{\sigma-2} + (\sigma-2) x^{\sigma-4} + 
\dotsc + (\sigma-2k) x^{\sigma-2(k+1)}$ if $\sigma \in \lbrack 2k,2(k+1)\rbrack$.

\begin{center}
* * *
\end{center}

We have left the case $\ell<0$ for the very end. In this case, we can afford to
use a straight ray from the origin as our contour of integration.
Let $C'$ be the ray at angle $\pi/4-\alpha$ from the origin, i.e.,
$y = (\tan (\pi/4-\alpha)) x$, $x>0$, where $\alpha>0$  is small.
Write $v = e^{(\pi/4-\alpha) i}$. The integral to be estimated is
\[I = \int_{C'} e^{-u^2/2} e(\delta u) u^{s-1} du.\]
 Let us try $\alpha=0$ first.
Much as in (\ref{eq:calvaire}) and (\ref{eq:pixie}), we obtain, for
$\ell<0$,
\begin{equation}\label{eq:nort}\begin{aligned}
|I| &\leq \int_0^\infty e^{-\left(\frac{- \ell t}{\sqrt{2}} + \frac{\pi}{4} \tau 
\right)} t^{\sigma-1} dt =
e^{-\frac{\pi}{4} \tau}
\int_0^\infty e^{- |\ell| t/\sqrt{2}} t^{\sigma} \frac{dt}{t} \\ &= 
e^{-\frac{\pi}{4} \tau} \cdot \left(\frac{\sqrt{2}}{|\ell|}\right)^{\sigma}
\int_0^\infty e^{-t} t^{\sigma} \frac{dt}{t} = 
\left(\frac{\sqrt{2}}{|\ell|}\right)^{\sigma}
\Gamma(\sigma) \cdot e^{-\frac{\pi}{4} \tau}
\end{aligned}\end{equation}
for $\sigma>0$. Recall that $\Gamma(\sigma)\leq \sigma^{-1}$ for
$0<\sigma<1$ (because $\sigma \Gamma(\sigma) = \Gamma(\sigma+1)$ and
$\Gamma(\sigma)\leq 1$ for all $\sigma\in \lbrack 1,2\rbrack$; the inequality
$\Gamma(\sigma)\leq 1$ for $\sigma\in \lbrack 1,2\rbrack$ can in turn be proven
by $\Gamma(1)=\Gamma(2)=1$, $\Gamma'(1)<0<\Gamma'(2)$ and the convexity of
$\Gamma(\sigma)$). We see that,
while (\ref{eq:nort}) is very good in most cases, it poses problems
when either $\sigma$ or $\ell$ is close to $0$.

Let us first deal with the issue of $\ell$ small. For general $\alpha$
and $\ell\leq 0$,
\[\begin{aligned}
|I| 
&\leq \int_0^\infty e^{-\left(\frac{t^2}{2} \sin 2 \alpha -
\ell t \cos \left(\frac{\pi}{4} - \alpha\right) + 
\left(\frac{\pi}{4} - \alpha\right) \tau 
\right)} t^{\sigma-1} dt \\ &\leq
e^{- \left(\frac{\pi}{4} - \alpha\right) \tau}
\int_0^\infty e^{- \frac{t^2}{2} \sin 2 \alpha} t^{\sigma} \frac{dt}{t} = 
\frac{e^{- \left(\frac{\pi}{4} - \alpha\right) \tau}}{(\sin 2 \alpha)^{\sigma/2}}
\int_0^\infty e^{-\frac{t^2}{2}} t^\sigma \frac{dt}{t} \\ &=
\frac{e^{- \left(\frac{\pi}{4} - \alpha\right) \tau}}{(\sin 2 \alpha)^{\sigma/2}}
\cdot 2^{\sigma/2-1} \int_0^\infty e^{-y} y^{\frac{\sigma}{2}} \frac{dy}{y} =
\frac{e^{\alpha \tau}}{2 (\sin 2 \alpha)^{\sigma/2}}
\cdot 2^{\sigma/2} \Gamma(\sigma/2) e^{-\frac{\pi}{4} \tau}.
\end{aligned}\]
Here we can choose $\alpha = (\arcsin 2/\tau)/2$ (for $\tau\geq 2$).
Then $2 \alpha \leq (\pi/2) \cdot (2/\tau) = \pi/\tau$, and so 
\begin{equation}\label{eq:dvorak}
|I| \leq \frac{e^{\frac{\pi}{2 \tau}\cdot \tau}}{2 (2/\tau)^{\sigma/2}}
\cdot 2^{\sigma/2} \Gamma(\sigma/2) e^{-\frac{\pi}{4} \tau} \leq
\frac{e^{\pi/2}}{2} \tau^{\sigma/2} \Gamma(\sigma/2) \cdot 
 e^{- \frac{\pi}{4} \tau}.
\end{equation}

The only issue that remains is that $\sigma$ may be close to $0$, in which
case $\Gamma(\sigma/2)$ can be large.
We can resolve this, as before, by doing an integration by parts.
In general, for $-1<\sigma <1$, $s\ne 0$:
\begin{equation}\label{eq:jokon}\begin{aligned}
|I| &\leq
e^{-u^2/2} e(\delta u) \frac{u^{s}}{s} |^{v \infty}_0 - 
\int_{C'} \left(e^{-u^2/2} e(\delta u)\right)' \frac{u^s}{s} du \\ &=
\int_{C'} (u + \ell i)
e^{-u^2/2 - \ell i u} \frac{u^s}{s} du \\ &=
\frac{1}{s} \int_{C'} e^{-u^2/2} e(\delta u) u^{s+1} du + \frac{\ell i}{s}
\int_{C'} e^{-u^2/2} e(\delta u) u^{s} du.\end{aligned}\end{equation}
Now we apply (\ref{eq:nort}) with $s+1$ and $s+2$ instead of $s$, and get that
\[\begin{aligned}
|I|
 &=
\frac{1}{|s|} \left(\frac{\sqrt{2}}{|\ell|}\right)^{\sigma+2}
\Gamma(\sigma+2) \cdot e^{-\frac{\pi}{4} \tau} +
\frac{|\ell|}{|s|} \left(\frac{\sqrt{2}}{|\ell|}\right)^{\sigma+1}
\Gamma(\sigma+1) \cdot e^{-\frac{\pi}{4} \tau}\\
&\leq \frac{1}{\tau} \left(\frac{\sqrt{2}}{|\ell|}\right)^\sigma
 \left(\frac{4}{\ell^2} + \sqrt{2}\right) e^{-\frac{\pi}{4} \tau}.
\end{aligned}\]
Alternatively, we may apply (\ref{eq:dvorak}) and obtain
\[\begin{aligned}
|I| &\leq
\frac{1}{|s|}
\frac{e^{\pi/2}}{2}
 \Gamma((\sigma+2)/2) \cdot \tau^{(\sigma+2)/2} e^{- \frac{\pi}{4} \tau} +
\frac{|\ell|}{|s|}
\frac{e^{\pi/2}}{2}
 \Gamma((\sigma+1)/2) \cdot \tau^{(\sigma+1)/2} e^{- \frac{\pi}{4} \tau}\\
&\leq \frac{e^{\pi/2} \tau^{\sigma/2}}{2}
 \left(1 + \frac{\sqrt{\pi} |\ell|}{\sqrt{\tau}}\right) 
e^{-\frac{\pi}{4} \tau} 
\end{aligned}\]
for $\sigma \in \lbrack 0,1\rbrack$,
where we are using the facts that $\Gamma(s)\leq \sqrt{\pi}$ for
$s\in \lbrack 1/2,1\rbrack$ and
 $\Gamma(s)\leq 1$ for $s\in \lbrack 1,2\rbrack$.

\subsection{Conclusion}\label{subs:conclusion}

Summing (\ref{eq:lutin}) with the bounds obtained in \S \ref{subs:rest},
we obtain our final estimate.
Recall that we can reduce the case $\tau<0$ to the case $\tau>0$
by reflection. This finishes the proof of Theorem \ref{thm:princo}.
Let us now extract its main corollary -- in effect, a simplified 
restatement of the theorem.

\begin{proof}[Proof of Corollary \ref{cor:amanita1}]

Let $E(\rho)$ be as in (\ref{eq:cormo}). 
Let
\begin{equation}\label{eq:dorof}
L(\rho) = \begin{cases} 0.1598 &\text{if $\rho\geq 1.5$,}\\
0.1065 \rho &\text{if $\rho< 1.5$.}\end{cases}\end{equation}
Note that $0.1598 \leq E(1.5)$, whereas $0.1065 \leq E(1.5)/1.5$.
We claim that $E(\rho)\geq L(\rho)$:
this is so for $\rho\geq 1.5$ because $E(\rho)$ is increasing on $\rho$, 
and for $\rho\leq 1.19$ because of (\ref{eq:rhosm}), and for $\rho \in
\lbrack 1.19, 1.5\rbrack$ by the bisection method 
(with $20$ iterations).


By Thm.~\ref{thm:princo}, for $0\leq k\leq 2$ and $s=\sigma + i \tau$
with $\sigma\in \lbrack 0,1\rbrack$ and $|\tau|\geq \max(4 \pi^2 |\delta|,
100)$,
\[|F_\delta(s+k)| + |F_\delta(k+1-s)|\] 
is at most
\begin{equation}\label{eq:sobror}
C_{0,\tau,\ell} \cdot e^{-E\left(\frac{|\tau|}{(\pi \delta)^2}\right)
  \cdot \tau} + 
C_{1,\tau} \cdot e^{-0.4798\tau} + C_{2,\tau,\ell} \cdot
e^{-\min\left(\frac{1}{8} \left(\frac{\tau}{\pi \delta}\right)^2,
\frac{25}{32} |\tau|\right)} + C_{\tau}' e^{-\frac{\pi}{4} |\tau|},
\end{equation}
where $C_{0,\tau,\ell}$ is at most
\[\begin{aligned}
&2\cdot (1+7.83^{1-\sigma})
\left(\frac{3/2}{2\pi}\right)^{1-\sigma}\leq 4.217\;\;\;\;\;\;\;\;
\text{if $k=0$,}\\
&2\cdot (1+1.63) \left(\frac{\min\left(
\frac{|\tau|}{|\ell|},\sqrt{|\tau|}\right)}{3/2}\right)
\leq 3.507 \min\left(\frac{|\tau|}{|\ell|}, \sqrt{|\tau|}\right)
\;\;\text{if $k=1$,}\\
&2\cdot (1+1.63^2) \cdot \left(\frac{\min\left(
\frac{|\tau|}{|\ell|},\sqrt{|\tau|}\right)}{3/2}\right)^2\leq
3.251 \min\left(\left(\frac{\tau}{|\ell|}\right)^2,|\tau|\right)\;\;
\text{if $k=2$,}
\end{aligned}\]
and where
\begin{equation}\label{eq:sokor}\begin{aligned}
C_{1,\tau}&\leq \left(1 + \frac{1 + \sqrt{2}}{|\tau|}\right)
|\tau|^{\frac{k+1}{2}},\;\;\;\;\;\;
C_{2,\tau,\ell}\leq \begin{cases}
\min\left(\frac{|\tau|}{|\ell|}, 
\frac{5}{4} \sqrt{|\tau|}\right)^{-1}
& \text{if $k=0$,}\\
1 & \text{if $k=1$,}\\
\min\left(\frac{|\tau|}{|\ell|}, 
\frac{5}{4} \sqrt{|\tau|}\right) 
+ 1& \text{if $k=2$,}\end{cases}\\
C_{\tau,\ell}' &\leq 
\frac{e^{\pi/2} |\tau|^{\frac{k+1}{2}}}{2}
\cdot \begin{cases}
1 + \frac{\sqrt{\pi} |\ell|}{\sqrt{|\tau|}}
&\text{for $k=0$,}\\
\sqrt{\pi} &\text{for $k=1$},\\
1 &\text{for $k=2$.}\end{cases}
\end{aligned}\end{equation}
(We define, as usual, $\ell = - 2\pi \delta$.)

Since $|\tau|\geq \max(2 \pi |\ell|,100)$,
\begin{equation}\label{eq:doutes}\begin{aligned}
&\left(1 + \frac{1+\sqrt{2}}{|\tau|}\right) +
\frac{e^{\pi/2}}{2} \max\left(1 + \frac{\sqrt{\pi} |\ell|}{\sqrt{|\tau|}},
\sqrt{\pi}\right)\cdot e^{-\left(\frac{\pi}{4} - 0.4798\right)
|\tau|}
 \\ &\leq 1.0242 + 2.406 \max\left(1 +
\frac{|\tau|^{1/2}}{2 \sqrt{\pi}},\sqrt{\pi}\right) e^{-\left(\frac{\pi}{4} - 0.4798\right)
|\tau|}\\ &\leq 1.0242 +
9.194 e^{- \left(\frac{\pi}{4} - 0.4798\right) 100}
 \leq 1.025
\end{aligned}\end{equation}
and so 
\begin{equation}\label{eq:tipi}
C_{1,\tau} \cdot e^{-0.4798 |\tau|} + 
C_{\tau}' e^{-\frac{\pi}{4} |\tau|} \leq
1.025 |\tau|^{\frac{k+1}{2}} e^{-0.4798 |\tau|}.
\end{equation}
Since $|\tau|\geq 100$ and $t^r e^{-c t}$ ($r,c\geq 0$) is decreasing for 
$t\geq r/c$,
\begin{equation}\label{eq:cocoko}\begin{aligned}
1.025 |\tau|^{(k+1)/2} e^{-0.4798 |\tau|} &\leq 1.025 \cdot 10^{k+1} \cdot e^{-
(0.4798-0.1598) \cdot 100} e^{- 0.1598 |\tau|}\\ &\leq
1.3\cdot 10^{-11} \cdot e^{- 0.1598 |\tau|}\end{aligned}\end{equation}
for $k\leq 2$. 

Yet again by $|\tau|\geq \max(4 \pi^2 |\delta|,100)$,
\begin{equation}\label{eq:coudra}\begin{aligned}
e^{-\frac{1}{8} \left(\frac{\tau}{\pi \delta}\right)^2}
&\leq e^{- 0.0184 \left(\frac{\tau}{\pi \delta}\right)^2} 
\cdot e^{- 0.1066 \left(\frac{\tau}{\pi \delta}\right)^2}
\leq 0.055 \cdot e^{- 0.1066\left(\frac{\tau}{\pi \delta}\right)^2},\\ 
e^{- \frac{25}{32} \tau} &\leq 1.03 \cdot 10^{-27}\cdot e^{- 0.1598 |\tau|}.
\end{aligned}\end{equation}

We also get  (starting from (\ref{eq:sokor}))
\begin{equation}\label{eq:tangomon}
\frac{0.055\cdot C_{2,\tau,\ell}}{\min\left(\left(\frac{\tau}{|\ell|}\right)^k,
|\tau|^{\frac{k}{2}}\right)} \leq
\begin{cases}
0.055/\min(2\pi,12.5) \leq 0.0088
&\text{if $k=0$,}\\
0.055/2\pi \leq 0.0088
&\text{if $k=1$,}\\
0.055 ((2\pi)^{-1} + (2\pi)^{-2}) \leq 0.0102
&\text{if $k=2$.}
\end{cases}
\end{equation}

It is easy to see from (\ref{eq:dorof}) that
\[e^{-L\left(\frac{|\tau|}{(\pi \delta)^2}\right) \cdot |\tau|} \geq
\max\left(e^{-0.1066  \left(\frac{\tau}{\pi \delta}\right)^2},
e^{-0.1598 |\tau|}\right),\]
whether $|\tau|/(\pi \delta)^2$ is greater than $1.5$ or not.
We conclude that (\ref{eq:sobror}) is at most
\[(4.217+0.0089)  e^{-L(|\tau|/(\pi \delta)^2)}\]
if $k=0$, at most
\[(3.507+0.0089) \min(|\tau|/\ell,\sqrt{|\tau|}) e^{-L(|\tau|/(\pi \delta)^2)}\]
if $k=1$, and at most
\[(3.251+0.0103) \min((|\tau|/\ell)^2,|\tau|) e^{-L(|\tau|/(\pi \delta)^2)}\]
if $k=2$. (We see that the error terms coming from (\ref{eq:cocoko})
and the second line of (\ref{eq:coudra}) are being absorbed by the least significant digits of the bound from (\ref{eq:tangomon}).)
We simplify matters further by using
$\min((|\tau|/\ell)^2,|\tau|) = (|\tau|/|\ell|)^2$ for
$|\tau| < 1.5 (\pi \delta)^2$, and bounding
$\min((|\tau|/\ell)^2,|\tau|) \leq |\tau|$ for
$|\tau| \geq 1.5 (\pi \delta)^2$.
\end{proof}



Let us take a second look at Thm.~\ref{thm:princo}.
While, in the present paper,
 we will use it only through Corollary \ref{cor:amanita1},
 it seems worthwhile to say a few words more on
 the behavior of its bounds.

The terms in (\ref{eq:wilen}) other than $C_{0,\tau,\ell} \cdot 
e^{-E(|\tau|/(\pi \delta)^2) |\tau|}$ are usually very small. 
In practice, Thm.~\ref{thm:princo} should be applied when
$|\tau|/2\pi |\delta|$ is larger than a moderate constant (say $8$)
and $|\tau|$ is larger than a somewhat larger constant (say $100$).
The assumptions of Cor.~\ref{cor:amanita1} are thus typical.

For comparison, the Mellin transform of $e^{-t^2/2}$ (i.e., $F_0 = M f_0$) is
$2^{s/2-1} \Gamma(s/2)$, which decays like $e^{-(\pi/4) |\tau|}$.
For $\tau$ very small (e.g., $|\tau|<2$), it can make sense to use the trivial
bound
\begin{equation}\label{eq:smetana}
|F_\delta(s)| \leq F_0(\sigma)= \int_0^\infty e^{-t^2/2} t^\sigma \frac{dt}{t}
= 2^{\sigma/2-1} \Gamma(\sigma/2) \leq \frac{2^{\sigma/2}}{\sigma} 
\end{equation}
for $\sigma \in (0,1\rbrack$.
Alternatively, we could use integration by parts (much as in (\ref{eq:jokon})), 
followed by the trivial bound:
\begin{equation}\label{eq:libuse}
F_\delta(s) = - \int_0^\infty \left(e^{-u^2/2} e(\delta u)\right)' \frac{u^s}{s}
du =  \frac{F_\delta(s+2)}{s} - \frac{2 \pi \delta i}{s} F_{\delta}(s+1),
\end{equation}
and so
\begin{equation}\label{eq:lidic}\begin{aligned}
|F_\delta(s)|&\leq \frac{2^{\frac{\sigma+2}{2}-1} \Gamma\left(\frac{\sigma+2}{2}\right)
+ 2^{\frac{\sigma+1}{2}-1} |2 \pi \delta| 
\Gamma\left(\frac{\sigma+1}{2}\right)}{|s|}\leq \sqrt{\frac{\pi}{2}} \cdot \frac{1 + 2 \pi |\delta|}{|s|}
\end{aligned}\end{equation}
for $0\leq \sigma\leq 1$, since $2^x \Gamma(x) \leq \sqrt{2 \pi}$ for
$x\in \lbrack 1/2,3/2\rbrack$.

In the proof of Corollary \ref{cor:amanita1}, 
we used a somewhat crude approximation to
the function $E(\rho)$ defined in (\ref{eq:cormo}). 
It is worthwhile to give some approximations to $E(\rho)$
that, while still simple, are a little better.

\begin{lem}\label{lem:kimil}
Let $E(\rho)$ and $\upsilon(\rho)$ be as in (\ref{eq:cormo}). Then
\begin{equation}\label{eq:rhosm}
E(\rho) \geq \frac{1}{8} \rho - \frac{5}{384} \rho^3\end{equation}
for all $\rho>0$. 
We can also write
\begin{equation}\label{eq:rhola}
E(\rho) = \frac{\pi}{4} - \frac{\beta}{2} - \frac{\sin 2 \beta}{4 (1 + \sin \beta)},
\end{equation}
where $\beta = \arcsin 1/\upsilon(\rho)$.
\end{lem}
Clearly, (\ref{eq:rhosm}) is useful for $\rho$ small, whereas (\ref{eq:rhola})
is useful for $\rho$ large (since then $\beta$ is close to $0$). 
Taking derivatives, we see that (\ref{eq:rhola}) implies that
$E(\rho)$ is decreasing on $\beta$; thus, $E(\rho)$ is increasing on 
$\rho$.
Note that
(\ref{eq:rhosm}) gives us that
\begin{equation}\label{eq:hokor}
E\left(\frac{|\tau|}{(\pi \delta)^2}\right) \cdot |\tau| \geq 
\frac{1}{2} \left(\frac{\tau}{2 \pi \delta}\right)^2 \cdot 
\left(1 - \frac{5}{48 \pi^4} \left(\frac{\tau}{|\delta|^2}\right)^2\right).
\end{equation}
\begin{proof}
Let $\alpha = \arccos 1/\upsilon(\rho)$. Then $\upsilon(\rho) = 1/(\cos \alpha)$,
whereas
\begin{equation}\label{eq:hansli}\begin{aligned}
\sqrt{1+ \rho^2} &= 2 \upsilon^2(\rho) - 1 = \frac{2}{\cos^2 \alpha} -  1,\\
\rho &= \sqrt{\left(\frac{2}{\cos^2 \alpha} - 1\right)^2 - 1} = 
\sqrt{\frac{4}{\cos^4 \alpha} - \frac{4}{\cos^2 \alpha}}\\
&= \frac{2 \sqrt{1 - \cos^2 \alpha}}{\cos^2 \alpha} = 
\frac{2 \sin \alpha}{\cos^2 \alpha}.
\end{aligned}\end{equation}

Thus
\begin{equation}\label{eq:canuto}
\begin{aligned}2 E(\rho) &= \alpha - \frac{2 \left(\frac{1}{\cos \alpha} -
1\right)}{\frac{2 \sin \alpha}{\cos^2 \alpha}}
= \alpha - \frac{(1 - \cos \alpha) \cos \alpha}{\sin \alpha} =
\alpha - \frac{(1 - \cos^2 \alpha) \cos \alpha}{\sin \alpha (1+\cos \alpha)}\\
&= \alpha - \frac{\sin \alpha \cos \alpha}{1 + \cos \alpha} =
\alpha - \frac{\sin 2 \alpha}{4 \cos^2 \frac{\alpha}{2}}.
\end{aligned}\end{equation}
By (\ref{eq:antiman}) and (\ref{eq:hansli}), this implies that
\[2 E(\rho) \geq \frac{\rho}{4} - \frac{5 \rho^3}{24\cdot 8},\]
giving us (\ref{eq:rhosm}).

To obtain (\ref{eq:rhola}), simply define $\beta = \pi/2-\alpha$; the desired
inequality follows from the last two steps of (\ref{eq:canuto}).
\end{proof}


\section{Explicit formulas}

An 
{\em explicit formula} is an expression restating a sum such as
$S_{\eta,\chi}(\delta/x,x)$ as a sum of the Mellin transform
$G_\delta(s)$ over the zeros of the $L$ function $L(s,\chi)$. 
More specifically, for us, $G_\delta(s)$ is the Mellin transform
of $\eta(t) e(\delta t)$ for some smoothing function $\eta$ and some
$\delta\in \mathbb{R}$. We want a formula whose error terms are good
both for $\delta$ very close or equal to $0$ and for $\delta$ farther away
from $0$. (Indeed, our choice(s) of $\eta$ will be made so that
$F_\delta(s)$ decays rapidly in both cases.)

We will be able to base all of our work on a single, general explicit formula,
namely, Lemma \ref{lem:agamon}. This explicit formula has simple error terms
given purely in terms of a few norms of the given smoothing function $\eta$.
We also give a common framework for
estimating the contribution of zeros on the critical strip (Lemmas
\ref{lem:garmola} and \ref{lem:hausierer}).

The first example we work out is that of the Gaussian smoothing
$\eta(t) = e^{-t^2/2}$. We actually do this in part for didactic purposes
and in part because of its likely applicability elsewhere; for
our applications, we will always use smoothing functions based on 
$t e^{-t^2/2}$ and $t^2 e^{-t^2/2}$, 
generally in combination with something else. Since
$\eta(t) = e^{-t^2/2}$ does not vanish at $t=0$, its Mellin transform
has a pole at $s=0$ -- something that requires some additional work (Lemma
\ref{lem:povoso}; see also the proof of Lemma \ref{lem:agamon}).

Other than that, for each function $\eta(t)$, all that has to be done
is to bound an integral (from Lemma \ref{lem:garmola}) and bound a few
norms. Still, both for $\eta_*$ and for $\eta_+$, we find a few interesting
complications. Since $\eta_+$ is defined in terms of a truncation of
a Mellin transform (or, alternatively, in terms of a multiplicative
convolution with a Dirichlet kernel, as in (\ref{eq:patra2}) and
(\ref{eq:dirich2})), bounding the norms of $\eta_+$ and $\eta_+'$ takes
a little work. We leave this to Appendix \ref{app:norsmo}. The effect
of the convolution is then just to delay the decay a shift, in that
a rapidly decaying function $f(\tau)$ will get replaced by
$f(\tau-H)$, $H$ a constant. 

The smoothing function
$\eta_*$ is defined as a multiplicative convolution of $t^2 e^{-t^2/2}$
with something else. Given that we have an explicit formula for
$t^2 e^{-t^2/2}$, we obtain an explicit formula for $\eta_*$ by what amounts
to just exchanging the order of a sum and an integral; this is an idea
valid in general (see (\ref{eq:chemdames})).


\subsection{A general explicit formula}\label{subs:genexpf}

We will prove an explicit formula valid whenever the smoothing $\eta$
and its derivative $\eta'$ satisfy rather mild assumptions -- they will
be assumed to be $L_2$-integrable and to have strips of definition containing
$\{s: 1/2\leq \Re(s)\leq 3/2\}$, though any strip of the form
$\{s: \epsilon\leq \Re(s)\leq 1 + \epsilon\}$ would do just as well.

(For explicit formulas with different sets of assumptions, see, e.g.,
\cite[\S 5.5]{MR2061214} and \cite[Ch. 12]{MR2378655}.)

The main idea in deriving any explicit formula is to start with an expression
giving a sum as integral over a vertical line with an integrand involving
a Mellin transform (here, $G_\delta(s)$)
and an $L$-function (here, $L(s,\chi)$). We then
shift the line of integration to the left. If stronger assumptions
were made (as in Exercise 5 in \cite[\S 5.5]{MR2061214}), we could shift
the integral all the way to $\Re(s) = -\infty$; the integral would then 
disappear, replaced entirely by a sum over zeros (or even, as in the same
Exercise 5, by a particularly simple integral). Another possibility is to 
shift the line
only to $\Re(s) = 1/2+\epsilon$ for some $\epsilon>0$ -- but 
this gives a weaker result, and at any rate the factor
$L'(s,\chi)/L(s,\chi)$ can be large and messy to estimate 
within the critical strip $0<\Re(s)<1$.

Instead, we will shift the line to $\Re s = -1/2$. We can do this because
the assumptions on $\eta$ and $\eta'$ are enough to continue $G_\delta(s)$
analytically up to there (with a possible pole at $s=0$). 
The factor $L'(s,\chi)/L(s,\chi)$ is easy to estimate
for $\Re s < 0$ and $s=0$ (by the functional equation), and the part of
the integral on $\Re s = -1/2$ coming from $G_\delta(s)$ can be estimated 
easily using the fact that the Mellin transform is an isometry.

\begin{lem}\label{lem:agamon}
 Let $\eta:\mathbb{R}_0^+\to \mathbb{R}$ be in $C^1$.
Let $x\in \mathbb{R}^+$, $\delta \in \mathbb{R}$. 
Let $\chi$ be a primitive character mod $q$, $q\geq 1$. 

Write $G_\delta(s)$ for the Mellin transform of $\eta(t) e(\delta t)$.
Assume that $\eta(t)$ and $\eta'(t)$ are in $\ell_2$ (with respect
to the measure $dt$) and that $\eta(t) t^{\sigma-1}$ and $\eta'(t) t^{\sigma-1}$
are in $\ell_1$ (again with respect to $dt$) for all $\sigma$ in an open 
interval containing $\lbrack 1/2,3/2 \rbrack$.

Then
\begin{equation}\label{eq:marmar}\begin{aligned}
\sum_{n=1}^\infty &\Lambda(n) \chi(n) e\left(\frac{\delta}{x} n\right) \eta(n/x) =
I_{q=1} \cdot \widehat{\eta}(-\delta) x - \sum_\rho G_\delta(\rho) x^\rho \\ &- R +
O^*\left(
(\log q + 6.01) \cdot
(|\eta'|_2 + 2 \pi |\delta| |\eta|_2 )\right) x^{-1/2}
,\end{aligned}\end{equation}
where \begin{equation}\label{eq:estromo}\begin{aligned}
I_{q=1} &= \begin{cases} 1 & \text{if $q=1$,} \\ 0
&\text{if $q\ne 1$,}\end{cases}\\
R &=  \eta(0) \left(\log \frac{2 \pi}{q} + \gamma -
\frac{L'(1,\chi)}{L(1,\chi)}\right) + O^*(c_0)\end{aligned}\end{equation}
for $q>1$, $R= \eta(0) \log 2\pi$ for $q=1$ and
\begin{equation}\label{eq:marenostrum}c_0= \frac{2}{3} O^*\left(
\left|\frac{\eta'(t)}{\sqrt{t}}\right|_1 + \left|\eta'(t) \sqrt{t}\right|_1 + 
2 \pi |\delta| \left(\left|\frac{\eta(t)}{\sqrt{t}}\right|_1 + |\eta(t) \sqrt{t}|_1\right)
\right).\end{equation}
The norms $|\eta|_2$, $|\eta'|_2$, $|\eta'(t)/\sqrt{t}|_1$, etc., 
are taken with respect to the usual
measure $dt$.
The sum $\sum_\rho$ is a sum over all non-trivial zeros $\rho$
of  $L(s,\chi)$.
\end{lem}
\begin{proof}
Since (a)
$\eta(t) t^{\sigma-1}$ is in $\ell_1$ for $\sigma$ in an open interval
containing $3/2$ and (b) $\eta(t) e(\delta t)$ 
has bounded variation (since $\eta, \eta'\in \ell_1$, implying that the
derivative of $\eta(t) e(\delta t)$ is also in $\ell_1$),
 the Mellin inversion formula (as in, e.g., \cite[4.106]{MR2061214}) holds:
\[\eta(n/x) e(\delta n/x) 
= \frac{1}{2\pi i} \int_{\frac{3}{2} - i\infty}^{\frac{3}{2} + i \infty}
G_\delta(s) x^s n^{-s} ds.\]
Since $G_\delta(s)$ is bounded for $\Re(s)=3/2$ 
(by $\eta(t) t^{3/2-1} \in \ell_1$) and 
$\sum_n \Lambda(n) n^{-3/2}$ is bounded as well, 
we can change the order of summation and 
integration as follows:
\begin{equation}\label{eq:sartai}\begin{aligned}
\sum_{n=1}^\infty \Lambda(n) \chi(n) e(\delta n/x) \eta(n/x) 
&= \sum_{n=1}^\infty \Lambda(n) \chi(n) \cdot
\frac{1}{2\pi i} \int_{\frac{3}{2} - i\infty}^{\frac{3}{2} + i \infty}
G_\delta(s) x^s n^{-s} ds\\
&= \frac{1}{2 \pi i} \int_{\frac{3}{2} - i\infty}^{\frac{3}{2} + i \infty}
 \sum_{n=1}^\infty \Lambda(n) \chi(n)  G_\delta(s) x^s n^{-s} ds\\
&= \frac{1}{2 \pi i}
\int_{\frac{3}{2} - i \infty}^{\frac{3}{2} + i \infty} - \frac{L'(s,\chi)}{L(s,\chi)}
G_\delta(s) x^s ds.\end{aligned}\end{equation}
(This is the way the procedure always starts: see, for instance,
 \cite[Lemma 1]{MR1555183} or, to look at a recent standard reference, 
\cite[p. 144]{MR2378655}. We are being very scrupulous about integration 
 because we are working with general $\eta$.)

The first question we should ask ourselves is: up to where can we extend
$G_\delta(s)$? Since $\eta(t) t^{\sigma-1}$ is in $\ell_1$ for $\sigma$ in an
open interval $I$ containing $\lbrack 1/2,3/2\rbrack$, the transform
$G_\delta(s)$ is defined for $\Re(s)$ in the same interval $I$. However,
we also know that the transformation rule $M(t f'(t))(s) = -s\cdot M f(s)$ (see
(\ref{eq:harva}); by integration by parts) is valid when $s$ is in
the holomorphy
strip for both $M(t f'(t))$ and $M f$. In our case 
($f(t) = \eta(t) e(\delta t)$), this happens when
$\Re(s) \in (I-1) \cap I$ (so that both sides of the equation in the
 rule are defined). Hence $s\cdot G_\delta(s)$ (which equals $s \cdot M f(s)$) 
can be analytically continued to $\Re(s)$ in $(I-1)\cup I$, 
which is an open interval containing $\lbrack -1/2,3/2\rbrack$.
This implies immediately that $G_\delta(s)$ can be analytically continued to
the same region, with a possible pole at $s=0$.

When does $G_\delta(s)$ have a pole at $s=0$? This happens when $s G_\delta(s)$
is non-zero at $s=0$, i.e., when $M(t f'(t))(0)\ne 0$ for $f(t) = \eta(t) e(\delta t)$. Now
\[M(t f'(t))(0) = \int_0^\infty f'(t) dt = \lim_{t\to \infty} f(t) - f(0).\]
We already know that $f'(t) = (d/dt) (\eta(t) e(\delta t))$ is in $\ell_1$.
Hence, $\lim_{t\to \infty} f(t)$ exists, and must be $0$ because 
$f$ is in $\ell_1$. Hence $- M(t f'(t))(0) = f(0) = \eta(0)$.

Let us look at the next term in the Laurent expansion of $G_\delta(s)$ at
$s=0$. It is
\[\begin{aligned}
\lim_{s\to 0} \frac{s G_\delta(s) - \eta(0)}{s} &= 
\lim_{s\to 0} \frac{- M(t f'(t))(s) - f(0)}{s} = - \lim_{s\to 0}
\frac{1}{s} \int_0^{\infty} f'(t) (t^s-1) dt
\\ &= - \int_0^\infty f'(t) \lim_{s\to 0} \frac{t^s-1}{s} dt = 
-\int_0^\infty f'(t) \log t\; dt.\end{aligned}\]
Here we were able to exchange the limit and the integral because
$f'(t) t^\sigma$ is in $\ell_1$ for $\sigma$ in a neighborhood of $0$;
in turn, this is true because $f'(t) = \eta'(t) + 2\pi i \delta \eta(t)$
and $\eta'(t) t^\sigma$ and $\eta(t) t^\sigma$ are both in $\ell_1$ for
$\sigma$ in a neighborhood of $0$. In fact, we will use the easy bounds
$|\eta(t) \log t|\leq (2/3) (|\eta(t) t^{-1/2}|_1 + |\eta(t) t^{1/2}|_1)$,
$|\eta'(t) \log t|\leq (2/3) (|\eta'(t) t^{-1/2}|_1 + |\eta'(t) t^{1/2}|_1)$,
resulting from the inequality
\begin{equation}\label{eq:hutterite}
\frac{2}{3} \left(t^{-\frac{1}{2}} + t^{\frac{1}{2}}\right) \leq |\log t|,
\end{equation}
valid for all $t>0$.

We conclude that the Laurent expansion of $G_\delta(s)$ at $s=0$ is
\begin{equation}\label{eq:estabba}
G_\delta(s) = \frac{\eta(0)}{s} + c_0 + c_1 s + \dotsc,
\end{equation}
where 
\[\begin{aligned}c_0 &= O^*(|f'(t) \log t|_1)\\ &= \frac{2}{3} O^*\left(
\left|\frac{\eta'(t)}{\sqrt{t}}\right|_1 + \left|\eta'(t) \sqrt{t}\right|_1 + 
2 \pi \delta \left(\left|\frac{\eta(t)}{\sqrt{t}}\right|_1 + |\eta(t) \sqrt{t}|_1\right)
\right).\end{aligned}\]

We shift the line of integration in (\ref{eq:sartai}) to $\Re(s)=-1/2$.
We obtain
\begin{equation}\label{eq:argeri}\begin{aligned}\frac{1}{2\pi i}
\int_{2 - i \infty}^{2 + i \infty} -\frac{L'(s,\chi)}{L(s,\chi)}
G_\delta(s) x^s ds 
&= I_{q=1} G_\delta(1) x - \sum_\rho G_\delta(\rho) x^\rho - R
\\ &- \frac{1}{2\pi i} \int_{-1/2-i\infty}^{-1/2+i\infty}
\frac{L'(s,\chi)}{L(s,\chi)}
G_\delta(s) x^s ds,\end{aligned}\end{equation}
where 
\[R = \Res_{s=0} \frac{L'(s,\chi)}{L(s,\chi)} G_\delta(s).\]
Of course,
\[G_\delta(1) = M(\eta(t) e(\delta t))(1) = \int_0^\infty \eta(t) e(\delta t) dt
= \widehat{\eta}(-\delta).\]

Let us work out the Laurent expansion of $L'(s,\chi)/L(s,\chi)$ at $s=0$.
By the functional equation
(as in, e.g., \cite[Thm. 4.15]{MR2061214}),
\begin{equation}\label{eq:funeq}\begin{aligned}
\frac{L'(s,\chi)}{L(s,\chi)} &= \log \frac{\pi}{q} 
- \frac{1}{2} \psi\left(\frac{s+\kappa}{2}\right) 
- \frac{1}{2} \psi\left(\frac{1-s+\kappa}{2}\right) 
- \frac{L'(1-s,\overline{\chi})}{L(1-s,\overline{\chi})},\end{aligned}
\end{equation} where $\psi(s) = \Gamma'(s)/\Gamma(s)$
and 
\[\kappa= \begin{cases} 0 &\text{if $\chi(-1)=1$}\\
1 &\text{if $\chi(-1)=-1$.}\end{cases}\]
By $\psi(1-x)-\psi(x) = \pi \cot \pi x$ (immediate from
$\Gamma(s) \Gamma(1-s) = \pi/\sin \pi s$) and $\psi(s)+\psi(s+1/2) =
2 (\psi(2s) - \log 2)$ (Legendre; \cite[(6.3.8)]{MR0167642}), 
\begin{equation}\label{eq:gorilo}- \frac{1}{2} \left( \psi \left(\frac{s+\kappa}{2}\right)
 + \psi\left(\frac{1-s+\kappa}{2}\right)\right) = - \psi(1-s) + \log 2 +
\frac{\pi}{2} \cot \frac{\pi (s+\kappa)}{2}.\end{equation}

Hence, unless $q=1$,
the Laurent expansion of $L'(s,\chi)/L(s,\chi)$ at $s=0$ is
\[
\frac{1-\kappa}{s} + \left(\log \frac{2 \pi}{q} - \psi(1) -
\frac{L'(1,\chi)}{L(1,\chi)}\right) 
+ \frac{a_1}{s} + \frac{a_2}{s^2} + \dotsc.
\]
Here $\psi(1) = -\gamma$, the Euler gamma constant \cite[(6.3.2)]{MR0167642}.

There is a special case for $q=1$ due to the pole of $\zeta(s)$ at
$s=1$. We know that $\zeta'(0)/\zeta(0) = \log 2\pi$ 
(see, e.g., \cite[p. 331]{MR2378655}).

From this and (\ref{eq:estabba}), we conclude that, if $\eta(0)=0$, then
\[R = \begin{cases} c_0 &\text{if $q>1$ and $\chi(-1)=1$,}\\
0 & \text{otherwise,}\end{cases}\]
where $c_0 = O^*(|\eta'(t) \log t|_1 +
2 \pi |\delta| |\eta(t) \log t|_1)$. If $\eta(0)\ne 0$, then
\[R = \eta(0) \left(\log \frac{2 \pi}{q} + \gamma -
\frac{L'(1,\chi)}{L(1,\chi)}\right) 
+ \begin{cases} c_0 &\text{if $\chi(-1)=1$}\\
0 & \text{otherwise.}\end{cases}\]
for $q>1$, and
\[R = \eta(0) \log 2 \pi\]
for $q=1$.

It is time to estimate the integral on the right side of (\ref{eq:argeri}).
For that, we will need to estimate $L'(s,\chi)/L(s,\chi)$ for $\Re(s)=-1/2$
using (\ref{eq:funeq}) and (\ref{eq:gorilo}).

If $\Re(z)=3/2$, then $|t^2+z^2|\geq 9/4$ for all real $t$.
Hence, by \cite[(5.9.15)]{MR2723248} and \cite[(3.411.1)]{MR1773820},
\begin{equation}\label{eq:malgach}\begin{aligned}
\psi(z) &= \log z - \frac{1}{2z} - 2 \int_0^\infty
\frac{t dt}{(t^2+ z^2) (e^{2\pi t}-1)} \\ &=
\log z - \frac{1}{2z} + 2\cdot O^*\left(\int_0^\infty
\frac{t dt}{\frac{9}{4} (e^{2\pi t}-1)}\right)\\
&=  \log z - \frac{1}{2z} +  \frac{8}{9} 
O^*\left(\int_0^\infty \frac{t dt}{e^{2\pi t} -1}\right) \\ &= 
 \log z - \frac{1}{2z} +  \frac{8}{9} \cdot
O^*\left(\frac{1}{(2 \pi)^2} \Gamma(2) \zeta(2)\right) \\ &= 
 \log z - \frac{1}{2z} +  O^*\left(\frac{1}{27}\right) = 
\log z + O^*\left(\frac{10}{27}\right).
\end{aligned}\end{equation}
Thus, in particular, $\psi(1-s) = \log(3/2- i\tau) + O^*(10/27)$, where
we write $s = 1/2 + i \tau$. Now
\[\left|\cot \frac{\pi (s+\kappa)}{2}\right| = 
\left|\frac{e^{\mp \frac{\pi}{4} i - \frac{\pi}{2} \tau} + 
e^{\pm \frac{\pi}{4} i + \frac{\pi}{2} \tau}}{
e^{\mp \frac{\pi}{4} i - \frac{\pi}{2} \tau} - 
e^{\pm \frac{\pi}{4} i + \frac{\pi}{2} \tau}}\right| = 1.\]
Since $\Re(s)=-1/2$, a comparison
of Dirichlet series gives
\begin{equation}\label{eq:koloko}
\left|\frac{L'(1-s,\overline{\chi})}{L(1-s,\overline{\chi})}\right|
\leq \frac{|\zeta'(3/2)|}{|\zeta(3/2)|} \leq 1.50524,\end{equation}
where $\zeta'(3/2)$ and $\zeta(3/2)$ can be evaluated by Euler-Maclaurin.
Therefore, (\ref{eq:funeq}) and (\ref{eq:gorilo})
 give us that, for $s = -1/2 + i\tau$,
\begin{equation}\label{eq:peanuts}\begin{aligned}
\left|\frac{L'(s,\chi)}{L(s,\chi)}\right| &\leq
\left| \log \frac{q}{\pi}\right| + 
\log \left|\frac{3}{2} + i \tau\right| + \frac{10}{27} + \log 2 +
\frac{\pi}{2} +1.50524\\
&\leq \left| \log \frac{q}{\pi}\right| + 
\frac{1}{2} \log \left(\tau^2 + \frac{9}{4}\right)  
+ 4.1396.
\end{aligned}\end{equation}

Recall that we must bound the integral on the right side of
(\ref{eq:argeri}). The absolute value of the integral is at most $x^{-1/2}$
times
\begin{equation}\label{eq:lorel}
\frac{1}{2\pi} \int_{-\frac{1}{2}-i \infty}^{- \frac{1}{2} + i \infty} 
\left|\frac{L'(s,\chi)}{L(s,\chi)} G_\delta(s)\right| ds.\end{equation}
By Cauchy-Schwarz, this is at most
\[\sqrt{
\frac{1}{2\pi} 
\int_{-\frac{1}{2}- i \infty}^{-\frac{1}{2} + i \infty} 
\left|\frac{L'(s,\chi)}{L(s,\chi)}\cdot \frac{1}{s}\right|^2
|ds| } \cdot \sqrt{\frac{1}{2\pi} 
\int_{-\frac{1}{2}- i \infty}^{-\frac{1}{2} + i \infty} 
\left|G_\delta(s) s\right|^2 |ds|}\]
By (\ref{eq:peanuts}),
\[\begin{aligned}
\sqrt{\int_{-\frac{1}{2}- i \infty}^{-\frac{1}{2} + i \infty} 
\left|\frac{L'(s,\chi)}{L(s,\chi)}\cdot \frac{1}{s}\right|^2 |ds|}
&\leq  
\sqrt{\int_{-\frac{1}{2}- i \infty}^{-\frac{1}{2} + i \infty} 
\left|\frac{\log q}{s}\right|^2 |ds|} \\&+
\sqrt{\int_{-\infty}^{\infty} 
\frac{\left|
\frac{1}{2} \log\left(\tau^2 + \frac{9}{4}\right) +
    4.1396 + \log \pi\right|^2}{\frac{1}{4} + \tau^2} d\tau}\\
&\leq \sqrt{2\pi} \log q  + \sqrt{226.844},
\end{aligned}\]
where we compute the last integral numerically.\footnote{By a
rigorous integration from $\tau = -100000$ to $\tau = 100000$ 
using VNODE-LP \cite{VNODELP}, which runs
on the PROFIL/BIAS interval arithmetic package\cite{Profbis}.}

Again, we use the fact that,
by (\ref{eq:harva}), $s G_\delta(s)$ is the Mellin transform of
\begin{equation}\label{eq:jamon}
- t \frac{d (e(\delta t) \eta(t))}{dt} = 
- 2 \pi i \delta t e(\delta t) \eta(t) 
- t e(\delta t) \eta'(t)
\end{equation}


Hence, by Plancherel (as in (\ref{eq:victi})),
\begin{equation}\label{eq:chorizo}\begin{aligned}
\sqrt{\frac{1}{2\pi} 
\int_{-\frac{1}{2}- i \infty}^{-\frac{1}{2} + i \infty} 
\left|G_\delta(s) s\right|^2 |ds|} &= 
\sqrt{\int_0^\infty \left|
- 2 \pi i \delta t e(\delta t) \eta(t) 
- t e(\delta t) \eta'(t)
\right|^2 t^{-2} dt} \\ &=
2 \pi |\delta|
\sqrt{\int_0^\infty |\eta(t)|^2 dt} + \sqrt{\int_0^\infty |\eta'(t)|^2 dt}.
\end{aligned}\end{equation}
Thus, (\ref{eq:lorel}) is at most
\[\left(\log q + \sqrt{\frac{226.844}{2 \pi}}\right) \cdot
\left(|\eta'|_2 + 2\pi |\delta| |\eta|_2\right).\]
\end{proof}

Lemma \ref{lem:agamon} leaves us with three tasks: bounding
the sum of $G_\delta(\rho) x^{\rho}$ over all non-trivial zeroes $\rho$
with small
imaginary part, bounding the sum of $G_\delta(\rho) x^{\rho}$ over all
non-trivial zeroes $\rho$ with large imaginary part, and bounding
$L'(1,\chi)/L(1,\chi)$. Let us start with the last task: while, in a narrow
sense, it is optional -- in that, in most of our applications, we will have
$\eta(0)=0$, thus making the term
$L'(1,\chi)/L(1,\chi)$ disappear -- it is also very easy and can be
dealt with quickly.

 Since we will be using
a finite GRH check in all later applications, we might as well use it here.
\begin{lem}\label{lem:povoso}
Let $\chi$ be a primitive character mod $q$, $q>1$. 
Assume that all non-trivial zeroes $\rho= \sigma+i t$ of
$L(s,\chi)$ with $|t|\leq 5/8$ satisfy $\Re(\rho) = 1/2$.
Then
\[\left|\frac{L'(1,\chi)}{L(1,\chi)}\right|\leq 
\frac{5}{2} \log M(q) + c,\]
where $M(q) = \max_n \left|\sum_{m\leq n} \chi(m)\right|$ and
\[c = 5 \log \frac{2 \sqrt{3}}{\zeta(9/4)/\zeta(9/8)} = 
15.07016\dotsc.\]
\end{lem}
\begin{proof}
By a lemma of Landau's (see, e.g., \cite[Lemma 6.3]{MR2378655}, where
the constants are easily made explicit) based on the Borel-Carath\'eodory Lemma
(as in \cite[Lemma 6.2]{MR2378655}), any function $f$ analytic and
zero-free on a disc $C_{s_0,R}=\{s:|s-s_0|\leq R\}$ of radius $R>0$
around $s_0$ satisfies
\begin{equation}\label{eq:landau}
\frac{f'(s)}{f(s)} = O^*\left(\frac{2 R \log M/|f(s_0)|}{(R-r)^2}\right)
\end{equation}
for all $s$ with $|s-s_0|\leq r$, where $0<r<R$ and 
$M$ is the maximum of $|f(z)|$ on $C_{s_0,R}$. Assuming $L(s,\chi)$ has
no non-trivial zeros off the critical line with $|\Im(s)|\leq H$, 
where $H>1/2$, we
set $s_0 = 1/2+H$, $r = H-1/2$, and let $R\to H^-$.
We obtain
\begin{equation}\label{eq:rabamar}
\frac{L'(1,\chi)}{L(1,\chi)} = O^*\left(8 H \log \frac{\max_{s\in C_{s_0,H}} 
|L(s,\chi)|}{|L(s_0,\chi)|}\right).
\end{equation}
Now 
\[|L(s_0,\chi)|\geq \prod_p (1+p^{-s_0})^{-1} = \prod_p
\frac{(1-p^{-2 s_0})^{-1}}{(1-p^{- s_0})^{-1}} = \frac{\zeta(2 s_0)}{\zeta(s_0)}.\]
Since $s_0 = 1/2+H$, $C_{s_0,H}$ is contained
in $\{s\in \mathbb{C}: \Re(s)>1/2\}$ for any value of $H$.
We choose (somewhat arbitrarily) $H=5/8$. 

By partial summation, for $s=\sigma+it$ with $1/2\leq \sigma<1$ and any 
$N\in \mathbb{Z}^+$,
\begin{equation}\label{eq:sellyou}\begin{aligned}
L(s,\chi) &= \sum_{n\leq N} \chi(m) n^{-s} -
\left(\sum_{m\leq N} \chi(m)\right) (N+1)^{-s} \\ &+
 \sum_{n\geq N+1}
 \left(\sum_{m\leq n} \chi(m)\right) (n^{-s}-(n+1)^{-s+1})\\
&= O^*\left(\frac{N^{1-1/2}}{1-1/2} + N^{1-\sigma} +
M(q) N^{-\sigma}\right),\end{aligned}\end{equation}
where $M(q) = \max_n \left|\sum_{m\leq n} \chi(m)\right|$. 
We set $N = M(q)/3$, and obtain
\begin{equation}\label{eq:thecrisis}
|L(s,\chi)| \leq 2 M(q) N^{-1/2} = 2 \sqrt{3} \sqrt{M(q)}.\end{equation}
We put this into (\ref{eq:rabamar}) and are done.
\end{proof}
Let $M(q)$ be as in the statement of Lem.~\ref{lem:povoso}. Since the sum
of $\chi(n)$ ($\chi \mo q$, $q>1$)
over any interval of length $q$ is $0$, it is easy to see that $M(q)\leq
q/2$. We also have the following explicit version of the P\'olya-Vinogradov
inequality:
\begin{equation}\label{eq:karbach}M(q) \leq
\begin{cases} \frac{2}{\pi^2} \sqrt{q} \log q + \frac{4}{\pi^2} \sqrt{q} 
\log \log q + \frac{3}{2} \sqrt{q} & \text{if $\chi(-1)=1$,}\\
\frac{1}{2\pi} \sqrt{q} \log q + \frac{1}{\pi} \sqrt{q} \log \log q + 
\sqrt{q} & \text{if $\chi(-1)=1$.}\end{cases}\end{equation}
Taken together with $M(q)\leq q/2$, this implies that
\begin{equation}\label{eq:wachetauf}M(q)\leq q^{4/5}\end{equation}
for all $q\geq 1$, and also that
\begin{equation}\label{eq:marlo}M(q)\leq 2 q^{3/5}\end{equation}
for all $q\geq 1$.

Notice, lastly, that
\[\left|\log \frac{2\pi}{q} + \gamma\right|\leq 
\log q + \log \frac{e^\gamma \cdot 2 \pi}{3^2}\]
for all $q\geq 3$. (There are no primitive characters modulo $2$, so 
we can omit $q=2$.)

We conclude that, for $\chi$ primitive and non-trivial,
\[\begin{aligned}
\left|\log \frac{2\pi}{q} + \gamma - \frac{L'(1,\chi)}{L(1,\chi)}\right|
&\leq \log \frac{e^\gamma \cdot 2 \pi}{3^2} + \log q + \frac{5}{2} \log q^{\frac{4}{5}} + 15.07017\\
&\leq 3 \log q + 15.289.
\end{aligned}\]
Obviously, $15.289$ is more than $\log 2 \pi$, the bound for $\chi$ trivial.
Hence,
 the absolute value of the quantity $R$ in the statement of Lemma \ref{lem:agamon}
is at most
\begin{equation}\label{eq:bleucol}
|\eta(0)| (3 \log q + 15.289) + |c_0|\end{equation}
for all primitive $\chi$.

It now remains to bound the sum $\sum_{\rho} G_{\delta}(\rho) x^\rho$
in (\ref{eq:marmar}). Clearly
\[\left|\sum_{\rho} G_{\delta}(\rho) x^\rho\right| \leq
\sum_{\rho} \left|G_{\delta}(\rho)\right| \cdot x^{\Re(\rho)}.\]
Recall that these are sums over the non-trivial zeros $\rho$ of $L(s,\chi)$.

We first prove a general lemma on sums of values of functions on the non-trivial
zeros of $L(s,\chi)$. This is little more than partial summation, given
a (classical) bound for the number of zeroes $N(T,\chi)$ of $L(s,\chi)$
with $|\Im(s)|\leq T$. 
The error term becomes particularly simple if $f$ is real-valued and 
decreasing; the statement is then practically identical to that of
\cite[Lemma 1]{MR0202686} (for $\chi$ principal), except for the fact
 that the error term is improved here.
\begin{lem}\label{lem:garmola}
Let $f:\mathbb{R}^+\to \mathbb{C}$ be piecewise $C^1$. Assume
$\lim_{t\to \infty} f(t) t \log t = 0$. Let $\chi$ be a primitive character
$\mod q$, $q\geq 1$; let $\rho$ denote the non-trivial zeros $\rho$ of 
$L(s,\chi)$.
Then, for any $y\geq 1$,
\begin{equation}\label{eq:jotok}\begin{aligned}
\mathop{\sum_{\text{$\rho$ non-trivial}}}_{\Im(\rho) > y} f(\Im(\rho))
&= \frac{1}{2\pi} \int_y^\infty f(T) \log \frac{q T}{2 \pi} dT\\
&+ 
\frac{1}{2} O^*\left(|f(y)| g_\chi(y) + \int_{y}^\infty \left|f'(T)\right| 
\cdot g_\chi(T)  dT\right),
\end{aligned}\end{equation}
where 
\begin{equation}\label{eq:ertr}
g_\chi(T) =
0.5 \log qT + 17.7\end{equation}

If $f$ is real-valued and decreasing on $\lbrack y,\infty)$, the second line of (\ref{eq:jotok})
equals
\[O^*\left(\frac{1}{4} \int_y^\infty \frac{f(T)}{T} dT\right).\]
\end{lem}
\begin{proof}
Write $N(T,\chi)$ for the number of non-trivial zeros of $L(s,\chi)$
with $|\Im(s)|\leq T$. 
Write 
$N^+(T,\chi)$ for the number of (necessarily non-trivial) zeros of $L(s,\chi)$
with $0 < \Im(s)\leq T$.
Then, for any $f:\mathbb{R}^+\to \mathbb{C}$ with
$f$ piecewise differentiable and $\lim_{t\to \infty} f(t) N(T,\chi) = 0$,
\[\begin{aligned}
\sum_{\rho: \Im(\rho)> y} f(\Im(\rho)) &= \int_{y}^\infty f(T)\; dN^+(T,\chi)
\\ &= - \int_{y}^\infty f'(T) (N^+(T,\chi) - N^+(y,\chi)) dT\\
&= - \frac{1}{2} \int_{y}^\infty f'(T) (N(T,\chi) - N(y,\chi)) dT
.\end{aligned}\]
Now, by \cite[Thms. 17--19]{MR0003018} and \cite[Thm. 2.1]{MR726004}
(see also \cite[Thm. 1]{Trudgian}),
\begin{equation}\label{eq:melos}
N(T,\chi) = \frac{T}{\pi} \log \frac{q T}{2\pi e} + O^*\left(g_{\chi}(T)\right)
\end{equation}
for $T\geq 1$, where $g_\chi(T)$ is as in (\ref{eq:ertr}). (This is a classical
formula; the references serve to prove the explicit form (\ref{eq:ertr}) 
for the error term $g_\chi(T)$.)

Thus, for $y\geq 1$,
\begin{equation}\label{eq:bochal}\begin{aligned}
\sum_{\rho: \Im(\rho)>y} f(\Im(\rho)) 
&= - \frac{1}{2}
\int_{y}^\infty f'(T) \left(\frac{T}{\pi} \log \frac{q T}{2\pi e} -
\frac{y}{\pi} \log \frac{q y}{2\pi e}\right) dT
\\ &+ \frac{1}{2} O^*\left(|f(y)| g_\chi(y) + \int_{y}^\infty \left|f'(T)\right| 
\cdot g_\chi(T) dT\right)
.\end{aligned}\end{equation}
Here
\begin{equation}\label{eq:linstrau}- \frac{1}{2}
\int_{y}^\infty f'(T) \left(\frac{T}{\pi} \log \frac{q T}{2\pi e}
- \frac{y}{\pi} \log \frac{q y}{2\pi e}\right) dT =
\frac{1}{2 \pi}
\int_{y}^\infty f(T) \log \frac{q T}{2\pi} dT.
\end{equation}
If $f$ is real-valued and decreasing (and so, by $\lim_{t\to \infty} f(t)=0$,
non-negative), 
 \[\begin{aligned}|f(y)| g_\chi(y)  + \int_{y}^\infty \left|f'(T)\right| 
\cdot g_\chi(T)  dT &= f(y) g_\chi(y) - 
\int_{y}^\infty f'(T) g_\chi(T) dT \\
&= 0.5 \int_{y}^\infty \frac{f(T)}{T} dT,
\end{aligned}\]
since $g_\chi'(T) \leq 0.5/T$ for all $T\geq T_0$.
\end{proof}

Now we can bound the sum $\sum_\rho G_\delta(\rho) x^\rho$. The bound we will
give is proportional to $\sqrt{T_0} \log q T_0$, whereas a more obvious approach
would give a bound proportional to $T_0 \log q T_0$. This (large) 
improvement is due to our usage of isometry (after an application of 
Cauchy-Schwarz) to bound integrals throughout. It is also this usage that
allows us to give a general bound depending only on a few norms of $\eta$
and its variants.
\begin{lem}\label{lem:hausierer}
Let $\eta:\mathbb{R}_0^+\to \mathbb{R}$ be such that
both $\eta(t)$ and $(\log t) \eta(t)$ lie in $L_1\cap L_2$ 
and $\eta(t)/\sqrt{t}$ lies in $L_1$ (with respect to $dt$). 
Let $\delta\in \mathbb{R}$. Let $G_\delta(s)$ be 
the Mellin transform of $\eta(t) e(\delta t)$. 

Let $\chi$ be a primitive character mod $q$, $q\geq 1$. Let $T_0\geq 1$. 
Assume that all non-trivial zeros $\rho$ of $L(s,\chi)$ with
$|\Im(\rho)|\leq T_0$ lie on the critical line.
Then
\[
\mathop{\sum_{\text{$\rho$ non-trivial}}}_{|\Im(\rho)|\leq T_0} \left|G_{\delta}(\rho)\right| 
\]
is at most
\begin{equation}\label{eq:arnbax}
\begin{aligned} (|\eta|_2 + |\eta\cdot \log|_2) &\sqrt{T_0} \log q T_0 + 
(17.21 |\eta\cdot \log|_2 - (\log 2\pi \sqrt{e}) |\eta|_2)
\sqrt{T_0}\\
&+ \left|\eta(t)/\sqrt{t}\right|_1\cdot (1.32 \log q + 34.5)
\end{aligned}\end{equation}
\end{lem}
\begin{proof}
For $s = 1/2 + i \tau$, we have the trivial bound
\begin{equation}\label{eq:sandunga}
|G_\delta(s)| \leq \int_0^\infty |\eta(t)| t^{1/2} \frac{dt}{t} =
\left|\eta(t)/\sqrt{t}\right|_1
,\end{equation}
where $F_\delta$ is as in (\ref{eq:guason}). We also have the trivial bound
\begin{equation}\label{eq:candonga}
|G_\delta'(s)| = \left|\int_0^\infty (\log t) \eta(t) t^s \frac{dt}{t}\right|
\leq \int_0^\infty |(\log t) \eta(t)| t^{\sigma} \frac{dt}{t} =
\left|(\log t) \eta(t) t^{\sigma-1}\right|_1
\end{equation}
for $s = \sigma+ i \tau$.

Let us start by bounding the contribution of very low-lying zeros
($|\Im(\rho)|\leq 1$). 
By (\ref{eq:melos}) and (\ref{eq:ertr}),
\[N(1,\chi) = \frac{1}{\pi} \log \frac{q}{2 \pi e} + O^*\left(
0.5 \log q + 17.7\right) = O^*(0.819 \log q + 16.8).\]
Therefore,
\[\mathop{\sum_{\text{$\rho$ non-trivial}}}_{|\Im(\rho)|\leq 1} 
\left|G_{\delta}(\rho)\right|
\leq \left|\eta(t) t^{- 1/2}\right|_1\cdot
(0.819 \log q + 16.8).\]

Let us now consider zeros $\rho$ with $|\Im(\rho)|>1$.
Apply Lemma \ref{lem:garmola} with $y=1$ and
\[f(t) = \begin{cases}
\left|G_\delta(1/2 + i t)\right| &\text{if $t\leq T_0$},\\
0 &\text{if $t> T_0$}.\end{cases}\]
This gives us that
\begin{equation}\label{eq:datora}\begin{aligned}
\sum_{\rho: 1<|\Im(\rho)|\leq T_0} f(\Im(\rho)) &= \frac{1}{\pi}
\int_1^{T_0} f(T) \log \frac{q T}{2 \pi} dT 
\\ &+ 
O^*\left(|f(1)| g_\chi(1) + \int_{1}^\infty |f'(T)|\cdot g_\chi(T) \; dT\right),
\end{aligned}\end{equation}
where we are using the fact that $f(\sigma + i\tau) = f(\sigma-i\tau)$
(because $\eta$ is real-valued).
By Cauchy-Schwarz,
\[
\frac{1}{\pi}
\int_1^{T_0} f(T) \log \frac{q T}{2 \pi} dT \leq
\sqrt{\frac{1}{\pi} \int_1^{T_0} |f(T)|^2 dT} \cdot 
 \sqrt{\frac{1}{\pi} \int_1^{T_0} \left(\log \frac{q T}{2 \pi}\right)^2
   dT}.\]
Now
\[\begin{aligned}
 \frac{1}{\pi} \int_1^{T_0} |f(T)|^2 dT &\leq 
\frac{1}{2 \pi } \int_{-\infty}^{\infty} \left|G_\delta\left(\frac{1}{2}+iT\right)\right|^2 dT
\leq \int_0^\infty |e(\delta t) \eta(t)|^2 dt = |\eta|_2^2
\end{aligned}\]
by Plancherel (as in (\ref{eq:victi})).
 We also have
\[\int_1^{T_0} \left(\log \frac{q T}{2 \pi}\right)^2 dT \leq
\frac{2 \pi}{q} \int_0^{\frac{q T_0}{2\pi}} (\log t)^2 dt
\leq \left(\left(\log \frac{q T_0}{2 \pi e}\right)^2 + 1 \right) \cdot
T_0.\]
Hence \[\frac{1}{\pi}
\int_1^{T_0} f(T) \log \frac{q T}{2 \pi} dT \leq
\sqrt{\left(\log \frac{q T_0}{2 \pi e}\right)^2 + 1} \cdot
|\eta|_2 \sqrt{T_0}.\]

Again by Cauchy-Schwarz,
\[\begin{aligned}
\int_{1}^\infty |f'(T)|&\cdot g_\chi(T) \; dT \leq
\sqrt{\frac{1}{2\pi} \int_{-\infty}^\infty |f'(T)|^2 dT}\cdot 
\sqrt{\frac{1}{\pi} \int_1^{T_0} |g_{\chi}(T)|^2
  dT}.\end{aligned}\]
Since $|f'(T)| = |G_\delta'(1/2+i T)|$ and $(M\eta)'(s)$ is
the Mellin transform of $\log(t)\cdot e(\delta t) \eta(t)$
(by (\ref{eq:harva})),
\[\frac{1}{2\pi} \int_{-\infty}^\infty |f'(T)|^2 dT =
|\eta(t) \log(t)|_2.\] 
Much as before,
\[\begin{aligned}
\int_1^{T_0} |g_{\chi}(T)|^2 dT &\leq \int_0^{T_0} (0.5 \log q T +17.7)^2
dT \\&= (0.25 (\log q T_0)^2 + 17.2 (\log q T_0) + 296.09) T_0.\end{aligned}\]

Summing, we obtain 
\[\begin{aligned}\frac{1}{\pi}
\int_1^{T_0} &f(T) \log \frac{q T}{2 \pi} dT  + 
\int_{1}^\infty |f'(T)|\cdot g_\chi(T) \; dT\\
&\leq \left(
\left(\log \frac{q T_0}{2 \pi e} + \frac{1}{2}\right) |\eta|_2+
 \left(\frac{\log q T_0}{2} + 17.21\right) |\eta(t) (\log t)|_2\right)
\sqrt{T_0}
\end{aligned}\]

Finally, by (\ref{eq:sandunga}) and (\ref{eq:ertr}),
\[|f(1)| g_\chi(1) \leq 
\left|\eta(t)/\sqrt{t}\right|_1\cdot (0.5 \log q + 17.7).\]
By (\ref{eq:datora}) and the assumption that all non-trivial zeros with
$|\Im(\rho)|\leq T_0$ lie
on the line $\Re(s)=1/2$, we conclude that
\[\begin{aligned}\mathop{\sum_{\text{$\rho$ non-trivial}}}_{1 < |\Im(\rho)|\leq T_0} 
\left|G_{\delta}(\rho)\right| &\leq 
(|\eta|_2 + |\eta\cdot \log|_2) \sqrt{T_0} \log q T_0 \\ &+ 
(17.21 |\eta\cdot \log|_2 - (\log 2\pi \sqrt{e}) |\eta|_2)
\sqrt{T_0}\\
&+ \left|\eta(t)/\sqrt{t}\right|_1\cdot (0.5 \log q + 17.7)
.\end{aligned}\] 
\end{proof}

\subsection{Sums and decay for the Gaussian: $\eta_\heartsuit(t) = e^{-t^2/2}$}

It is now time to derive our bounds for the Gaussian smoothing.
Thanks to the general work we have done so far, there is really only one thing
left to do, namely, an estimate for the sum
$\sum_\rho |F_\delta(\rho)|$ over non-trivial zeros with $|\Im(\rho)|>T_0$.

\begin{lem}\label{lem:garmonas}
Let $\eta_\heartsuit(t) = e^{-t^2/2}$. Let $x\in \mathbb{R}^+$, $\delta\in \mathbb{R}$.
Let $\chi$ be a primitive character mod $q$, $q\geq 1$. 
Assume that all non-trivial zeros $\rho$ of $L(s,\chi)$
with $|\Im(\rho)|\leq T_0$
satisfy $\Re(s)=1/2$. Assume that $T_0\geq \max(4\pi^2 |\delta|,100)$.

Write $F_\delta(s)$ for the Mellin transform of $\eta(t) e(\delta t)$. Then
\[\begin{aligned}
\mathop{\sum_{\text{$\rho$ non-trivial}}}_{|\Im(\rho)|>T_0} 
\left|F_{\delta}(\rho)\right| \leq 
\log \frac{q T_0}{2 \pi} \cdot
\left(4.329 e^{-0.1598 T_0} + 
0.802 |\delta| e^{-0.1065 \left(\frac{T_0}{\pi |\delta|}\right)^2}\right).
\end{aligned}\]
\end{lem}
Here we have preferred to give a bound with a simple form.
It is probably feasible to derive
from Theorem \ref{thm:princo}
 a bound
essentially proportional to $e^{-E(\rho) T_0}$, where 
$\rho = T_0/(\pi \delta)^2$ and
$E(\rho)$ is as in (\ref{eq:cormo}). (As we discussed in 
\S \ref{subs:conclusion}, $E(\rho)$ behaves as
$e^{-(\pi/4) T_0}$ for $\rho$ large and as $e^{-0.125 (T_0/(\pi
  \delta))^2}$ for $\rho$ small.)
\begin{proof}
First of all,
\[\mathop{\sum_{\text{$\rho$ non-trivial}}}_{|\Im(\rho)|>T_0} 
\left|F_{\delta}(\rho)\right| = 
\mathop{\sum_{\text{$\rho$ non-trivial}}}_{\Im(\rho)>T_0} 
\left(\left|F_{\delta}(\rho)\right| + 
\left|F_{\delta}(1-\rho)\right|\right)
,\]
by the functional equation (which implies that non-trivial
zeros come in pairs $\rho$, $1-\rho$). Hence, by Cor.~\ref{cor:amanita1},
\begin{equation}\label{eq:gargara}
\mathop{\sum_{\text{$\rho$ non-trivial}}}_{|\Im(\rho)|>T_0} 
\left|F_{\delta}(\rho)\right| \leq
\mathop{\sum_{\text{$\rho$ non-trivial}}}_{\Im(\rho)>T_0}
f(\Im(\rho)),\end{equation}
where
\begin{equation}\label{eq:darmo}
f(\tau) = 
4.226 \cdot
\begin{cases}
e^{-0.1065 \left(\frac{\tau}{\pi \delta}\right)^2} &
\text{if $|\tau| < \frac{3}{2} (\pi \delta)^2$,}\\ 
e^{-0.1598 |\tau|} & \text{if $|\tau| \geq \frac{3}{2} (\pi \delta)^2$.}
\end{cases}
\end{equation}
It is easy to check that $f(\tau)$ 
is a decreasing function of $\tau$ for $\tau\geq T_0$.

We now apply Lemma \ref{lem:garmola}. We obtain that
\begin{equation}\label{eq:happin}
\mathop{\sum_{\text{$\rho$ non-trivial}}}_{\Im(\rho)>T_0} f(\Im(\rho))
\leq \int_{T_0}^\infty f(T) \left(\frac{1}{2\pi}  \log \frac{q T}{2\pi}
+ \frac{1}{4 T}\right) dT.\end{equation}

We just need to estimate some integrals. For any $y\geq 1$, $c,c_1>0$,
\[\begin{aligned}
\int_y^\infty \left(\log t + \frac{c_1}{t}\right) e^{-c t} dt
&\leq \int_y^\infty \left(\log t - \frac{1}{c t}\right) e^{-c t} dt +
\left(\frac{1}{c}+c_1\right) 
\int_y^\infty \frac{e^{-c t}}{t} dt\\
&= \frac{(\log y) e^{-c y}}{c} +   
\left(\frac{1}{c}+c_1\right) E_1(cy),
\end{aligned}\]
where $E_1(x) = \int_x^\infty e^{-t} dt/t$. 
Clearly, $E_1(x)\leq \int_x^\infty e^{-t} dt/x = e^{-x}/x$. Hence
\[\int_y^\infty \left(\log t + \frac{c_1}{t}\right) e^{-c t} dt
\leq \left(\log y +  \left( \frac{1}{c} + c_1\right) \frac{1}{y}
 \right) \frac{e^{-c y}}{c}.
\] We conclude that
\[\begin{aligned}
\int_{T_0}^\infty e^{-0.1598 t} &\left(\frac{1}{2\pi} \log \frac{q t}{2\pi} +
\frac{1}{4t}\right) dt\\ &\leq \frac{1}{2\pi} \int_{T_0}^\infty 
\left(\log t + \frac{\pi/2}{t}\right) e^{-ct} dt + \frac{\log \frac{q}{2\pi}}{2\pi c} 
\int_{T_0}^\infty e^{-c t} dt\\
&= \frac{1}{2\pi c} \left(\log T_0 + \log \frac{q}{2\pi} +
\left(\frac{1}{c} + \frac{\pi}{2}\right)
\frac{1}{T_0}\right) e^{-c T_0} 
\end{aligned}\]
with $c=0.1598$. Since $T_0\geq 100$, this is at most
\begin{equation}\label{eq:haroldo}1.0242 \log \frac{q T_0}{2\pi} e^{-c T_0}.
\end{equation}

Now let us deal with the Gaussian term. (It appears only if
$T_0 < (3/2) (\pi \delta)^2$, as otherwise 
$|\tau|\geq (3/2) (\pi \delta)^2$ holds whenever $|\tau|\geq T_0$.)
For any $y\geq e$, $c\geq 0$,
\begin{equation}\label{eq:analo1}\int_{y}^\infty e^{-c t^2} dt = \frac{1}{\sqrt{c}} \int_{\sqrt{c} y}^\infty
e^{-t^2} dt \leq \frac{1}{c y} \int_{\sqrt{c} y}^\infty t e^{-t^2} dt \leq
\frac{e^{-c y^2}}{2 c y},
\end{equation}
\begin{equation}\label{eq:analo2}\int_y^\infty  \frac{e^{-c t^2}}{t} dt =
\int_{c y^2}^\infty \frac{e^{-t}}{2 t} dt = 
 \frac{E_1(c y^2)}{2} \leq \frac{e^{-c y^2}}{2 c y^2},
\end{equation}
\begin{equation}\label{eq:analo3}
\int_{y}^\infty (\log t) e^{-c t^2} dt
\leq \int_y^\infty \left(\log t + \frac{\log t - 1}{2 c t^2}\right) e^{-c t^2}
dt = \frac{\log y}{2 c y} e^{-c y^2}.
\end{equation} Hence
\begin{equation}\label{eq:milstei}
\begin{aligned}\int_{T_0}^\infty &e^{- 0.1065 \left(\frac{T}{\pi \delta}\right)^2}
\left(\frac{1}{2\pi} \log \frac{q T}{2 \pi} + \frac{1}{4 T}\right) dT
\\&= \int_{\frac{T_0}{\pi |\delta|}}^\infty e^{-0.1065 t^2} 
\left(\frac{|\delta|}{2} \log \frac{q |\delta| t}{2} + \frac{1}{4 t}\right)  dt\\
&\leq \left(\frac{\frac{|\delta|}{2} \log \frac{T_0}{\pi |\delta|}}{2 c \frac{T_0}{\pi |\delta|}} 
+ \frac{\frac{|\delta|}{2} \log \frac{q |\delta|}{2}}{2 c \frac{T_0}{\pi |\delta|}}+ \frac{1}{8 c \left(\frac{T_0}{\pi |\delta|}\right)^2}\right) 
e^{-c \left(\frac{T_0}{\pi |\delta|}\right)^2} 
\end{aligned}\end{equation}
with $c=0.1065$. Since $T_0\geq 100$ and $q\geq 1$,
\[
\frac{2 \pi}{8 T_0} \leq 
\frac{\pi}{400} 
\leq 0.0057 \cdot \frac{1}{2} \log \frac{q T_0}{2 \pi} 
\] 
Thus, the last line of (\ref{eq:milstei}) is less than 
\begin{equation}\label{eq:certeru}
1.0057 \frac{\frac{|\delta|}{2} \log \frac{q T_0}{2 \pi}}{\frac{2 c T_0}{\pi |\delta|}}
e^{-c \left(\frac{T_0}{\pi |\delta|}\right)^2} = 
1.0057 \frac{\pi \delta^2}{4 c T_0} e^{-c \left(\frac{T_0}{\pi |\delta|}\right)^2}
.\end{equation}
Again by $T_0\geq 4\pi^2 |\delta|$, we see that
$1.0057 \pi |\delta|/(4 c T_0) \leq 1.0057/(16 c \pi) \leq 0.18787$. 

To obtain our final bound, we simply sum (\ref{eq:haroldo}) and
(\ref{eq:certeru}), and multiply by $4.266$ (the constant in (\ref{eq:darmo})).
We conclude that the integral in (\ref{eq:happin}) is at most
\[\left(4.329 e^{-0.1598 T_0} + 
0.8015 |\delta| e^{-0.1065 \left(\frac{T_0}{\pi |\delta|}\right)^2}\right)
\log \frac{q T_0}{2 \pi}.
\]
\end{proof}

We need to record a few norms related to the Gaussian 
$\eta_\heartsuit(t)=e^{-t^2/2}$ before we proceed. Recall we are working
with the one-sided Gaussian, i.e., we set $\eta_\heartsuit(t)=0$ for $t<0$.
Symbolic integration then gives
\begin{equation}\label{eq:simphard}\begin{aligned}
|\eta_\heartsuit|_2^2 &= \int_0^\infty e^{-t^2} dt = \frac{\sqrt{\pi}}{2},\\
|\eta_\heartsuit'|_2^2 &= \int_0^\infty (t e^{-t^2/2})^2 dt = \frac{\sqrt{\pi}}{4},\\
|\eta_\heartsuit \cdot \log|_2^2 &= \int_0^\infty e^{-t^2} (\log t)^2 dt \\ &=
 \frac{\sqrt{\pi}}{16}
\left(\pi^2 + 2 \gamma^2 + 8 \gamma \log 2 + 8 (\log 2)^2\right) \leq 1.94753,\\
|\eta_\heartsuit(t)/\sqrt{t}|_1 &= \int_0^\infty \frac{e^{-t^2/2}}{\sqrt{t}} dt = 
\frac{\Gamma(1/4)}{2^{3/4}} \leq 2.15581\\
|\eta_\heartsuit'(t)/\sqrt{t}| = |\eta_\heartsuit(t) \sqrt{t}|_1 &= \int_0^\infty e^{-\frac{t^2}{2}} \sqrt{t} dt = 
\frac{\Gamma(3/4)}{2^{1/4}} \leq 1.03045\\
\left|\eta_\heartsuit'(t) t^{1/2}\right|_1  = 
\left|\eta_\heartsuit(t) t^{3/2}\right|_1 
&= \int_0^\infty e^{- \frac{t^2}{2}}
t^{\frac{3}{2}} dt = 1.07791
.\end{aligned}\end{equation}

We can now state what is really our main result for the Gaussian smoothing.
(The version in the introduction will, as we shall later see, follow from
this, given numerical inputs.)

\begin{prop}\label{prop:bargo}
Let $\eta(t) = e^{-t^2/2}$.
 Let $x\geq 1$, $\delta\in \mathbb{R}$.
Let $\chi$ be a primitive character mod $q$, $q\geq 1$. 
Assume that all non-trivial zeros $\rho$ of $L(s, \chi)$
with $|\Im(\rho)|\leq T_0$
lie on the critical line.
Assume that $T_0\geq \max(4\pi^2 |\delta|,100)$.

 Then
\begin{equation}
\sum_{n=1}^\infty \Lambda(n) \chi(n) e\left(\frac{\delta}{x} n\right) 
\eta\left(\frac{n}{x}\right) =
\begin{cases} \widehat{\eta}(-\delta) x + 
O^*\left(\err_{\eta,\chi}(\delta,x)\right)\cdot x
&\text{if $q=1$,}\\
O^*\left(\err_{\eta,\chi}(\delta,x)\right)\cdot x
&\text{if $q>1$,}\end{cases}\end{equation}
where
\[\begin{aligned}
\err_{\eta,\chi}(\delta,x) &= 
\log \frac{q T_0}{2 \pi} \cdot
\left(4.329 e^{-0.1598 T_0} + 
0.802 |\delta| e^{-0.1065 \left(\frac{T_0}{\pi |\delta|}\right)^2}\right)\\
&+
(2.337  \sqrt{T_0} \log q T_0 + 21.817 \sqrt{T_0} + 2.85 \log q + 74.38)
x^{-\frac{1}{2}}\\
&+ (3 \log q + 14 |\delta| + 17) x^{-1} +
(\log q + 6) \cdot
(1 + 5 |\delta|)\cdot x^{-3/2}.
\end{aligned}\]
\end{prop}
\begin{proof}
Let $F_\delta(s)$ be the Mellin transform of $\eta_\heartsuit(t) e(\delta t)$.
By Lemmas \ref{lem:hausierer} (with $G_\delta=F_\delta$) and \ref{lem:garmonas},
\[\left|\sum_{\text{$\rho$ non-trivial}} F_\delta(\rho) x^\rho\right|\]
is at most (\ref{eq:arnbax}) (with $\eta=\eta_\heartsuit$) times $\sqrt{x}$, plus
\[\log \frac{q T_0}{2 \pi} \cdot
\left(4.329 e^{-0.1598 T_0} + 
0.802 |\delta| e^{-0.1065 \left(\frac{T_0}{\pi |\delta|}\right)^2}\right) \cdot x.
\]
By the norm computations in (\ref{eq:simphard}), we see that
(\ref{eq:arnbax}) is at most
\[
2.337  \sqrt{T_0} \log q T_0 + 21.817 \sqrt{T_0} + 2.85 \log q + 74.38.
\]

Let us now apply Lemma \ref{lem:agamon}. We saw that the value of $R$
in Lemma \ref{lem:agamon} is bounded by (\ref{eq:bleucol}). 
We know that $\eta_\heartsuit(0)=1$. Again by (\ref{eq:simphard}), we get
from (\ref{eq:marenostrum}) that
$c_0 \leq 1.4056 + 13.3466 |\delta|$. Hence
\[|R| \leq 3 \log q + 13.347 |\delta| + 16.695.\]
Lastly,
\[|\eta_\heartsuit'|_2 + 2\pi |\delta| |\eta_\heartsuit|_2 
\leq 0.942 + 4.183 |\delta| \leq 1 + 5 |\delta|.\]
 Clearly
\[(6.01-6)\cdot (1+5 |\delta|) + 13.347 |\delta| + 16.695 < 14 |\delta| + 17,\]
and so we are done. 
\end{proof}

\subsection{The case of $\eta(t) = t^2 e^{-t^2/2}$ and $\eta_*(t)$}
We will now work with a weight based on the Gaussian:
\[\eta(t) = \begin{cases} t^2 e^{-t^2/2} &\text{if $t\geq 0$,}\\
0 &\text{if $t<0$.}\end{cases}\]
The fact that this vanishes at $t=0$ actually makes it easier to work with
at several levels.

Its Mellin transform is just a shift of that of the Gaussian. Write
 \begin{equation}\label{eq:guason}
\begin{aligned}
F_\delta(s) &= (M( e^{- \frac{t^2}{2}} e(\delta t)))(s),\\
G_\delta(s) &= (M(\eta(t) e(\delta t)))(s).\end{aligned}\end{equation}
Then, by the definition of the Mellin transform,
\[G_\delta(s) = F_\delta(s+2).\]


We start by bounding the contribution of zeros with large imaginary part,
just as before.

\begin{lem}\label{lem:festavign}
Let $\eta(t) = t^2 e^{-t^2/2}$. Let $x\in \mathbb{R}^+$, $\delta\in \mathbb{R}$.
Let $\chi$ be a primitive character mod $q$, $q\geq 1$. 
Assume that all non-trivial zeros $\rho$ of $L(s,\chi)$
with $|\Im(\rho)|\leq T_0$
satisfy $\Re(s)=1/2$. Assume that $T_0\geq \max(4\pi^2 |\delta|,100)$.

Write $G_\delta(s)$ for the Mellin transform of $\eta(t) e(\delta t)$. Then
\[\begin{aligned}
\mathop{\sum_{\text{$\rho$ non-trivial}}}_{|\Im(\rho)|>T_0} 
\left|G_{\delta}(\rho)\right| \leq 
T_0 \log \frac{q T_0}{2\pi} \cdot \left( 
3.5 e^{-0.1598 T_0} + 
0.64  e^{-0.1065 \cdot \frac{T_0^2}{(\pi \delta)^2}}\right)
.\end{aligned}\]
\end{lem}
\begin{proof}
We start by writing
\[\mathop{\sum_{\text{$\rho$ non-trivial}}}_{|\Im(\rho)|>T_0} 
\left|G_{\delta}(\rho)\right| = 
\mathop{\sum_{\text{$\rho$ non-trivial}}}_{\Im(\rho)>T_0} 
\left(\left|F_{\delta}(\rho+2)\right| + 
\left|F_{\delta}((1-\rho)+2)\right|\right)
,\]
where we are using $G_\delta(\rho) = F_\delta(\rho+2)$ and the fact that
non-trivial zeros come in pairs $\rho$, $1-\rho$.

By Cor.~\ref{cor:amanita1} with $k=2$,
\[
\mathop{\sum_{\text{$\rho$ non-trivial}}}_{|\Im(\rho)|>T_0} 
\left|G_{\delta}(\rho)\right| \leq \mathop{\sum_{\text{$\rho$ non-trivial}}}_{\Im(\rho)>T_0} f(\Im(\rho)),
\]
where 
\begin{equation}\label{eq:asgo}
f(\tau) = c_2\cdot \begin{cases}
|\tau| e^{-0.1598 |\tau|} + 
\frac{1}{4} \left(\frac{|\tau|}{\pi \delta}\right)^2 e^{-
0.1065 \left(\frac{|\tau|}{\pi \delta}\right)^2} 
& \text{if $|\tau| < \frac{3}{2} (\pi \delta)^2$,}\\ 
|\tau| e^{-0.1598 |\tau|}
& \text{if $|\tau| \geq \frac{3}{2} (\pi \delta)^2$,}
\end{cases}
\end{equation}
where $c_2 = 3.262$.
We are including the term $c_2 |\tau| e^{-0.1598 |\tau|}$ in both cases
in part because we will not bother to take it out (just as we did not bother
in the proof of Lem.~\ref{lem:garmonas}) and in part to ensure that
$f(\tau)$ is a decreasing function of $\tau$ for $\tau\geq T_0$.

We can now apply Lemma \ref{lem:garmola}. We obtain, again,
\begin{equation}
\mathop{\sum_{\text{$\rho$ non-trivial}}}_{\Im(\rho)>T_0} f(\Im(\rho))
\leq \int_{T_0}^\infty f(T) \left(\frac{1}{2\pi}  \log \frac{q T}{2\pi}
+ \frac{1}{4 T}\right) dT.
\end{equation}
Just as before, we will need to estimate some integrals.

For any $y\geq 1$, $c,c_1>0$,
\[\int_{y}^\infty t e^{-ct} dt 
= \left(\frac{y}{c}+\frac{1}{c^2}\right) e^{-c y},\]
\[\begin{aligned}
\int_{y}^\infty \left(t \log t + \frac{c_1}{t}\right) e^{-ct} dt 
&\leq \int_{y}^\infty \left(\left(t + \frac{a-1}{c}\right) \log t 
-\frac{1}{c} - \frac{a}{c^2 t}\right) e^{-ct} dt \\
&= \left(\frac{y}{c} + \frac{a}{c^2}\right) e^{- c y} \log y,
\end{aligned}\]
where \[a = \frac{\frac{\log y}{c} + \frac{1}{c} + \frac{c_1}{y}}{
\frac{\log y}{c} - \frac{1}{c^2 y}}.\]
Setting $c = 0.1598$, $c_1 = \pi/2$, $y = T_0\geq 100$, we
obtain that 
\[\begin{aligned}
\int_{T_0}^\infty &\left(\frac{1}{2\pi} \log \frac{q T}{2 \pi} + 
\frac{1}{4 T}\right) T e^{-0.1598 T} dT\\
&\leq \frac{1}{2\pi} \left(\log \frac{q}{2 \pi} \cdot \left(\frac{T_0}{c}
+ \frac{1}{c^2}\right) + 
\left(\frac{T_0}{c} + \frac{a}{c^2}\right) \log T_0\right)
e^{- 0.1598 T_0}
\end{aligned}\]
and
\[a= \frac{\frac{\log T_0}{0.1598} + \frac{1}{0.1598} + \frac{\pi/2}{T_0}}{
\frac{\log T_0}{0.1598} - \frac{1}{0.1598^2 T_0}}\leq 1.235.\]
Multiplying by $c_2=3.262$ and simplifying by the assumption $T_0\geq 100$, 
we obtain that
\begin{equation}\label{eq:jokr}
\int_{T_0}^\infty 3.255 T e^{-0.1598 T} \left(\frac{1}{2\pi}
\log \frac{q T_0}{2\pi} + \frac{1}{4 T}\right) dT \leq
3.5 T_0 \log \frac{q T_0}{2\pi} \cdot e^{-0.1598 T_0}.\end{equation}


Now let us examine the Gaussian term.
First of all -- when does it arise? If $T_0\geq (3/2) (\pi \delta)^2$,
then $|\tau|\geq (3/2) (\pi \delta)^2$ holds whenever $|\tau|\geq T_0$,
and so (\ref{eq:asgo}) does not give us a Gaussian term. 
Recall that $T_0\geq 4 \pi^2 |\delta|$,
which means that $|\delta|\leq 8/3$ implies that $T_0\geq (3/2) (\pi \delta)^2$.
We can thus assume from now on that $|\delta|>8/3$, since otherwise there is
no Gaussian term to treat.

For any $y\geq 1$, $c,c_1>0$,
\[\int_y^\infty t^2 e^{-c t^2} dt <
\int_y^\infty \left(t^2 + \frac{1}{4 c^2 t^2}\right) e^{- c t^2}
dt = \left(\frac{y}{2 c} + \frac{1}{4 c^2 y}\right)\cdot e^{- c y^2},
\]
\[\begin{aligned}
\int_y^\infty (t^2 \log t + c_1 t) \cdot e^{-c t^2} dt &\leq 
\int_y^\infty \left(t^2 \log t + \frac{a t\log e t}{2 c} - \frac{\log e t}{2 c}
- \frac{a}{4 c^2 t} \right) e^{-c t^2} dt \\&= 
 \frac{(2 c y + a) \log y + a}{4 c^2} \cdot e^{- c y^2},
\end{aligned}\]
where
\[a = 
\frac{c_1 y + \frac{\log ey}{2 c} 
}{\frac{y \log e y}{2 c} - \frac{1}{4
  c^2 y}} =
\frac{1}{y} +
\frac{c_1 y + \frac{1}{4 c^2 y^2}}{\frac{y \log e y}{2 c} - \frac{1}{4
  c^2 y}} 
.\]             
(Note that $a$ decreases as $y\geq 1$ increases.)
Setting $c=0.1065$, $c_1 = 1/(2 |\delta|)\leq 3/16$ and $y = T_0/(\pi |\delta|)
\geq 4 \pi$, we obtain
\[\begin{aligned}
\int_{\frac{T_0}{\pi |\delta|}}^\infty
&\left(\frac{1}{2\pi} \log \frac{q |\delta| t}{2} + \frac{1}{4 \pi
    |\delta| t}\right) t^2 e^{-0.1065 t^2} dt
\\ &\leq \left(\frac{1}{2 \pi} \log \frac{q |\delta|}{2}\right) \cdot 
\left(\frac{T_0}{2 \pi c |\delta|} + \frac{1}{4 c^2 \cdot 4 \pi} \right)
\cdot e^{-0.1065 \left(\frac{T_0}{\pi |\delta|}\right)^2}\\
&+ \frac{1}{2 \pi} \cdot \frac{\left(2 c \frac{T_0}{\pi |\delta|} +
    a\right)
\log \frac{T_0}{\pi |\delta|} + a}{4 c^2}
\cdot e^{-0.1065 \left(\frac{T_0}{\pi |\delta|}\right)^2}
\end{aligned}\]
and \[a\leq  \frac{1}{4 \pi} + 
\frac{4\pi \cdot \frac{3}{16} + \frac{1}{4\cdot 0.1065^2\cdot (4 \pi)^2}}{
\frac{4\pi \log 4\pi e}{2\cdot 0.1065} -
\frac{1}{4\cdot 0.1065^2 \cdot 4 \pi}}\leq 0.092.\]
Multiplying by $(c_2/4) \pi |\delta|$, we get that 
\begin{equation}\label{eq:skinoad}
\int_{T_0}^\infty \frac{c_2}{4} \left(\frac{T}{\pi |\delta|}\right)^2
e^{-0.1065 \left(\frac{T}{\pi |\delta|}\right)^2} 
\left(\frac{1}{2\pi} \log \frac{q T_0}{2\pi} + \frac{1}{4 T}\right) 
dT\end{equation}
is at most $e^{-0.1065
\left(\frac{T_0}{\pi |\delta|}\right)^2}$ times
\begin{equation}\label{eq:caplor}\begin{aligned}
&\left((0.61 T_0 + 0.716 |\delta|)
\cdot \log \frac{q |\delta|}{2} +
0.61 T_0 \log \frac{T_0}{\pi |\delta|}
+ 0.827 |\delta| \log \frac{e T_0}{\pi |\delta|} \right)\\
&\leq \left(0.61+0.828 \cdot \frac{1 + \frac{1}{\log T_0/\pi |\delta|}}{T_0/|\delta|}\right) T_0 \log \frac{q T_0}{2 \pi} \leq
0.64 T_0 \log \frac{q T_0}{2\pi},
\end{aligned}\end{equation}
where we are using several times the assumption that
$T_0\geq 4 \pi^2 |\delta|$.

We sum (\ref{eq:jokr}) and the estimate for (\ref{eq:skinoad}) we have just got
to reach our conclusion.
\end{proof}

Again, we record some norms obtained by symbolic integration:
\begin{equation}\label{eq:drachcat}
\begin{aligned}
|\eta|_2^2 &= \frac{3}{8} \sqrt{\pi},\;\;\;\;\;\;
|\eta'|_2^2 = \frac{7}{16} \sqrt{\pi},\\
|\eta\cdot \log|_2^2 &= \frac{\sqrt{\pi}}{64}
\left(8 (3\gamma-8) \log 2 + 3 \pi^2 + 6 \gamma^2 + 24 (\log 2)^2
+ 16 - 32 \gamma\right) \\ &\leq 0.16364,\\
|\eta(t)/\sqrt{t}|_1 &= \frac{2^{1/4} \Gamma(1/4)}{4} \leq 1.07791,\;\;\;\;\;
|\eta(t) \sqrt{t}|_1 = \frac{3}{4} 2^{3/4} \Gamma(3/4)\leq
1.54568,\\
|\eta'(t)/\sqrt{t}|_1 &= \int_0^{\sqrt{2}} t^{3/2} e^{-\frac{t^2}{2}} dt - 
\int_{\sqrt{2}}^\infty t^{3/2} e^{-\frac{t^2}{2}} dt \leq 1.48469,\\
|\eta'(t) \sqrt{t}|_1 &\leq 1.72169
.\end{aligned}\end{equation}

\begin{prop}\label{prop:magoma}
Let $\eta(t) = t^2 e^{-t^2/2}$.
 Let $x\geq 1$, $\delta\in \mathbb{R}$.
Let $\chi$ be a primitive character mod $q$, $q\geq 1$. 
Assume that all non-trivial zeros $\rho$ of $L(s, \chi)$
with $|\Im(\rho)|\leq T_0$
lie on the critical line.
Assume that $T_0\geq \max(4\pi^2 |\delta|,100)$.

 Then
\begin{equation}
\sum_{n=1}^\infty \Lambda(n) \chi(n) e\left(\frac{\delta}{x} n\right) \eta(n/x) =
\begin{cases} \widehat{\eta}(-\delta) x + 
O^*\left(\err_{\eta,\chi}(\delta,x)\right)\cdot x
&\text{if $q=1$,}\\
O^*\left(\err_{\eta,\chi}(\delta,x)\right)\cdot x
&\text{if $q>1$,}\end{cases}\end{equation}
where
\begin{equation}\label{eq:camdiaw}\begin{aligned}
\err_{\eta,\chi}(\delta,x) &= 
T_0 \log \frac{q T_0}{2\pi} \cdot
\left( 3.5 e^{-0.1598 T_0} + 0.64
  e^{-0.1065 \cdot \frac{T_0^2}{(\pi \delta)^2}}\right)\\
&+
\left(1.22 \sqrt{T_0} \log q T_0 + 5.053 \sqrt{T_0} + 1.423 \log q + 37.19\right)\cdot x^{-1/2}\\
&+ (3+11 |\delta|) x^{-1} + (\log q + 6) \cdot
(1 + 6 |\delta|)\cdot x^{-3/2}.
\end{aligned}\end{equation}
\end{prop}
\begin{proof}
We proceed as in the proof of Prop.~\ref{prop:bargo}. The contribution of 
Lemma \ref{lem:festavign} is
\[T_0 \log \frac{q T_0}{2\pi} \cdot \left( 
3.5 e^{-0.1598 T_0} + 
0.64  e^{-0.1065 \cdot \frac{T_0^2}{(\pi \delta)^2}}\right)\cdot x,\]
whereas the contribution of Lemma \ref{lem:hausierer} is at most
\[
(1.22 \sqrt{T_0} \log q T_0 + 5.053 \sqrt{T_0} + 1.423 \log q + 37.188)
\sqrt{x}.
\]

Let us now apply Lemma \ref{lem:agamon}.
Since $\eta(0)=0$, we have
\[R = O^*(c_0) = O^*(2.138 + 10.99 |\delta|).\]
Lastly,
\[|\eta'|_2 + 2 \pi |\delta| |\eta|_2 \leq 0.881 + 5.123 |\delta|.\]
\end{proof}

Now that we have Prop.~\ref{prop:magoma}, we can derive from it
similar bounds for a smoothing defined as the multiplicative
convolution of $\eta$ with something else.
In general,
for $\varphi_1,\varphi_2:\lbrack 0,\infty)\to \mathbb{C}$, if we know how
to bound 
sums of the form
\begin{equation}\label{eq:kostor}
S_{f,\varphi_1}(x) = \sum_n f(n) \varphi_1(n/x),\end{equation}
we can bound sums of the form $S_{f,\varphi_1 \ast_M \varphi_2}$, simply by changing
the order of summation and integration:
\begin{equation}\label{eq:chemdames}\begin{aligned}
S_{f,\varphi_1 \ast_M \varphi_2} &=
\sum_n f(n) (\varphi_1 \ast_M \varphi_2)\left(\frac{n}{x}\right) \\
&= 
 \int_0^{\infty} \sum_n f(n) \varphi_1\left(\frac{n}{w x}\right) \varphi_2(w)
   \frac{dw}{w}
= \int_{0}^{\infty} S_{f,\varphi_1}(w x) 
\varphi_2(w) \frac{dw}{w}.\end{aligned}\end{equation}
This is particularly nice if $\varphi_2(t)$ vanishes in a neighbourhood of
the origin, since then the argument $w x$ of $S_{f,\varphi_1}(w x)$ is always
large.

We will use $\varphi_1(t) = t^2 e^{-t^2/2}$, $\varphi_2(t) = \eta_1 \ast_M
\eta_1$, where $\eta_1$ is $2$ times the characteristic function of the
interval $\lbrack 1/2,1\rbrack$. This is the example that will be used
in \cite{HelfTern}. The motivation for the choice of
$\varphi_1$ and $\varphi_2$ is clear: we have just got bounds based on
$\varphi_1(t)$ in the major arcs, and the minor-arc bounds in \cite{Helf}
(used in \cite{HelfTern}, not here) were obtained for the weight
$\varphi_2(t)$.

\begin{cor}\label{cor:kolona}
Let $\eta(t) = t^2 e^{-t^2/2}$, $\eta_1 = 2 
\cdot I_{\lbrack 1/2,1\rbrack}$, $\eta_2 = \eta_1 \ast_M \eta_1$. Let
$\eta_* = \eta_2 \ast_M \eta$.
 Let $x\in \mathbb{R}^+$, $\delta\in \mathbb{R}$.
Let $\chi$ be a primitive character mod $q$, $q\geq 1$. 
Assume that all non-trivial zeros $\rho$ of $L(s,\chi)$
with $|\Im(\rho)|\leq T_0$
lie on the critical line.
Assume that $T_0\geq \max(4\pi^2 |\delta|,100)$.

 Then
\begin{equation}\label{eq:expec}
\sum_{n=1}^\infty \Lambda(n) \chi(n) e\left(\frac{\delta}{x} n\right) \eta_*(n/x) =
\begin{cases} \widehat{\eta_*}(-\delta) x + 
O^*\left(\err_{\eta_*,\chi}(\delta,x)\right)\cdot x
&\text{if $q=1$,}\\
O^*\left(\err_{\eta_*,\chi}(\delta,x)\right)\cdot x
&\text{if $q>1$,}\end{cases}\end{equation}
where
\begin{equation}\label{eq:monte}\begin{aligned}
\err_{\eta,\chi_*}(\delta,x) &= 
T_0 \log \frac{q T_0}{2\pi} \cdot
\left( 3.5 e^{-0.1598 T_0} + 0.0019\cdot 
  e^{-0.1065 \cdot \frac{T_0^2}{(\pi \delta)^2}}\right)\\
&+
\left(1.675 \sqrt{T_0} \log q T_0 + 6.936 \sqrt{T_0} +
1.954 \log q + 51.047\right)\cdot x^{-\frac{1}{2}}\\
&+ (6 + 22 |\delta|) x^{-1} + (\log q + 6) \cdot
(3 + 17 |\delta|)\cdot x^{-3/2}.
\end{aligned}\end{equation}
\end{cor}
\begin{proof}
The left side of (\ref{eq:expec}) equals
\[\begin{aligned}
&\int_0^\infty \sum_{n=1}^\infty \Lambda(n) \chi(n) e\left(\frac{\delta
    n}{x}\right) \eta\left(\frac{n}{w x}\right) \eta_2(w) \frac{dw}{w}\\
\\ &=\int_{\frac{1}{4}}^1 \sum_{n=1}^\infty \Lambda(n) \chi(n) e\left(\frac{\delta w
    n}{w x}\right) \eta\left(\frac{n}{w x}\right) \eta_2(w) \frac{dw}{w},
\end{aligned}\]
since $\eta_2$ is supported on $\lbrack -1/4,1\rbrack$.
By Prop.~\ref{prop:magoma}, the main term (if $q=1$) contributes
\[\begin{aligned}
&\int_{\frac{1}{4}}^1 \widehat{\eta}(-\delta w) x w\cdot
\eta_2(w) \frac{dw}{w}
= x \int_{0}^\infty \widehat{\eta}(-\delta w) \eta_2(w) dw\\
&= x \int_{0}^\infty \int_{-\infty}^\infty
\eta(t) e(\delta w t) dt \cdot \eta_2(w) dw
= x \int_{0}^\infty \int_{-\infty}^\infty
\eta\left(\frac{r}{w}\right) e(\delta r) 
\frac{dt}{w} \eta_2(w) dw\\
&= x \int_{-\infty}^\infty \left(\int_0^\infty 
\eta\left(\frac{r}{w}\right) \eta_2(w) \frac{dw}{w}\right)
 e(\delta r) dr = \widehat{\eta_*}(-\delta)\cdot x.\end{aligned}\]
The error term is
\begin{equation}\label{eq:braca}\int_{\frac{1}{4}}^1 
\err_{\eta,\chi}(\delta w, w x) \cdot w x \cdot \eta_2(w) \frac{dw}{w}
= x \cdot 
\int_{\frac{1}{4}}^1 
\err_{\eta,\chi}(\delta w, w x) \eta_2(w) dw.
\end{equation}
Since \[\begin{aligned}
\int_w \eta_2(w) dw &= 1,\;\;\;\;\;\;\;\;\;\;\;\;\;\;\;\;\;\;\;\;\;\;\;\;\;\;\;\;\;\;
 \int_w w^{-1/2} \eta_2(w) dw \leq 1.37259,\\
\int_w w^{-1} \eta_2(w) dw &= 4 (\log 2)^2 \leq 1.92182,\;\;\;\;\;\;\; \int_w w^{-3/2} \eta_2(w) dw \leq 2.74517\end{aligned}\]
and\footnote{By rigorous integration from $1/4$ to $1/2$ and from $1/2$ to
  $1$ using VNODE-LP \cite{VNODELP}.}
\[\int_w e^{-0.1065 \cdot (4 \pi)^2 \left(\frac{1}{w^2} - 1\right)}
\eta_2(w) dw \leq 0.002866,\]
we see that (\ref{eq:camdiaw}) and (\ref{eq:braca}) imply (\ref{eq:monte}).
\end{proof}

\subsection{The case of $\eta_+(t)$}\label{subs:astardo}
We will work with
\begin{equation}\label{eq:dmit}
\eta(t) = \eta_+(t) = h_H(t) t
\eta_{\heartsuit}(t) = h_H(t) t e^{-t^2/2},\end{equation}
where $h_H$ is as in (\ref{eq:dirich2}). We recall that $h_H$ is
a band-limited approximation to the function $h$ defined in (\ref{eq:hortor})
-- to be more precise, $M h_H(i t)$ is the truncation of 
$M h(i t)$ to the interval $\lbrack -H, H\rbrack$.

We are actually defining $h$, $h_H$ and $\eta$ in a slightly different way from
what was done in the first draft of the present paper. The difference is 
instructive. There, $\eta(t)$ was defined as $h_H(t) e^{-t^2/2}$, and
$h_H$ was a band-limited approximation to a function $h$ defined as in
(\ref{eq:hortor}), but with $t^3 (2-t)^3$ instead of $t^2 (2-t)^3$. 
The reason for our
new definitions is that now the truncation of $M h (i t)$ will not break
the holomorphy of $M\eta$, and so we will be able to prove the general results
we proved in \S \ref{subs:genexpf}.

In essence, $M h$ will still be holomorphic because the Mellin transform
of $t \eta_{\heartsuit}(t)$ is holomorphic in the domain we care about, unlike
the Mellin transform of $\eta_{\heartsuit}(t)$, which does have a pole at $s=0$.

As usual, we start by bounding
the contribution of zeros with large imaginary part. 
The procedure is much as before: since $\eta_+(t) = \eta_H(t) 
\eta_\heartsuit(t)$ , the Mellin transform $M\eta_+$ is a convolution
of $M(t e^{-t^2/2})$ and something of support in $\lbrack - H, H\rbrack i$,
 namely,
$M \eta_H$ restricted to the imaginary axis.
This means that the decay of
$M \eta_+$ is (at worst) like the decay of $M(t e^{-t^2/2})$, delayed by $H$.

\begin{lem}\label{lem:schastya}
Let $\eta=\eta_+$ be as in (\ref{eq:dmit}) for some $H \geq 5$.
Let $x\in \mathbb{R}^+$, $\delta\in \mathbb{R}$.
Let $\chi$ be a primitive character mod $q$, $q\geq 1$. 
Assume that all non-trivial zeros $\rho$ of $L(s,\chi)$ 
with $|\Im(\rho)|\leq T_0$
satisfy $\Re(s)=1/2$, where $T_0\geq H+\max(4\pi^2 |\delta|,100)$.

Write $G_\delta(s)$ for the Mellin transform of $\eta(t) e(\delta t)$. Then
\[\begin{aligned}
\mathop{\sum_{\text{$\rho$ non-trivial}}}_{|\Im(\rho)|>T_0} 
&\left|G_{\delta}(\rho)\right|\leq 
\left(9.462 \sqrt{T_0'} e^{-0.1598 T_0'} +
11.287 |\delta|  e^{-0.1065 \left(
\frac{T_0'}{\pi \delta}\right)^2}
\right) \log \frac{q T_0}{2\pi},
\end{aligned}\] 
where $T_0' = T_0-H$.
\end{lem}
\begin{proof}
As usual,
\[\mathop{\sum_{\text{$\rho$ non-trivial}}}_{|\Im(\rho)|>T_0} 
\left|G_{\delta}(\rho)\right| = 
\mathop{\sum_{\text{$\rho$ non-trivial}}}_{\Im(\rho)>T_0} 
\left(\left|G_{\delta}(\rho)\right| + \left|G_{\delta}(1-\rho)\right|\right)
.\]
Let $F_\delta$ be as in (\ref{eq:guason}). Then, since
$\eta_+(t) e(\delta t)= h_H(t) t e^{-t^2/2} e(\delta t)$, where
$h_H$ is as in (\ref{eq:dirich2}), we see by (\ref{eq:mouv}) that
\[G_\delta(s) = \frac{1}{2\pi} \int_{-H}^H Mh(ir) F_\delta(s+1-ir) dr,\]
and so, since $|Mh(ir)| = |Mh(-ir)|$,
\begin{equation}\label{eq:sunodo}
\left|G_{\delta}(\rho)\right| + \left|G_{\delta}(1-\rho)\right| \leq
\frac{1}{2\pi} \int_{-H}^H |Mh(ir)| (|F_\delta(1+\rho-ir)| + |F_\delta(2-(\rho-ir)
)|) dr.\end{equation}

We apply Cor.~\ref{cor:amanita1} with $k=1$ and $T_0-H$ instead of $T_0$, 
and obtain that
$|F_\delta(\rho)|+|F_\delta(1-\rho)|\leq g(\tau)$, where
\begin{equation}\label{eq:saintsze}g(\tau) = c_1\cdot 
\left(\sqrt{|\tau|} e^{-0.1598 |\tau|} + \frac{|\tau|}{2 \pi |\delta|} \cdot 
e^{-0.1065 \left(\frac{\tau}{\pi \delta}\right)^2}\right)
\end{equation}
where $c_1 = 3.516$. (As in the proof of Lemma \ref{lem:festavign},
we are really putting in extra terms so as to simplify our integrals.)

From (\ref{eq:sunodo}), we conclude that
\[|G_\delta(\rho)| + |G_\delta(1-\rho)| \leq f(\tau),\]
for $\rho = \sigma+i
\tau$, $\tau>0$, where
\[f(\tau) = \frac{|M h(i r)|_1}{2\pi} \cdot g(\tau-H)\]
is decreasing for $\tau\geq T_0$ (because $g(\tau)$ is decreasing for
$\tau\geq T_0-H$). By (\ref{eq:marpales}), $|M h(i r)|_1\leq 16.193918$. 

We apply Lemma \ref{lem:garmola}, and get that
\begin{equation}\label{eq:natju}
\begin{aligned}\mathop{\sum_{\text{$\rho$ non-trivial}}}_{|\Im(\rho)|>T_0} 
\left|G_{\delta}(\rho)\right| &\leq \int_{T_0}^\infty
f(T) \left(\frac{1}{2\pi}  \log \frac{q T}{2\pi}
+ \frac{1}{4 T}\right) dT\\
&= \frac{|Mh(i r)|_1}{2\pi} 
\int_{T_0}^\infty
g(T-H) \left(\frac{1}{2\pi}  \log \frac{q T}{2\pi}
+ \frac{1}{4 T }\right) dT.\end{aligned}\end{equation}

Now we just need to estimate some integrals.
For any $y\geq e^2$, $c>0$ and $\kappa,\kappa_1\geq 0$,
\[\int_y^\infty \sqrt{t} e^{-c t} dt \leq
\left(\frac{\sqrt{y}}{c} + \frac{1}{2 c^2 \sqrt{y}}\right) e^{-c y},\]
\[\begin{aligned}
\int_y^\infty \left(\sqrt{t} \log (t + \kappa) + \frac{\kappa_1}{\sqrt{t}}
\right) e^{-c t} dt \leq
\left(\frac{\sqrt{y}}{c} + \frac{a}{c^2 \sqrt{y}}\right) \log(y+\kappa) 
e^{-c y},
\end{aligned}\]
where \[a = \frac{1}{2} + 
\frac{1 + c \kappa_1}{\log(y+\kappa)}.\]

The contribution of the exponential term in (\ref{eq:saintsze})
to (\ref{eq:natju}) thus equals 
\begin{equation}\label{eq:sibenize}\begin{aligned}
&\frac{c_1 |Mh(i r)|_1}{2\pi} 
\int_{T_0}^\infty \left(\frac{1}{2\pi} \log \frac{q T}{2\pi} + \frac{1}{4 T}
\right) \sqrt{T-H} \cdot e^{-0.1598 (T-H)} dT\\
&\leq 9.06194
\int_{T_0 - H}^\infty \left(\frac{1}{2\pi} \log (T+H) + \frac{\log \frac{q}{2\pi}}{2\pi} + \frac{1}{4 T}\right) \sqrt{T} e^{-0.1598 T} dT\\
&\leq \frac{9.06194}{2\pi}
 \left(\frac{\sqrt{T_0-H}}{0.1598} + \frac{a}{0.1598^2 \sqrt{T_0-H}}
\right) \log \frac{q T_0}{2 \pi}\cdot  e^{-0.1598 (T_0-H)},
\end{aligned}
\end{equation}
where $a = 1/2 + (1+0.1598 \pi/2)/\log T_0$.
Since $T_0-H\geq 100$ and $T_0\geq 105$, this is at most
\[9.4612 \sqrt{T_0-H} \log \frac{q T_0}{2\pi} \cdot e^{-0.1598 (T_0-H)}.\]

We now estimate a few more integrals so that we can handle the Gaussian term
in (\ref{eq:saintsze}). For any $y>1$, $c>0$, $\kappa,\kappa_1\geq 0$,
\[\int_y^\infty t e^{-c t^2} dt = \frac{e^{-c y^2}}{2 c},\]
\[\int_y^\infty (t \log (t+\kappa) + \kappa_1) e^{-c t^2} dt \leq 
\left(1 + \frac{\kappa_1 + \frac{1}{2 c y}}{y \log (y + \kappa)}
\right) \frac{\log (y+\kappa) \cdot e^{-c y^2}}{2 c}
\]
Proceeding just as before, we see that the contribution of the
Gaussian term in (\ref{eq:saintsze}) to (\ref{eq:natju}) is at most
\begin{equation}\label{eq:sibabita}\begin{aligned}
&\frac{c_1 |Mh(i r)|_1}{2\pi} 
\int_{T_0}^\infty \left(\frac{1}{2\pi} \log \frac{q T}{2\pi} + \frac{1}{4 T}
\right) \frac{T-H}{2 \pi |\delta|} \cdot e^{-0.1065 \left(\frac{T-H}{\pi \delta}\right)^2} dT\\
&\leq 9.06194\cdot \frac{|\delta|}{4}
\int_{\frac{T_0 - H}{\pi |\delta|}}^\infty \left(\log \left(T+
\frac{H}{\pi |\delta|}\right) + \log \frac{q |\delta|}{2} + \frac{\pi/2}{T}\right) T e^{-0.1065 T^2} dT\\
&\leq 9.06194\cdot \frac{|\delta|}{8\cdot 0.1065}
 \left(1 + \frac{\frac{\pi}{2} + \frac{\pi |\delta|}{2\cdot 0.1065\cdot (T_0-H)}
}{\frac{T_0-H}{\pi |\delta|} \log \frac{T_0}{\pi |\delta|}}
\right) \log \frac{q T_0}{2 \pi}\cdot  e^{-0.1065 \left(\frac{T_0-H}{\pi \delta}
\right)^2},
\end{aligned}
\end{equation}
Since $(T_0-H)/(\pi |\delta|) \geq 4 \pi$, this is at most
\[11.287 |\delta| \log \frac{q T_0}{2\pi} \cdot e^{-0.1065 \left(
\frac{T_0-H}{\pi \delta}\right)^2}.\]
\end{proof}

\begin{prop}\label{prop:unease}
Let $\eta=\eta_+$ be as in (\ref{eq:dmit}) for some $H \geq 50$.
Let $x\geq 10^3$, $\delta\in \mathbb{R}$.
Let $\chi$ be a primitive character mod $q$, $q\geq 1$. 
Assume that all non-trivial zeros $\rho$ of $L(s,\chi)$
with $|\Im(\rho)|\leq T_0$
lie on the critical line, where $T_0\geq H+\max(4\pi^2 |\delta|,100)$.

Then
\begin{equation}
\sum_{n=1}^\infty \Lambda(n) \chi(n) e\left(\frac{\delta}{x} n\right) \eta_+(n/x) =
\begin{cases} \widehat{\eta_+}(-\delta) x + 
O^*\left(\err_{\eta_+,\chi}(\delta,x)\right)\cdot x
&\text{if $q=1$,}\\
O^*\left(\err_{\eta_+,\chi}(\delta,x)\right)\cdot x
&\text{if $q>1$,}\end{cases}\end{equation}
where
\begin{equation}\label{eq:nochpai}\begin{aligned}
\err_{\eta_+,\chi}(\delta,x) &= 
\left(9.462 \sqrt{T_0'} \cdot e^{-0.1598 T_0'} +
11.287 |\delta|  e^{-0.1065 \left(
\frac{T_0'}{\pi \delta}\right)^2}
\right) \log \frac{q T_0}{2\pi}\\
&+ (1.631 \sqrt{T_0} \log q T_0 + 12.42 \sqrt{T_0} + 1.321 \log q + 34.51)
x^{-1/2},\\
& + (9 + 11 |\delta|) x^{-1} + (\log q) (11 + 6 |\delta|) x^{-3/2},
\end{aligned}\end{equation}
where $T_0' = T_0 - H$.
\end{prop}
\begin{proof}
We can apply Lemmas \ref{lem:agamon} and Lemma \ref{lem:hausierer}
because $\eta_+(t)$, $(\log t) \eta_+(t)$ and $\eta_+'(t)$ are in $\ell_2$
(by (\ref{eq:mastodon}), (\ref{eq:pamiatka}) and (\ref{eq:miran}))
and $\eta_+(t) t^{\sigma-1}$ and $\eta_+'(t) t^{\sigma-1}$ are in $\ell_1$
for $\sigma$ in an open interval containing $\lbrack 1/2,3/2\rbrack$
(by (\ref{eq:paytoplay}) and (\ref{eq:uzsu})).
(Because of (\ref{eq:hutterite}), the fact that $\eta_+(t) t^{-1/2}$ and
$\eta_+(t) t^{1/2}$ are in $\ell_1$ implies that $\eta_+(t) \log t$ is also
in $\ell_1$, as is required by Lemma \ref{lem:hausierer}.)

We apply Lemmas \ref{lem:agamon}, 
 \ref{lem:hausierer} and \ref{lem:schastya}. We bound the norms
involving $\eta_+$ using the estimates in \S \ref{subs:daysold} and
\S \ref{subs:weeksold}. Since $\eta_+(0)=0$ (by the definition
(\ref{eq:patra})
of $\eta_+$), the term $R$ in (\ref{eq:estromo}) is at most $c_0$,
where $c_0$ is as in (\ref{eq:marenostrum}). We bound
\[\begin{aligned}c_0&\leq \frac{2}{3} 
\left(2.922875 \left(\sqrt{\Gamma(1/2)} + \sqrt{\Gamma(3/2)}\right) +
1.062319 \left(\sqrt{\Gamma(5/2)} + \sqrt{\Gamma(7/2)}\right)\right)
\\
&+ \frac{4\pi}{3} |\delta| \cdot 1.062319
\left(\sqrt{\Gamma(3/2)} + \sqrt{\Gamma(5/2)}\right)
\leq 6.536232 + 
9.319578 |\delta|
\end{aligned}\]
using (\ref{eq:paytoplay}) and (\ref{eq:uzsu}).
By (\ref{eq:mastodon}), (\ref{eq:miran}) and the assumption
$H\geq 50$,
\[|\eta_+|_2\leq 0.80044,\;\;\;\;\;\;\;\;
|\eta_+'|_2 \leq 10.845789.\]
Thus, the error terms in (\ref{eq:marmar}) total at most
\begin{equation}\label{eq:hust}\begin{aligned}
6.536232 + 
&9.319578 |\delta| + (\log q + 6.01) (
10.845789 + 2\pi\cdot 0.80044 |\delta| ) x^{-1/2}\\
&\leq 
 9 + 11 |\delta| + (\log q) (11 + 6 |\delta|) x^{-1/2}.
\end{aligned}\end{equation}

The part of the sum $\sum_\rho G_\delta(\rho) x^\rho$ in (\ref{eq:marmar})
corresponding to zeros $\rho$ with $|\Im(\rho)|>T_0$ gets estimated by
Lem \ref{lem:schastya}. By Lemma \ref{lem:hausierer},
the part of the sum corresponding to zeros $\rho$ with $|\Im(\rho)|\leq T_0$
is at most
\[(1.631 \sqrt{T_0} \log q T_0 + 12.42 \sqrt{T_0} + 1.321 \log q + 34.51)
x^{1/2},\]
where we estimate the norms $|\eta_+|_2$, $|\eta\cdot \log |_2$ and
$|\eta(t)/\sqrt{t}|_1$ by (\ref{eq:mastodon}), (\ref{eq:pamiatka}) and (\ref{eq:paytoplay}).
\end{proof}

\subsection{A sum for $\eta_+(t)^2$}

Using a smoothing function sometimes leads to considering sums involving the
square of the smoothing function. In particular, \cite{HelfMaj} requires
a result involving $\eta_+^2$ -- something that could be
slightly challenging to prove, given the way in which $\eta_+$ is defined. 
Fortunately, we have bounds on $|\eta_+|_\infty$ and other
$\ell_\infty$-norms (see Appendix \ref{subs:byron}).
Our task will also be made easier by the fact
that we do not have a phase $e(\delta n/x)$ this time. 
All in all, this will be yet another demonstration of the generality
of the framework developed in \S \ref{subs:genexpf}.

\begin{prop}\label{prop:konechno}
Let $\eta=\eta_+$ be as in (\ref{eq:dmit}), $H\geq 50$. Let $x\geq 10^8$.
Assume that all non-trivial zeros $\rho$ of the Riemann zeta function $\zeta(s)$
with $|\Im(\rho)|\leq T_0$ lie on the critical line, where 
$T_0\geq 2 H + \max(H/4,50)$.

Then
\begin{equation}\label{eq:horrorshow}
\sum_{n=1}^\infty \Lambda(n) (\log n) \eta_+^2(n/x)
 = x\cdot \int_0^\infty \eta_+^2(t) \log x t\; dt
+ O^*(\err_{\ell_2,\eta_+})\cdot x \log x,\end{equation}
where
\begin{equation}\label{eq:galo}\begin{aligned}
\err_{\ell_2,\eta_+} &= 
\left(0.607 \frac{(\log T_0)^2}{\log x} + 1.21 (\log T_0)\right)
\sqrt{T_0} e^{-\frac{\pi (T_0 - 2 H)}{4}}\\
&+ (2.06 \sqrt{T_0} \log T_0 + 43.87) \cdot x^{-1/2}
\end{aligned}\end{equation}
\end{prop}
\begin{proof}
We will need to consider two smoothing functions, namely,
$\eta_{+,0}(t) = \eta_+(t)^2$ and $\eta_{+,1} =\eta_+(t)^2 \log t$. Clearly,
\[\sum_{n=1}^\infty \Lambda(n) (\log n) \eta_+^2(n/x)
= (\log x) \sum_{n=1}^\infty \Lambda(n) \eta_{+,0}(n/x) +  
\sum_{n=1}^\infty \Lambda(n) \eta_{+,1}(n/x).
\]
Since $\eta_+(t) = h_H(t) t e^{-t^2/2}$,
\[\eta_{+,0}(r) = h_H^2(t) t e^{-t^2},\;\;\;\;\;\;\;\;\;\;
\eta_{+,1}(r) = h_H^2(t) (\log t) t e^{-t^2}.\]
Let $\eta_{+,2} = (\log x) \eta_{+,0} + \eta_{+,1} = \eta_+^2(t) \log x t$.

We wish to apply Lemma \ref{lem:agamon}. For this, we must first
check that some norms are finite. Clearly,
\[\begin{aligned}
\eta_{+,2}(t) &= \eta_{+}^2(t) \log x + \eta_+^2(t) \log t\\
\eta_{+,2}'(t) &= 2\eta_+(t) \eta_+'(t) \log x +
2 \eta_+(t) \eta_+'(t) \log t + \eta_+^2(t)/t .\end{aligned}\]
Thus, we see that
$\eta_{+,2}(t)$ is in $\ell_2$ because
$\eta_+(t)$ is in $\ell_2$ and $\eta_+(t)$, $\eta_+(t) \log t$ are both in
$\ell_\infty$ (see (\ref{eq:mastodon}), (\ref{eq:malgache}),
(\ref{eq:dalida})):
\begin{equation}\label{eq:alcmeon}\begin{aligned}
\left|\eta_{+,2}(t)\right|_2 &\leq  
\left|\eta_+^2(t)\right|_2 \log x + \left|\eta_+^2(t) \log t \right|_2\\
&\leq \left|\eta_+\right|_\infty  
\left|\eta_+\right|_2 \log x + \left|\eta_+(t) \log t \right|_\infty
\left|\eta_+\right|_2 .
\end{aligned}\end{equation}
Similarly, $\eta_{+,2}'(t)$ is in $\ell_2$ because $\eta_+(t)$ is in
$\ell_2$, $\eta_+'(t)$ is in $\ell_2$ (\ref{eq:miran}), and
$\eta_+(t)$, $\eta_+(t) \log t$ and  
$\eta_+(t)/t$ (see (\ref{eq:gobmark})) are all in $\ell_\infty$:
\begin{equation}\label{eq:comor}\begin{aligned}
\left|\eta_{+,2}'(t)\right|_2 &\leq  
\left|2 \eta_+(t) \eta_+'(t)\right|_2 \log x + 
\left|2 \eta_+(t) \eta_+'(t) \log t \right|_2 +
\left|\eta_+^2(t)/t\right|_2\\
&\leq 2 \left|\eta_+\right|_\infty \left|\eta_+'\right|_2 \log x + 
2 \left|\eta_+(t) \log t\right|_\infty \left|\eta_+'\right|_2 + 
\left|\eta_+(t)/t\right|_\infty \left|\eta_+\right|_2 . 
\end{aligned}\end{equation}
In the same way, we see that $\eta_{+,2}(t) t^{\sigma-1}$ is in
$\ell_1$ for all $\sigma$ in $(-1,\infty)$ (because
the same is true of $\eta_+(t) t^{\sigma-1}$ (\ref{eq:paytoplay}),
and $\eta_+(t)$, $\eta_+(t) \log t$ are both in
$\ell_\infty$) and $\eta_{+,2}'(t) t^{\sigma-1}$ is in
$\ell_1$ for all $\sigma$ in $(0,\infty)$ (because the same is true
of $\eta_+(t) t^{\sigma-1}$ and $\eta_+'(t) t^{\sigma-1}$ (\ref{eq:uzsu}),
and $\eta_+(t)$, $\eta_+(t) \log t$, $\eta_+(t)/t$  are all in $\ell_\infty$).

We now apply Lemma \ref{lem:agamon} with $q=1$, $\delta=0$. 
Since $\eta_{+,2}(0)=0$, the residue term $R$ equals $c_0$, which, in turn,
is at most $2/3$ times
\[\begin{aligned}
&\left(\left|\eta_+\right|_\infty \log x + 
\left|\eta_+(t) \log t\right|_\infty\right) \left( 
\left|\eta_+(t)/\sqrt{t}\right|_1 + \left|\eta_+(t) \sqrt{t}\right|_1
\right)\\
&+
2 \left(\left|\eta_+\right|_\infty \log x + 
 \left|\eta_+(t) \log t\right|_\infty\right) \left( 
\left|\eta_+'(t)/\sqrt{t}\right|_1 + \left|\eta_+'(t) \sqrt{t}\right|_1
\right)\\
&+ \left|\eta_+(t)/t\right|_\infty \left( 
\left|\eta_+(t)/\sqrt{t}\right|_1 + \left|\eta_+(t) \sqrt{t}\right|_1
\right).
\end{aligned}\]
Using the bounds 
 (\ref{eq:malgache}), (\ref{eq:dalida}), (\ref{eq:gobmark})
(with the assumption $H\geq 50$),
(\ref{eq:paytoplay}) and (\ref{eq:uzsu}), we get that this means that
\[c_0\leq 18.15014 \log x + 7.84532 .
\]

Since $q=1$ and $\delta=0$, we get from (\ref{eq:comor}) that
\[\begin{aligned}
(\log q + 6.01) \cdot &\left(\left|\eta_{+,2}'\right|_2 + 2 \pi |\delta|
\left|\eta_{+,2}\right|_2\right) x^{-1/2} \\&= 6.01
\left|\eta_{+,2}'\right|_2 x^{-1/2}
\leq (162.56 \log x + 59.325) x^{-1/2}.\end{aligned}\]
Using the assumption $x\geq 10^8$, we obtain
\begin{equation}\label{eq:pastafrola}c_0 + (162.56 \log x + 59.325) x^{-1/2}\leq
18.593 \log x.\end{equation}

We now apply Lemma \ref{lem:hausierer} -- as we may, because of the 
finiteness of the norms
we have already checked, together with
\begin{equation}\label{eq:sansan}\begin{aligned}
\left|\eta_{+,2}(t) \log t\right|_2 &\leq 
\left|\eta_+^2(t) \log t\right|_2 \log x +
\left|\eta_+^2(t) (\log t)^2 \right|_2 \\ &\leq
\left|\eta_+(t) \log t\right|_\infty 
\left(\left|\eta_+(t)\right|_2 \log x + \left|\eta_+(t) \log
    t\right|_2\right)\\
&\leq 0.40742\cdot (0.80044 \log x + 0.82999) \leq 0.32396 \log x + 0.33592
\end{aligned}\end{equation}
(by (\ref{eq:dalida}), (\ref{eq:mastodon}) 
and (\ref{eq:pamiatka}); use the assumption $H\geq 50$). We will also need
the bounds
\begin{equation}\label{eq:cuahcer}
\left|\eta_{+,2}(t)\right|_2\leq 0.99811 \log x + 0.32612\end{equation}
(from (\ref{eq:alcmeon}), by the norm bounds (\ref{eq:malgache}),
(\ref{eq:dalida}) and (\ref{eq:mastodon}), all with $H\geq 50$) and
\begin{equation}\label{eq:jostume}\begin{aligned}
\left|\eta_{+,2}(t)/\sqrt{t}\right|_1 &\leq 
\left(\left|\eta_+(t)\right|_\infty \log x + 
 \left|\eta_+(t) \log t\right|_\infty\right) \left|\eta_+(t)/\sqrt{t}\right|_1 \\
&\leq 1.24703 \log x + 0.40745
\end{aligned}\end{equation}
(by (\ref{eq:malgache}), (\ref{eq:dalida}) (again with $H\geq 50$)
and (\ref{eq:paytoplay})).
We obtain that the sum $\sum_\rho |G_0(\rho)| x^\rho$ (where
$G_0(\rho) = M \eta_{+,2}(\rho)$) over
all non-trivial zeros $\rho$ with $|\Im(\rho)|\leq T_0$ is at most
$x^{1/2}$ times
\begin{equation}\label{eq:trucidor}
\begin{aligned}
(1.3221 \log x + 0.6621) \sqrt{T_0} \log T_0 &+ 
(3.2419 \log x + 5.0188) \sqrt{T_0}\\ &+ 43.1 \log x + 14.1,\end{aligned}
\end{equation}
where we are bounding norms by (\ref{eq:cuahcer}), (\ref{eq:sansan}) 
and (\ref{eq:jostume}).
(We are using the fact that
$T_0\geq 2 \pi \sqrt{e}$ to ensure that the quantity
$\sqrt{T_0} \log T_0 - (\log 2 \pi \sqrt{e}) \sqrt{T_0}$ being multiplied
by $|\eta_{+,2}|_2$ is positive; thus, an upper bound for
$|\eta_{+,2}|_2$ suffices.) By the assumptions $x\geq 10^{8}$,
$T_0\geq 150$, (\ref{eq:trucidor}) is at most
\[
2.06 \log x \cdot \sqrt{T_0} \log T_0 + 43.866 \log x.\]
Note that the term $18.539 \log x$ from (\ref{eq:pastafrola})
is at most $0.002 x^{1/2} \log x$.

It remains to bound the sum of $M\eta_{+,2}(\rho)$ over non-trivial zeros
with $|\Im(\rho)|> T_0$. This we will do, as usual, by Lemma
\ref{lem:garmola}. For that, we will need to bound 
$M\eta_{+,2}(\rho)$ for $\rho$ in the critical strip.

The Mellin transform of $e^{-t^2}$ is $\Gamma(s/2)/2$, and so
the Mellin transform of $t e^{-t^2}$ is $\Gamma((s+1)/2)/2$.
By (\ref{eq:harva}),
this implies that the Mellin transform
of $(\log t) t e^{-t^2}$ is $\Gamma'((s+1)/2)/4$.
 Hence, by (\ref{eq:mouv}), 
\begin{equation}\label{eq:moses}
M\eta_{+,2}(s) = \frac{1}{4\pi} \int_{-\infty}^{\infty}
M(h_H^2)(ir) \cdot F_x\left(s-ir\right) dr,
\end{equation}
where
\begin{equation}\label{eq:luke}
F_x(s) = (\log x) \Gamma\left(\frac{s+1}{2}\right) + 
\frac{1}{2} \Gamma'\left(\frac{s+1}{2}\right).
\end{equation}

Moreover, 
\begin{equation}\label{eq:langosta}
M(h_H^2)(i r) = \frac{1}{2\pi} \int_{-\infty}^{\infty}
Mh_H(iu) Mh_H(i(r-u))\;  du, 
\end{equation}
and so  $M(h_H^2)(ir)$ is supported
on $\lbrack -2H,2H\rbrack$. We also see that 
$|Mh_H^2(ir)|_1\leq |Mh_H(ir)|_1^2/2\pi$.
We know that $|Mh_H(ir)|_1^2/2\pi\leq 41.73727$
by (\ref{eq:marpales}).

Hence
\begin{equation}\label{eq:rancune}\begin{aligned}
&|M\eta_{+,2}(s)| \leq \frac{1}{4\pi} \int_{-\infty}^\infty |M(h_H^2)(i r)| dr
\cdot \max_{|r|\leq 2H} |F_x(s-ir)|\\
&\leq \frac{41.73727}{4\pi} \cdot
\max_{|r|\leq 2H} |F_x(s-ir)| \leq 
3.32135 \cdot \max_{|r|\leq 2H} |F_x(s-ir)|.
\end{aligned}\end{equation}

By \cite[5.6.9]{MR2723248}
(Stirling with explicit constants),
\begin{equation}
|\Gamma(s)|\leq \sqrt{2\pi} |s|^{\Re(s)-1/2} e^{-\pi |\Im(s)|/2}
e^{1/6|z|},\end{equation}
and so
\begin{equation}\label{eq:lucia}
|\Gamma(s)|\leq
2.526 \sqrt{|\Im(s)|} e^{-\pi |\Im(s)|/2}
\end{equation}
for $s\in \mathbb{C}$ with $0<\Re(s)\leq 1$ and $|\Im(s)|\geq 25$.
Moreover, by \cite[5.11.2]{MR2723248} and the remarks at the beginning of
\cite[5.11(ii)]{MR2723248},
\[
\frac{\Gamma'(s)}{\Gamma(s)} = 
\log s - \frac{1}{2s} + O^*\left(\frac{1}{12 |s|^2} \cdot 
\frac{1}{\cos^3 \theta/2}\right)\]
for $|\arg(s)| < \theta$ ($\theta\in (-\pi,\pi)$). 
Again, for $s= \sigma + i \tau$ with $0< \sigma \leq 1$ and 
$|\tau|\geq 25$, this gives us
\[\begin{aligned}\frac{\Gamma'(s)}{\Gamma(s)} &= \log |\tau| + 
\log \frac{\sqrt{|\tau|^2+1}}{|\tau|} + O^*\left(\frac{1}{2 |\tau|}\right) 
+ O^*\left(\frac{1}{12 |\tau|^2} \cdot
\frac{1}{\cos^3 \frac{\pi - \arctan |\tau|}{2}}\right) 
\\ &= \log |\tau| + O^*\left(\frac{1}{2 |\tau|^2} + \frac{1}{2 |\tau|}\right)
+ \frac{O^*(0.236)}{|\tau|^2} \\ 
&= \log |\tau| + O^*\left(\frac{0.53}{|\tau|}\right).\end{aligned}\]
Hence, for $-1\leq \Re(s)\leq 1$ and $|\Im(s)|\geq 50$,
\begin{equation}\label{eq:adela} \begin{aligned}
|F_x(s)| &\leq \left((\log x) + \frac{1}{2} \log \left|\frac{\tau}{2}\right| 
+ \frac{1}{2} O^*\left(\frac{0.53}{|\tau/2|}\right)\right)
\Gamma\left(\frac{s+1}{2}\right)\\
&\leq 2.526 ((\log x) + \frac{1}{2} \log |\tau| - 0.335)
\sqrt{|\tau|} e^{-\pi |\tau|/2}.\end{aligned}\end{equation}

Thus, by (\ref{eq:rancune}),
for $\rho=\sigma+i\tau$ with $|\tau|\geq T_0\geq 2H+50$
and $-1\leq \sigma\leq 1$,
\[|M\eta_{+,2}(\rho)|\leq f(\tau)\]
where
\begin{equation}\label{eq:cajun}
f(T) = 
 8.39 \left(\log x + \frac{1}{2} \log T\right) 
\sqrt{\frac{|\tau|}{2}-H} \cdot e^{-\frac{\pi
  (|\tau|-2H)}{4}}.
\end{equation}
The functions $t\mapsto \sqrt{t} e^{-\pi t/2}$
and $t\mapsto (\log t) \sqrt{t} e^{-\pi t/2}$ are decreasing for $t\geq 3/\pi$;
setting $t = T/2 - H$, we see that the right side of
(\ref{eq:cajun}) is a decreasing function of $T$ for $T\geq T_0$,
since $T_0/2 - H\geq 25 > 3/\pi$.

We can now apply Lemma \ref{lem:garmola}, and get that
\begin{equation}\label{eq:kolmo}
\mathop{\sum_{\text{$\rho$ non-trivial}}}_{|\Im(\rho)|>T_0}
|M\eta_{+,2}(\rho)| 
\leq \int_{T_0}^\infty f(T) \left(\frac{1}{2\pi} \log \frac{T}{2 \pi} +
\frac{1}{4 T}\right) dT.\end{equation}
Since $T\geq T_0\geq 150>2$, we know that
$((1/2\pi) \log(T/2\pi) + 1/4T) \leq (1/2\pi) \log T$.
Hence, the right side of 
(\ref{eq:kolmo}) is at most
\begin{equation}\label{eq:gegersh}\frac{8.39}{2\pi} \int_{T_0}^\infty 
\left((\log x) (\log T) + \frac{(\log T)^2}{2}\right) \sqrt{\frac{T}{2}} 
e^{-\frac{\pi (T - 2 H)}{4}} dT.\end{equation}
In general, for $T_0\geq e^2$,
\[ \begin{aligned}
\int_{T_0}^\infty \sqrt{T} (\log T)^2 e^{-\pi T/4} dT
&\leq \frac{4}{\pi} \left(\sqrt{T_0} (\log T_0)^2
+  \frac{2/\pi}{\sqrt{T_0}} (\log e^2 T_0)^2\right) e^{-\frac{\pi T_0}{4}},
\\
\int_{T_0}^\infty \sqrt{T} (\log T) e^{-\pi T/4} dT
&\leq  \frac{4}{\pi} \left(\sqrt{T_0} (\log T_0)
+ \frac{2/\pi}{\sqrt{T_0}} \log e^2 T_0\right) e^{-\frac{\pi T_0}{4}};
\end{aligned}\]
for $T_0\geq 150$, the quantities on the right are at most $1.284
\sqrt{T_0} (\log T_0)^2 e^{-\pi T_0/4}$ and
$1.281 \sqrt{T_0} (\log T_0) e^{-\pi T_0/4}$, respectively.
Thus, (\ref{eq:kolmo}) and (\ref{eq:gegersh}) give us that
\[\begin{aligned}
&\mathop{\sum_{\text{$\rho$ non-trivial}}}_{|\Im(\rho)|>T_0}
|M\eta_{+,2}(\rho)| \\&\leq
\frac{8.39}{2\pi\cdot \sqrt{2}} \left(\frac{1.284}{2} (\log T_0)^2
+1.281 (\log x) (\log T_0)
\right) \sqrt{T_0} e^{-\frac{\pi (T_0 - 2 H)}{4}}\\
&\leq  (0.607 (\log T_0)^2 + 1.21 (\log x) (\log T_0))
\sqrt{T_0} e^{-\frac{\pi (T_0 - 2 H)}{4}}.
\end{aligned}\]
\end{proof}
\subsection{A verification of zeros and its consequences}

David Platt verified in his doctoral thesis \cite{Platt}, 
 that, for every
primitive character $\chi$ of conductor $q\leq 10^5$, all the non-trivial
zeroes of $L(s,\chi)$ with imaginary part $\leq 10^8/q$ lie on the critical
line, i.e., have real part exactly $1/2$. (We call this a {\em
GRH verification up to $10^8/q$}.)

In work undertaken in coordination with the present project \cite{Plattfresh},
Platt has extended these computations to
\begin{itemize}
\item all odd $q\leq 3\cdot 10^5$, with $T_q = 10^8/q$,
\item all even $q\leq 4\cdot 10^5$, with 
$T_q = \max(10^8/q,200 + 7.5\cdot 10^7/q)$.
\end{itemize}  
The method used was rigorous; its implementation uses interval 
arithmetic.

Let us see what this verification gives us when used 
as an input to Prop.~\ref{prop:bargo}. We are interested in bounds on
$|\err_{\eta,\chi^*}(\delta,x)|$ for $q\leq r$ and $|\delta|\leq \delta_0 r/2q$.
We set $r=3\cdot 10^5$ and $\delta_0 = 8$,
and so $|\delta|\leq 4 r/q$. (We will not be using the verification for
$q$ even with $3\cdot 10^5 < q\leq 4\cdot 10^5$.)

We let $T_0 = 10^8/q$. Thus,
\begin{equation}\label{eq:maljust}
\begin{aligned}T_0 &\geq \frac{10^8}{3 \cdot 10^5} = \frac{1000}{3},\\
\frac{T_0}{\pi |\delta|} &\geq 
 \frac{10^8/q}{\pi \cdot 4 r/q} 
 = \frac{1000}{12 \pi} \end{aligned}\end{equation}
and so
\[\begin{aligned}4.329 e^{-0.1598 T_0} &\leq 3.184 \cdot 10^{-23},\\
0.802 e^{-0.1065 \frac{T_0^2}{(\pi \delta)^2}} &\leq 4.3166 \cdot 10^{-33}.
\end{aligned}\]
Since $|\delta|\leq 4 r/q \leq 1.2\cdot 10^6/q\leq 1.2\cdot 10^6$
 and $q T_0 \leq 10^8$, this
gives us
\[\begin{aligned}
\log \frac{q T_0}{2\pi} \cdot \left(4.329 e^{-0.1598 T_0} 
+ 0.802 |\delta| e^{-0.1065 \frac{T_0^2}{(\pi \delta)^2}}\right) &\leq 5.28\cdot 10^{-22} + \frac{8.59\cdot 10^{-26}}{q} \\ &\leq
5.281 \cdot 10^{-22}
.\end{aligned}\]
Again by $T_0=10^8/q$,
\[2.337  \sqrt{T_0} \log q T_0 + 21.817 \sqrt{T_0} + 2.85 \log q + 74.38\]
is at most
\[\frac{648662}{\sqrt{q}} + 111,\]
and 
\[\begin{aligned}
3 \log q + 14 |\delta| + 17 &\leq 55 + \frac{1.7\cdot 10^7}{q},\\
(\log q + 6)\cdot (1 + 5 |\delta|)&\leq 19 + \frac{1.2\cdot 10^8}{q}.
\end{aligned}\]
Hence, assuming $x\geq 10^8$ to simplify, we see that Prop.~\ref{prop:bargo}
gives us that
\[\begin{aligned}
\err_{\eta,\chi}(\delta,x) &\leq 
5.281 \cdot 10^{-22} + 
\frac{\frac{648662}{\sqrt{q}} + 111}{\sqrt{x}} + 
\frac{55 + \frac{1.7\cdot 10^7}{q}}{x} + 
\frac{19 + \frac{1.2\cdot 10^8}{q}}{x^{3/2}}\\
&\leq 5.281 \cdot 10^{-22} + \frac{1}{\sqrt{x}}
\left( \frac{650400}{\sqrt{q}} + 112\right)
\end{aligned}\]
for $\eta(t) = e^{-t^2/2}$. This proves Theorem \ref{thm:gowo1}.

Let us now see what Platt's calculations give us when used as an input to
Prop.~\ref{prop:magoma} and Cor.~\ref{cor:kolona}.
Again, we set $r=3\cdot 10^5$, $\delta_0=8$, $|\delta|\leq 4 r/q$
 and $T_0 = 10^8/q$, so 
(\ref{eq:maljust}) is still valid. We obtain
\[
T_0 \log \frac{q T_0}{2\pi} \cdot
\left( 3.5 e^{-0.1598 T_0} + 0.64
  e^{-0.1065 \cdot \frac{T_0^2}{(\pi \delta)^2}}\right)\leq
\frac{4.269\cdot 10^{-14}}{q}.\]
We use the same bound when we have $0.0019$ instead of $0.64$ on the left
side, as in (\ref{eq:monte}). (The coefficient affects what is by far
the smaller term, so we are wasting nothing.) 
Again by $T_0=10^8/q$ and $q\leq r$,
\[\begin{aligned}
1.22 \sqrt{T_0} \log q T_0 + 5.053 \sqrt{T_0} + 1.423 \log q + 37.19
&\leq \frac{275263}{\sqrt{q}} + 55.2\\
1.675 \sqrt{T_0} \log q T_0 + 6.936 \sqrt{T_0} + 1.954 \log q + 51.047
&\leq \frac{377907}{\sqrt{q}} + 75.7.
\end{aligned}\]
For $x\geq 10^8$, we use
 $|\delta|\leq 4 r/q \leq 1.2\cdot 10^6/q$ to bound
\[(3+11 |\delta|) x^{-1} + (\log q + 6) \cdot
(1 + 6 |\delta|)\cdot x^{-3/2} 
\leq \left(0.0004 + \frac{1322}{q}\right) x^{-1/2}.\]
\[
(6 + 22 |\delta|) x^{-1} + (\log q + 6) \cdot
(3 + 17 |\delta|)\cdot x^{-3/2} \leq
\left(0.0007 + \frac{2644}{q}\right) x^{-1/2}.\]
Summing, we obtain
\[\err_{\eta,\chi} \leq \frac{4.269\cdot 10^{-14}}{q} + 
\frac{1}{\sqrt{x}} \left(\frac{276600}{\sqrt{q}} + 56\right)
\]
for $\eta(t) = t^2 e^{-t^2/2}$ and
\[\err_{\eta,\chi} \leq \frac{4.269\cdot 10^{-14}}{q} + 
\frac{1}{\sqrt{x}} \left(\frac{380600}{\sqrt{q}} + 76\right)
\]
for $\eta(t) = t^2 e^{-t^2/2} \ast_M \eta_2(t)$. This proves Theorem \ref{thm:janar}
and Corollary \ref{cor:coprar}.

Now let us work with the smoothing weight $\eta_+$. 
This time around, set $r=150000$ if $q$ is odd, and
$r=300000$ if $q$ is even.
As before, we assume
\[q\leq r,\;\;\;\;\;\;\; |\delta|\leq 4 r/q.\]
We can see that Platt's verification \cite{Plattfresh}, mentioned before,
allows us to take
\[T_0 = H + \frac{250 r}{q},\;\;\;\;\; H = 200,\]
since $T_q$ is always at least this
($T_q = 10^8/q > 200 + 3.75\cdot 10^7/q$ for $q\leq 150000$ odd,
$T_q \geq 200 + 7.5\cdot 10^7/q$ for $q\leq 300000$ even).

Thus,
\[\begin{aligned}
T_0-H &\geq \frac{250 r}{r} = 250,\\
\\
\frac{T_0-H}{\pi \delta} &\geq \frac{250 r}{\pi \delta q} \geq
\frac{250}{4 \pi} = 19.89436\dotsc
\end{aligned}\]
and also
\[T_0 \leq 200 + 250\cdot 150000 \leq 3.751 \cdot 10^7
,\;\;\;\;\;\;\;\;
q T_0 \leq r H + 250 r \leq 1.35 \cdot 10^8.\]
Hence
\[\begin{aligned}
9.462 \sqrt{T_0-H} e^{-0.1598 (T_0-H)} &+ 11.287 |\delta| 
e^{-0.1065 \frac{(T_0-H)^2}{(\pi \delta)^2}}
\\ &\leq 4.2259 \cdot 10^{-17} \sqrt{\frac{250 r}{q}}
 + \frac{4 r}{q} \cdot 5.57888\cdot 10^{-18}\\
&\leq \frac{3.6598 \cdot 10^{-13}}{\sqrt{q}} 
+ \frac{6.6947 \cdot 10^{-12}}{q}.\end{aligned}\]
Examining (\ref{eq:nochpai}), we get
\[\begin{aligned}&\err_{\eta_+,\chi}(\delta,x)
\leq \log \frac{1.35\cdot 10^8}{2\pi} \cdot
 \left(\frac{3.6598 \cdot 10^{-13}}{\sqrt{q}} 
+ \frac{6.6947 \cdot 10^{-12}}{q}\right)\\
&+ \left(\left(1.631 \log \left(1.35\cdot 10^8\right)
 + 12.42\right) \sqrt{\frac{1.35\cdot 10^8}{q}} + 1.321 \log 300000 
+ 34.51\right)
x^{-\frac{1}{2}}\\
&+ \left(9+11\cdot \frac{1.2\cdot 10^6}{q}\right) x^{-1} +
(\log 300000) \left(11+ 6
\cdot \frac{1.2\cdot 10^6}{q}\right) x^{-3/2}\\
&\leq \frac{6.18\cdot 10^{-12}}{\sqrt{q}} +
\frac{1.14\cdot 10^{-10}}{q} \\ &+ 
\left(\frac{499076}{\sqrt{q}} +
51.17 + \frac{1.32\cdot 10^6}{q \sqrt{x}} + \frac{9}{\sqrt{x}} +
\frac{9.1\cdot 10^7}{q x} + \frac{139}{x}\right) \frac{1}{\sqrt{x}}
\end{aligned}\]
Making the assumption $x\geq 10^{12}$, we obtain 
\[\err_{\eta_+,\chi}(\delta,x)\leq 
\frac{6.18\cdot 10^{-12}}{\sqrt{q}} +
\frac{1.14\cdot 10^{-10}}{q} 
+ \left(\frac{499100}{\sqrt{q}} + 52\right) \frac{1}{\sqrt{x}}.\]
This proves Theorem \ref{thm:malpor} for general $q$.

Let us optimize things a little more carefully for the trivial
character $\chi_T$. Again, we will make the assumption $x\geq 10^{12}$.
We will also assume, as we did before, that $|\delta|\leq 4 r/q$; this now
gives us $|\delta|\leq 600000$, since $q=1$ and $r=150000$ for $q$ odd.
We will go up to a height $T_0 = H+600000\pi\cdot t$, where $H=200$
and $t\geq 10$. Then
\[\frac{T_0-H}{\pi \delta} = \frac{600000\pi t}{4\pi r} \geq t.\]
Hence
\[\begin{aligned}
9.462 &\sqrt{T_0-H} e^{-0.1598 (T_0 - H)} + 11.287 |\delta| e^{-0.1065
  \frac{(T_0-H)^2}{(\pi \delta)^2}}\\  &\leq 
10^{-1300000} + 6773000 e^{-0.1065
  t^2}.\end{aligned}\]
Looking at (\ref{eq:nochpai}), we get
\[\begin{aligned}
\err_{\eta_+,\chi_T}(\delta,x) &\leq \log \frac{T_0}{2\pi}
\cdot \left(10^{-1300000} + 6773000 e^{-0.1065
  t^2}\right)\\
&+ ((1.631 \log T_0 + 12.42) \sqrt{T_0} + 34.51) x^{-1/2} +
6600009 x^{-1}   .\end{aligned}\]
The value $t=20$ seems good enough; we choose it because it is not far
from optimal in the range $10^{27}\leq x\leq 10^{30}$. 
We get that $T_0 = 12000000 \pi + 200$; since $T_0<10^8$,
we are within the range of the computations in \cite{Plattfresh} (or
for that matter \cite{Wed} or \cite{PlattPi}). We obtain
\[\err_{\eta_+,\chi_T}(\delta,x) \leq 3.34 \cdot 10^{-11} +  
\frac{251100}{\sqrt{x}}.\]

Lastly, let us look at the sum estimated in (\ref{eq:horrorshow}).
Here it will be enough to go up to just $T_0 = 2 H + \max(50,H/4) = 450$,
where, as before, $H = 200$. Of course, the verification of the
zeros of the Riemann zeta function does go that far; as we already said, 
it goes until $10^8$ (or rather more: see \cite{Wed} and \cite{PlattPi}).
We make, again, the assumption $x\geq 10^{12}$.
We look at (\ref{eq:galo}) and obtain
\begin{equation}\label{eq:malko}\begin{aligned}
\err_{\ell_2,\eta_+} &\leq 
\left(0.607 \frac{(\log 450)^2}{\log 10^{12}} + 1.21 \log 450\right)
\sqrt{450} e^{-\frac{\pi}{4} \cdot 50}\\
&+ \left(2.06 \sqrt{450} \log 450 + 43.87\right) \cdot x^{-\frac{1}{2}}\\
&\leq 1.536\cdot 10^{-15} + \frac{310.84}{\sqrt{x}}.
\end{aligned}\end{equation}
It remains only to estimate the integral in (\ref{eq:horrorshow}).
First of all,
\[\begin{aligned}
\int_0^\infty &\eta_+^2(t) \log x t \;dt 
= \int_0^\infty \eta_\circ^2(t) \log x t\; dt
\\ &+ 2 \int_0^\infty (\eta_+(t)-\eta_\circ(t)) \eta_\circ(t) \log xt \;dt +
\int_0^\infty (\eta_+(t)-\eta_\circ(t))^2 \log xt \;dt.\end{aligned}\]
The main term will be given by 
\[\begin{aligned}\int_0^\infty \eta_\circ^2(t) \log x t\; dt &= 
\left(0.64020599736635 + O\left(10^{-14}\right)\right) \log x \\
&- 0.021094778698867 + O\left(10^{-15}\right),
\end{aligned}\]
where the integrals were computed rigorously using VNODE-LP \cite{VNODELP}.
(The integral $\int_0^\infty \eta_\circ^2(t) dt$ can also be computed
symbolically.) By Cauchy-Schwarz and the triangle inequality,
\[\begin{aligned}
\int_0^\infty &(\eta_+(t)-\eta_\circ(t)) \eta_\circ(t) \log xt \;dt \leq
|\eta_+-\eta_\circ|_2 |\eta_\circ(t) \log xt|_2\\
&\leq |\eta_+-\eta_\circ|_2 (|\eta_\circ|_2 \log x + |\eta_\circ\cdot \log|_2)\\
&\leq \frac{274.86}{H^{7/2}} (0.80013 \log x + 0.214)\\
&\leq 1.944 \cdot 10^{-6}\cdot \log x + 5.2\cdot 10^{-7},
\end{aligned} \]
where we are using (\ref{eq:impath}) and evaluate $|\eta_\circ \cdot \log|_2$
rigorously as above.
By (\ref{eq:impath}) and (\ref{eq:lozhka}),
\[\begin{aligned}\int_0^\infty (\eta_+(t)-\eta_\circ(t))^2 \log xt\; dt &\leq
\left(\frac{274.86}{H^{7/2}}\right)^2 \log x + \frac{27428}{H^7} \\
&\leq 5.903 \cdot 10^{-12}\cdot \log x + 2.143\cdot 10^{-12}. \end{aligned}\]
We conclude that
\begin{equation}\label{eq:kokord}\begin{aligned}
\int_0^\infty &\eta_+^2(t) \log x t \;dt \\ &= 
(0.640206 + O^*(1.95\cdot 10^{-6})) \log x -
0.021095 +O^*(5.3\cdot 10^{-7})
\end{aligned}\end{equation}
We add to this the error term $1.536\cdot 10^{-15} + 310.84/\sqrt{x}$ from
(\ref{eq:malko}), and simplify using the assumption $x\geq 10^{12}$. We obtain:
\begin{equation}\label{eq:komary}\begin{aligned}
\sum_{n=1}^\infty \Lambda(n) (\log n) \eta_+^2(n/x) &= 
0.640206 x \log x - 0.021095 x
\\ &+ O^*\left(2\cdot 10^{-6} x \log x + 310.84 \sqrt{x} \log x\right),
\end{aligned}\end{equation}
and so Prop.~\ref{prop:konechno} gives us Proposition \ref{prop:malheur}.

As we can see, 
the relatively large error term $4\cdot 10^{-6}$ comes from the fact that
we have wanted to give the main term in (\ref{eq:horrorshow})
as an explicit constant, rather than as an integral. This is satisfactory;
Prop.~\ref{prop:malheur} is an auxiliary result needed for \cite{HelfTern},
as opposed to Thms.~\ref{thm:gowo1}--\ref{thm:malpor}, which, while crucial
for \cite{HelfTern}, are also of general applicability and interest.

\appendix

\section{Extrema via bisection and truncated series}\label{subs:exbi}
In the above, we found ourselves several times in the following familiar 
situation. Let $f:I\to \mathbb{R}$, $I\subset\mathbb{R}$. We wish to find the 
minima and maxima of $f$ in $I$ numerically, but rigorously. 

(This is a situation in which a ``proof by plot'' would be convincing, but
not, of course, rigorous.)

The bisection method (as described in, e.g., \cite[\S 5.2]{MR2807595}) can be used
to show that the minimum (or maximum) of $f$ on a compact interval $I$
lies within an interval (usually a very small one). 
We will need
to complement it by other arguments if either (a) $I$ is not compact, or
(b) we want to know the minimum or maximum exactly.

As in \S \ref{subs:melltwist}, let $j(\rho) = (1+\rho^2)^{1/2}$
and $\upsilon(\rho) = \sqrt{(1+j(\rho))/2}$ for $\rho\geq 0$.
Let $\Upsilon$, $\cos \theta_0$, $\sin \theta_0$,
$c_0$ and $c_1$ be understood as 
one-variable real-valued functions on $\rho$, given by (\ref{eq:brahms}),
(\ref{eq:lokrat}) and (\ref{eq:hoxha}).

First, let us bound $\Upsilon(\rho)$ from below. By the bisection method\footnote{Implemented by the author from the description in \cite[p. 87--88]{MR2807595}, using
D. Platt's interval arithmetic package.} applied with 32 iterations,
\[0.798375987 \leq \min_{0\leq \rho\leq 10} \Upsilon(\rho) \leq 0.798375989.\]
Since $j(\rho)\geq \rho$ and $\upsilon(\rho)\geq \sqrt{j(\rho)/2} \geq
\sqrt{\rho/2}$,
\[0\leq \frac{\rho}{2 \upsilon(\rho) (\upsilon(\rho) + j(\rho))} \leq
\frac{\rho}{\sqrt{2} \rho^{3/2}} = \frac{1}{\sqrt{2 \rho}},\]
and so
\begin{equation}\label{eq:corman}
\Upsilon(\rho) \geq 1 - \frac{\rho}{2 \upsilon(\rho) (\upsilon(\rho)+ j(\rho))} \geq 1 - \frac{1}{\sqrt{2 \rho}} .\end{equation}
Hence $\Upsilon(\rho)\geq 0.8418$ for $\rho\geq 20$. We conclude that
\begin{equation}\label{eq:amigusur}
0.798375987 \leq \min_{\rho\geq 0} \Upsilon(\rho) \leq 0.798375989.
\end{equation}

Now let us bound $c_0(\rho)$ from below. 
For $\rho\geq 8$,
\[\sin \theta_0 = \sqrt{\frac{1}{2} - \frac{1}{2 \upsilon}}
\geq \sqrt{\frac{1}{2} - \frac{1}{\sqrt{2 \rho}}}\geq \frac{1}{2},\]
whereas $\cos \theta_0 \geq 1/\sqrt{2}$ for all $\rho\geq 0$. 
Hence, by (\ref{eq:amigusur})
\begin{equation}\label{eq:unu}
c_0(\rho) \geq \frac{0.7983}{\sqrt{2}} + \frac{1}{2} > 1.06\end{equation}
for $\rho\geq 8$. The bisection method applied with 28 iterations gives us that
\begin{equation}\label{eq:eka}
\max_{0.01 \leq \rho\leq 8} c_0(\rho) \geq 1 + 5\cdot 10^{-8} > 1.\end{equation}
It remains to study $c_0(\rho)$ for $\rho\in \lbrack 0,0.01\rbrack$.
 The method we are about to give actually works for all $\rho \in 
\lbrack 0,1\rbrack$.

Since
\[\begin{aligned}
\left(\sqrt{1+x}\right)'
 &= \frac{1}{2 \sqrt{1+x}},\;\;\;\;\;
\left(\sqrt{1+x}\right)'' = 
- \frac{1}{4 (1+x)^{3/2}},\\
\left(\frac{1}{\sqrt{1+x}}\right)' &= \frac{-1}{2 (1+x)^{3/2}},\;\;\;\;\;
\left(\frac{1}{\sqrt{1+x}}\right)'' = \left(\frac{-1/2}{(1+x)^{3/2}}\right)' = 
\frac{3/4}{(1+x)^{5/2}},
\end{aligned}\]
a truncated Taylor expansion gives us that, for $x\geq 0$,
\begin{equation}\label{eq:golon}\begin{aligned}
1 + \frac{1}{2} x - \frac{1}{8} x^2 \leq \sqrt{1+x} &\leq 1 + \frac{1}{2} x\\
1 - \frac{1}{2} x \leq 
\frac{1}{\sqrt{1+x}} &\leq 1 - \frac{1}{2} x + \frac{3}{8} x^2.\end{aligned}
\end{equation}
Hence, for $\rho \geq 0$,
\begin{equation}\label{eq:rencor}\begin{aligned}
1 + \rho^2/2 - \rho^4/8 &\leq j(\rho) \leq 1 + \rho^2/2,\\
1 + \rho^2/8 - 5 \rho^4/128
+\rho^6/256-\rho^8/2048 &\leq \upsilon(\rho) \leq 1 + \rho^2/8,
\end{aligned}\end{equation}
and so
\begin{equation}\label{eq:coz}
\upsilon(\rho) \geq 1 + \rho^2/8 - 5 \rho^4/128\end{equation}
for $\rho\leq 8$. We also get from (\ref{eq:golon}) that
\begin{equation}\label{eq:emo}\begin{aligned}
\frac{1}{\upsilon(\rho)} &= \frac{1}{\sqrt{1 + \frac{j(\rho)-1}{2}}} \leq
1 - \frac{1}{2} \frac{j(\rho)-1}{2} + \frac{3}{8} \left(\frac{j(\rho)-1}{2}
\right)^2\\ &\leq 1 - \frac{1}{2} \left(\frac{\rho^2}{4} - \frac{\rho^4}{16}
\right) + \frac{3}{8} \frac{\rho^4}{16} \leq 
1 - \frac{\rho^2}{8} + \frac{7 \rho^4}{128},\\
\frac{1}{\upsilon(\rho)} &= \frac{1}{\sqrt{1 + \frac{j(\rho)-1}{2}}} \geq
1 - \frac{1}{2} \frac{j(\rho)-1}{2} \geq 1 - \frac{\rho^2}{8} .
\end{aligned}\end{equation}
Hence
\begin{equation}\label{eq:sindo}\begin{aligned}
\sin \theta_0 &= \sqrt{\frac{1}{2} - \frac{1}{2 \upsilon(\rho)}} \geq
\sqrt{\frac{\rho^2}{16} - \frac{7 \rho^4}{256}} = 
\frac{\rho}{4} \sqrt{1 - \frac{7}{16} \rho^2},\\
\sin \theta_0 &\leq \sqrt{\frac{\rho^2}{16}} = \frac{\rho}{4},
\end{aligned}\end{equation}
while
\begin{equation}\label{eq:cosdo}\begin{aligned}
\cos \theta_0 = \sqrt{\frac{1}{2} + \frac{1}{2 \upsilon(\rho)}}
\geq \sqrt{1 - \frac{\rho^2}{16}},\;\;\;\;\;\;
\cos \theta_0 \leq \sqrt{1 - \frac{\rho^2}{16} + \frac{7 \rho^4}{256}},
\end{aligned}\end{equation}

By (\ref{eq:rencor}) and (\ref{eq:emo}),
\begin{equation}\label{eq:tango}
\frac{\rho}{2 \upsilon (\upsilon+j)} \geq \frac{\rho}{2} \frac{1 - \rho^2/8}{
2+5 \rho^2/8} \geq \frac{\rho}{2} \left(\frac{1}{2} - \frac{3 \rho^2}{32}
\right) = 
\frac{\rho}{4} - \frac{3 \rho^3}{64}.
\end{equation}
Assuming $0\leq \rho\leq 1$,
\[\frac{1}{1 + \frac{5 \rho^2}{16} - 
\frac{9 \rho^4}{64}} \leq 
\left(1 - \frac{5 \rho^2}{16} + \frac{9 \rho^4 }{64} + 
\left(\frac{5 \rho^2}{16} - \frac{9 \rho^4}{64}\right)^2\right)
\leq 1 - \frac{5 \rho^2}{16} + \frac{61 \rho^4}{256},\]
and so, by (\ref{eq:coz}) and (\ref{eq:emo}),
\[\begin{aligned}
\frac{\rho}{2 \upsilon (\upsilon+j)} &\leq \frac{\rho}{2}
\frac{1 - \frac{\rho^2}{8} + \frac{7 \rho^4}{128}}{2 + \frac{5 \rho^2}{8} - 
\frac{21 \rho^4}{128}} \\ &\leq \frac{\rho}{4}
\left(1 - \frac{\rho^2}{8} + \frac{7 \rho^4}{128}\right)
\left(1 - \frac{5 \rho^2}{16} + \frac{46 \rho^4 }{256}\right)\\
&\leq \frac{\rho}{4} \left( 1 - \frac{7 \rho^2}{16} + 
\frac{35}{128} \rho^4 - \frac{81}{2048} \rho^6 + \frac{161}{2^{14}} \rho^8 \right)
\leq \frac{\rho}{4} - \frac{7 \rho^3}{64} + \frac{35 \rho^5}{512}.
\end{aligned}\]

Hence, we obtain
\begin{equation}\label{eq:myas}\begin{aligned}
\Upsilon(\rho) &= \sqrt{1 + \left(\frac{\rho}{2 \upsilon (\upsilon + j)}
\right)^2}
 - \frac{\rho}{2 \upsilon (\upsilon + j)} \\ &\geq
1 + \frac{1}{2} \left(\frac{\rho}{4} - \frac{3 \rho^3}{64}
\right)^2
 - \frac{1}{8} \left(\frac{\rho^2}{16}\right)^2  - 
\left(\frac{\rho}{4} - \frac{7 \rho^3}{64} + \frac{35 \rho^5}{512}\right)\\
&\geq 1 - \frac{\rho}{4} + \frac{\rho^2}{32} + \frac{7 \rho^3}{64} 
- 
\left(\frac{3}{256} + 
\frac{1}{2048}\right) \rho^4 - \frac{35 \rho^5}{512} +
\frac{9 \rho^6}{ 2^{13}}\\
&\geq
1 - \frac{\rho}{4} + \frac{\rho^2}{32} + \frac{7 \rho^3}{64} 
- 
\frac{165 \rho^4}{2048},\end{aligned}\end{equation}
where, in the last line, we use again the assumption $\rho\leq 1$.

For $x\in \lbrack -1/4,0\rbrack$,
\[\begin{aligned}
\sqrt{1+x} &\geq 1 + \frac{1}{2} x - \frac{x^2}{2} \frac{1}{4 (1 -1/4)^{3/2}}
= 1 + \frac{x}{2} - \frac{x^2}{3^{3/2}}\\
\sqrt{1+x} &\leq 1 + \frac{1}{2} x - \frac{x^2}{8} \leq 1 + \frac{1}{2} x.
\end{aligned}\]
Hence  
\begin{equation}\label{eq:synod}\begin{aligned}
1 - \frac{\rho^2}{32} - \frac{\rho^4}{3^{3/2}\cdot 256} &\leq
\cos \theta_0 \leq \sqrt{1 - \frac{\rho^2}{16} + \frac{7 \rho^4}{256}}
\leq 1 - \frac{\rho^2}{32} + \frac{7 \rho^4}{512}\\
\frac{\rho}{4}\left( 1 - \frac{7}{32} \rho^2 -
\frac{49}{3^{3/2}\cdot 256} \rho^4\right) &\leq
\sin \theta_0 \leq \frac{\rho}{4}
\end{aligned}\end{equation}
for $\rho\leq 1$.
Therefore,
\[\begin{aligned}
c_0(\rho) &= \Upsilon(\rho)\cdot \cos \theta_0 + \sin \theta_0 \\
&\geq \left(1 - \frac{\rho}{4} + \frac{\rho^2}{32} 
+ \frac{7 \rho^3}{64}
- \frac{165}{2048} \rho^4 
\right) \left(1 - \frac{\rho^2}{32} - \frac{\rho^4}{3^{3/2}\cdot 256}\right) 
\\ &+
\frac{\rho}{4} - \frac{7}{128} \rho^3 -
\frac{49}{3^{3/2}\cdot 1024} \rho^5\\
&\geq 1 + \frac{ \rho^3}{16} - 
\left(\left(\frac{\sqrt{3}}{2304} + \frac{167}{2048}\right) +
\left(\frac{7}{2048} + \frac{\sqrt{3}}{192}\right) \rho +
\frac{7 \sqrt{3}}{147456} \rho^3\right) \rho^4\\
&\geq 1 + \frac{\rho^3}{16} - 0.0949 \rho^4,
\end{aligned}\]
where we are again using $\rho\leq 1$. 
We conclude that, for all $\rho\in (0,1/2\rbrack$,
\[c_0(\rho) > 1.\]
Together with (\ref{eq:unu}) and (\ref{eq:eka}), this shows that
\begin{equation}\label{eq:pekka}
c_0(\rho) >1 \;\;\;\;\;\; \forall \rho>0.
\end{equation}
It is easy to check that $c_0(0) = 1$.

(The truncated-Taylor series estimates above could themselves have been done 
automatically; see \cite[Ch. 4]{MR2807595} (automatic differentiation). 
The footnote in \cite[p. 72]{MR2807595} (referring to the work of Berz
and Makino \cite{MR1652147} on ``Taylor models'') seems particularly
relevant here. We have preferred to do matters ``by hand'' in the above.)

Now let us examine $\eta(\rho)$, given as in (\ref{eq:malina}).
Let us first focus on the case of $\rho$ large. We can use
the lower bound (\ref{eq:corman}) on $\Upsilon(\rho)$. To obtain a good upper
bound on $\Upsilon(\rho)$, we need to get truncated series expansions
on $1/\rho$ for $\upsilon$ and $j$. These are:
\begin{equation}\label{eq:ciano}\begin{aligned}
j(\rho) &= \sqrt{\rho^2 + 1} = \rho \sqrt{1 + \frac{1}{\rho^2}} \leq
\rho \left(1 + \frac{1}{2 \rho^2}\right) = \rho + \frac{1}{2 \rho},\\
\upsilon(\rho) &= \sqrt{\frac{1+j}{2}} \leq
\sqrt{\frac{\rho}{2} + \frac{1}{2} + \frac{1}{4 \rho}} = 
\sqrt{\frac{\rho}{2}} \sqrt{1 + \frac{1}{\rho} + \frac{1}{2 \rho^2}}
\leq \sqrt{\frac{\rho}{2}} \left(1 + \frac{1}{\sqrt{2} \rho}\right),
\end{aligned}\end{equation}
together with the trivial bounds $j(\rho)\geq \rho$ and $\upsilon(\rho)
\geq \sqrt{j(\rho)/2} \geq \sqrt{\rho/2}$.
By (\ref{eq:ciano}),
\begin{equation}\label{eq:juro}\begin{aligned}
\frac{1}{\upsilon^2-\upsilon} &\geq 
\frac{1}{\frac{\rho}{2} \left(1 + \frac{1}{\sqrt{2} \rho}\right)^2
 - \sqrt{\frac{\rho}{2}}} =
\frac{\left(1 + \sqrt{\frac{2}{\rho}}\right)}{
\frac{\rho}{2} \left(\left(1 + \frac{1}{\sqrt{2} \rho}\right)^2 - 
\sqrt{\frac{2}{\rho}}
\right) \left(1 + \sqrt{\frac{2}{\rho}}\right)}\\
&= 
\frac{\frac{2}{\rho} \left(1 + \sqrt{\frac{2}{\rho}}\right)}{1 - 
\frac{2}{\rho} + \left(\frac{\sqrt{2}}{\rho} + \frac{1}{2 \rho^2}\right)
\left(1 + \sqrt{\frac{2}{\rho}}\right)}
\geq \frac{2}{\rho} + \frac{\sqrt{8}}{\rho^{3/2}}
\end{aligned}\end{equation}
for $\rho\geq 15$, and so
\begin{equation}\label{eq:jurt}
\frac{j}{\upsilon^2-\upsilon} \geq 2 + \sqrt{\frac{8}{\rho}}
\end{equation}
for $\rho\geq 15$. In fact, the bisection method (applied with $20$
iterations, including $10$ ``initial'' iterations after which the
possibility of finding a minimum within each interval is tested)
shows that (\ref{eq:juro}) (and hence (\ref{eq:jurt}))
holds for all $\rho\geq 1$. By (\ref{eq:ciano}),
\begin{equation}\label{eq:taita}
\begin{aligned}\frac{\rho}{2 \upsilon (\upsilon + j)} &\geq
\frac{\rho}{\sqrt{2 \rho} \left(1 + \frac{1}{\sqrt{2} \rho}\right)
\left(\sqrt{\frac{\rho}{2}} \left(1 + \frac{1}{\sqrt{2} \rho}\right) 
+ \rho + \frac{1}{2 \rho}\right)} \\ &\geq
\frac{1}{\sqrt{2 \rho}} \cdot
\frac{1}{1 + \frac{1}{\sqrt{2 \rho}} + \frac{1}{\rho}} \geq
\frac{1}{\sqrt{2 \rho}} - \frac{1}{2\rho} - \frac{1}{\sqrt{2} \rho^{3/2}}
 \end{aligned}\end{equation}
for $\rho\geq 16$. (Again, (\ref{eq:taita}) is also true for $1\leq \rho\leq
16$ by the bisection method; it is trivially true for $\rho\in \lbrack 0,1
\rbrack$, since the last term of (\ref{eq:taita}) is then negative.)
We also have the easy upper bound
\begin{equation}\label{eq:perrit}
\frac{\rho}{2 \upsilon (\upsilon + j)} \leq
\frac{\rho}{2\cdot \sqrt{\frac{\rho}{2}} \cdot (\sqrt{\frac{\rho}{2}}+\rho)} =
\frac{1}{\sqrt{2 \rho} + 1} \leq \frac{1}{\sqrt{2 \rho}}
- \frac{1}{2 \rho} + \frac{1}{(2 \rho)^{3/2}}\end{equation}
valid for $\rho\geq 1/2$.

Hence, by (\ref{eq:golon}), (\ref{eq:taita}) and (\ref{eq:perrit}),
\[\begin{aligned}
\Upsilon&= \sqrt{1 + \left(
\frac{\rho}{2 \upsilon (\upsilon + j)}\right)^2} - 
\frac{\rho}{2 \upsilon (\upsilon + j)}\\ &\leq 1 + 
\frac{1}{2} \left(\frac{1}{\sqrt{2 \rho}} - \frac{1}{2 \rho}\right)^2 - 
\frac{1}{\sqrt{2 \rho}} + \frac{1}{2 \rho} + \frac{1}{\sqrt{2} \rho^{3/2}}
\leq 1 - \frac{1}{\sqrt{2 \rho}} + \frac{1}{\rho}
\end{aligned}\]
for $\rho\geq 3$. Again, we use the bisection method (with $20$ iterations)
on $\lbrack 1/2,3 \rbrack$, and note that $1/\sqrt{2\rho} < 1/\rho$ for $\rho<1/2$;
we thus obtain
\begin{equation}\label{eq:gorjeo}
\Upsilon \leq 1 - \frac{1}{\sqrt{2 \rho}} + \frac{1}{\rho}\end{equation}
for all $\rho>0$.

We recall (\ref{eq:malina}) and the lower bounds
(\ref{eq:jurt}) and (\ref{eq:corman}). We get
\begin{equation}\label{eq:pimpin}\begin{aligned}
\eta &\geq \frac{1}{\sqrt{2}} \sqrt{2 + \sqrt{\frac{8}{\rho}}}
\left(1 + \left(1 - \frac{1}{\sqrt{2 \rho}}\right)^2\right) - \frac{1}{2}
\left(1 - \frac{1}{\sqrt{2 \rho}} + \frac{1}{\rho}\right)^2\\ &+
\frac{1}{2} - \frac{1}{\sqrt{2 \rho}}
- \frac{\rho}{\rho+1} \cdot \left(1 - \frac{1}{\sqrt{2 \rho}} + \frac{1}{\rho}
\right)\\
&\geq \left(1 + \frac{1}{\sqrt{2 \rho}} - \frac{1}{4 \rho}\right)
\left(2 - \frac{\sqrt{2}}{\sqrt{\rho}} + \frac{1}{2 \rho}\right) -
\frac{1}{2} \left(1 - \frac{2}{\sqrt{2 \rho}} + 
\frac{5}{2 \rho}\right)\\
&+ \frac{1}{2} - \frac{1}{\sqrt{2 \rho}} - \left(1 - \rho^{-1} + \rho^{-2}
\right)\left(1 - \frac{1}{\sqrt{2 \rho}} + \frac{1}{\rho}
\right)\\
&\geq 1 + \frac{1}{\sqrt{2 \rho}} - \frac{9}{4 \rho} - \frac{1}{8 \rho^2} +
\frac{1}{\sqrt{2} \rho^{5/2}} - \frac{1}{\rho^3}
\geq 1 + \frac{1}{\sqrt{2 \rho}} - \frac{37}{16 \rho}
\end{aligned}\end{equation}
for $\rho\geq 2$. This implies that $\eta(\rho)>1$ for $\rho\geq 11$.
(Since our estimates always give an error of at most $O(1/\sqrt{\rho})$,
we also get $\lim_{\rho\to \infty} \eta(\rho) = 1$.)
The bisection method (with $20$ iterations, including $6$ initial iterations)
gives that $\eta(\rho)>1$ also holds for
$1\leq \rho\leq 11$.

Let us now look at what happens for $\rho\leq 1$. From (\ref{eq:myas}),
we get the simpler bound
\begin{equation}\label{eq:cort}
\Upsilon \geq 1 - \frac{\rho}{4} + \frac{\rho^2}{32} +
\frac{3 \rho^3}{32} \geq
1 - \frac{\rho}{4}\end{equation}
valid for $\rho\leq 1$, implying that
\[\Upsilon^2 \geq 1 - \frac{\rho}{2} + \frac{\rho^2}{8} +
\frac{11 \rho^3}{64} - \frac{23 \rho^4}{1024}\]
for $\rho\leq 1$. 
We also have, by (\ref{eq:brahms}) and (\ref{eq:tango}),
\begin{equation}\label{eq:caliph}\begin{aligned}
\Upsilon &\leq 1 + \frac{1}{2} \left(\frac{\rho}{2 \upsilon (\upsilon + j)}
\right)^2 - \frac{\rho}{2 \upsilon (\upsilon + j)} \leq
1 + \frac{1}{2} \left(\frac{\rho}{4}\right)^2 - 
\left(\frac{\rho}{4} - \frac{3 \rho^3}{64}\right)\\
&\leq 1 - \frac{\rho}{4} + \frac{\rho^2}{32} + \frac{3 \rho^3}{64} 
\leq 1 - \frac{\rho}{4} + \frac{5 \rho^2}{64}
\end{aligned}\end{equation}
for $\rho\leq 1$. (This immediately implies the
easy bound $\Upsilon\leq 1$, which follows anyhow from
(\ref{eq:horem}) for all $\rho\geq 0$.)

By (\ref{eq:rencor}),
\[\begin{aligned}
\frac{j}{\upsilon^2-\upsilon} &\geq
\frac{1 + \rho^2/2-\rho^4/8}{\left(1 + \frac{\rho^2}{8}\right)^2 -
\left(1 + \frac{\rho^2}{8}\right)} \geq
\frac{1 + \rho^2/2-\rho^4/8}{\frac{\rho^2}{8} + \frac{\rho^4}{64}}
 \geq \frac{8}{\rho^2}
\end{aligned}\]
for $\rho\leq 1$.
Therefore, by (\ref{eq:malina}),
\[\begin{aligned}
\eta &\geq \frac{1}{\sqrt{2}} \sqrt{\frac{8}{\rho^2}} 
\left(2 - \frac{\rho}{2} + \frac{\rho^2}{8} + \frac{11 \rho^3}{64} - 
\frac{ 3 \rho^4}{128}
\right) - \frac{1}{2} \left(1 - \frac{\rho}{4} + \frac{5 \rho^2}{64}
\right)^2
+ \frac{1}{2} - 
\frac{1}{2} - \frac{\rho}{2}\\
&\geq \frac{4}{\rho} - 1 + \frac{\rho}{4} 
+ \frac{11 \rho^2}{32} - \frac{3 \rho^3}{64}
- \frac{1}{2} 
\left(1 - \frac{\rho}{2} + \frac{7 \rho^2}{32}\right)
- \frac{\rho}{2}
\geq \frac{4}{\rho} - \frac{3}{2}  + \frac{15 \rho^2}{64}
- \frac{3 \rho^3}{64}\\ &\geq \frac{4}{\rho}-\frac{3}{2}
\end{aligned}\]
for $\rho\leq 1$. This implies the bound $\eta(\rho)> 1$ for all $\rho\leq 1$.
Conversely, $\eta(\rho)\geq 4/\rho - 3/2$ follows from $\eta(\rho)>1$ for
$\rho>8/5$. We check $\eta(\rho)\geq 4/\rho - 3/2$ for $\rho \in 
\lbrack 1,8/5\rbrack$ by the bisection method ($5$ iterations).

We conclude that, for all $\rho>0$,
\begin{equation}\label{eq:borneta}
\eta \geq \max\left(1, \frac{4}{\rho}-\frac{3}{2}\right).\end{equation}
This bound has the right asymptotics for $\rho\to 0^+$ and $\rho\to +\infty$.

Let us now bound $c_0$ from above. By (\ref{eq:synod}) and (\ref{eq:caliph}),
\begin{equation}\label{eq:lolm}\begin{aligned}
c_0(\rho) &= \Upsilon(\rho)\cdot \cos \theta_0 + \sin \theta_0 \leq
\left(1 - \frac{\rho}{4} + \frac{5 \rho^2}{64}\right)
\left(1 - \frac{\rho^2}{32} + \frac{7 \rho^4}{512}\right) + \frac{\rho}{4}\\
&\leq 1 + \frac{3 \rho^2}{64} + \frac{\rho^3}{128} + \frac{23 \rho^4}{2048}
- \frac{7 \rho^5}{2048} + \frac{35 \rho^6}{2^{15}}
\leq 1 + \frac{\rho^2}{15}
\end{aligned}\end{equation}
for $\rho\leq 1$. Since $\Upsilon\leq 1$ and
$\theta_0 \in \lbrack 0,\pi/4\rbrack \subset \lbrack 0,\pi/2\rbrack$, the
bound
\begin{equation}\label{eq:malc}
c_0(\rho) \leq \cos \theta_0 + \sin \theta_0 \leq \sqrt{2}\end{equation}
holds for all $\rho\geq 0$. By (\ref{eq:gorjeo}), we also know that,
for $\rho\geq 2$,
\begin{equation}\label{eq:penti}\begin{aligned}
c_0(\rho) &\leq \left(1 - \frac{1}{\sqrt{2 \rho}} + \frac{1}{\rho}\right)
\cos \theta_0 + \sin \theta_0\\
&\leq \sqrt{\left(1 - \frac{1}{\sqrt{2 \rho}} + \frac{1}{\rho}\right)^2 +
  1}
\leq \sqrt{2} \left(1 - \frac{1}{2 \sqrt{2 \rho}} + \frac{9}{16 \rho}\right).
\end{aligned}\end{equation}

From (\ref{eq:borneta}) and (\ref{eq:malc}), we obtain that
\begin{equation}\label{eq:hoj}
\frac{1}{\eta} \left(1 + 2 c_0^2\right) 
\leq 1\cdot (1 + 2\cdot 2) = 5\end{equation}
for all $\rho\geq 0$. At the same time, (\ref{eq:borneta}) and
(\ref{eq:lolm}) imply that
\[\begin{aligned}
\frac{1}{\eta} \left(1 + 2 c_0^2\right) 
&\leq \left(\frac{4}{\rho} - \frac{3}{2}\right)^{-1} 
\left(3 + \frac{4 \rho^2}{15} + \frac{2 \rho^4}{15^2}\right)\\ &=
\frac{3 \rho}{4} \left(1 - \frac{3\rho }{8}\right)^{-1}
\left(1 + \frac{4 \rho^2}{45} + \frac{\rho^4}{675}\right) \leq
\frac{3 \rho}{4} \left(1 + \frac{\rho}{2}\right)
\end{aligned}\]
for $\rho\leq 0.4$. Hence $(1+2 c_0^2)/\eta \leq 0.86 \rho$ for
$\rho<0.29$. The bisection method ($20$ iterations, starting
by splitting the range into $2^8$ equal intervals) shows that
$(1+2 c_0^2)/\eta \leq 0.86 \rho$ also holds for $0.29\leq \rho\leq 6$;
for $\rho>6$, the same inequality holds by (\ref{eq:hoj}).

We have thus shown that
\begin{equation}\label{eq:suit}
\frac{1 + 2 c_0^2}{\eta} \leq \min(5,0.86 \rho)
\end{equation}
for all $\rho>0$.

Now we wish to bound $\sqrt{(\upsilon^2-\upsilon)/2}$ from below.
By (\ref{eq:coz}) and (\ref{eq:rencor}),
\begin{equation}\label{eq:autom}\begin{aligned}
\upsilon^2 - \upsilon &\geq \left(1 + \frac{\rho^2}{8} - \frac{5 \rho^4}{128}
\right)^2 - \left(1 + \frac{\rho^2}{8}\right)\\
&= 1 + \frac{\rho^2}{4} - \frac{5 \rho^4}{64} +
\left(\frac{5 \rho^2}{128} - \frac{1}{8}\right)^2 \rho^4
- \left(1 + \frac{\rho^2}{8}\right)\geq \frac{\rho^2}{8} - \frac{5 \rho^4}{64},\end{aligned}\end{equation}
for $\rho\geq 1$, and so 
\[
\sqrt{\frac{\upsilon^2-\upsilon}{2}} \geq \frac{\rho}{4} \sqrt{1 - 
\frac{ 5 \rho^2}{8}},
\]
and this is greater than $\rho/6$ for $\rho\leq 1/3$.
The bisection method ($20$ iterations, $5$ initial steps) confirms that
$\sqrt{(\upsilon^2-\upsilon)/2} > \rho/6$ also holds for $2/3<\rho\leq 4$.
On the other hand, by (\ref{eq:ciano}) and 
$\upsilon^2 = (1+j)/2 \geq (1+\rho)/2$,
\begin{equation}\label{eq:krich}\begin{aligned}
\sqrt{\frac{\upsilon^2 - \upsilon}{2}} &\geq
\sqrt{\frac{\frac{\rho+1}{2} 
- \sqrt{\frac{\rho}{2}} \left(1 + \frac{1}{\sqrt{2} \rho}
\right)}{2}}\geq \frac{\sqrt{\rho}}{2}
\sqrt{1 + \frac{1}{\rho} - \sqrt{\frac{2}{\rho}} \left(1 + \frac{1}{\sqrt{2} \rho}\right)}\\
&\geq \frac{\sqrt{\rho}}{2} \sqrt{1 - \sqrt{\frac{2}{\rho}} + \frac{1}{2 \rho}}
\geq \frac{\sqrt{\rho}}{2} \left(1 - \sqrt{\frac{1}{2 \rho}}\right) =
\frac{\sqrt{\rho}}{2} - \frac{1}{2^{3/2}}
\end{aligned}\end{equation}
for $\rho\geq 4$. We check by the bisection method ($20$ iterations)
that $\sqrt{(\upsilon^2-\upsilon)/2} \geq
\sqrt{\rho}/2 - 1/2^{3/2}$ also holds for all $0\leq \rho\leq 4$.

We conclude that
\begin{equation}\label{eq:frais}
\sqrt{\frac{\upsilon^2-\upsilon}{2}} \geq \begin{cases}
\rho/6 &\text{if $\rho\leq 4$,}\\
\frac{\sqrt{\rho}}{2} - \frac{1}{2^{3/2}}  &\text{for all $\rho$.}
\end{cases}\end{equation}

We still have a few other inequalities to check. Let us first 
derive an easy lower bound on $c_1(\rho)$ for $\rho$ large:
by (\ref{eq:corman}), (\ref{eq:juro}) and (\ref{eq:golon}),
\[\begin{aligned}
c_1(\rho) &= \sqrt{\frac{1+1/\upsilon}{\upsilon^2-\upsilon}}\cdot \Upsilon
\geq \sqrt{\frac{1}{\upsilon^2-\upsilon}}\cdot \left(1 - \frac{1}{\sqrt{2 \rho}}\right)
\geq \sqrt{\frac{2}{\rho} + \frac{\sqrt{8}}{\rho^{3/2}}}
\cdot \left(1 - \frac{1}{\sqrt{2 \rho}}\right) \\ &= 
\sqrt{\frac{2}{\rho}} \left(1 + \frac{1}{\sqrt{2\rho}} - 
\frac{1}{4 \rho}\right) \cdot
\left(1 - \frac{1}{\sqrt{2 \rho}}\right) \geq 
\sqrt{\frac{2}{\rho}} \left(1 - \frac{3}{4 \rho}\right) 
\end{aligned}\]
for $\rho\geq 1$. Together with (\ref{eq:penti}), this implies that,
for $\rho\geq 2$,
\[\frac{c_0-1/\sqrt{2}}{\sqrt{2} c_1 \rho} \leq 
\frac{\sqrt{2} \left(\frac{1}{2} - \frac{1}{2 \sqrt{2 \rho}} + \frac{9}{16 \rho}\right)}{
\sqrt{2} \rho \sqrt{\frac{2}{\rho}} \left(1 - \frac{3}{4 \rho}\right)}
= \frac{1}{\sqrt{2 \rho}} \cdot 
\frac{
1 - \frac{1}{\sqrt{2 \rho}} + \frac{9}{8 \rho}}{2 \left(1 - \frac{3}{4 \rho}\right)},
\]
again for $\rho\geq 1$. This is 
$\leq 1/\sqrt{8 \rho}$ for $\rho \geq 8$. Hence it is $\leq 
1/\sqrt{8\cdot 25} < 0.071$ for $\rho \geq 25$.

Let us now look at $\rho$ small. By (\ref{eq:rencor}),
\[\upsilon^2 - \upsilon \leq \left(1 + \frac{\rho^2}{8}\right)^2 -
\left(1 + \frac{\rho^2}{8} - \frac{5 \rho^4}{32}\right) = 
\frac{\rho^2}{8} + \frac{9 \rho^4}{32}\]
for any $\rho>0$. Hence, by (\ref{eq:emo}) and (\ref{eq:cort}),
\[\begin{aligned}
c_1(\rho) &= \sqrt{\frac{1+1/\upsilon}{\upsilon^2-\upsilon}}\cdot \Upsilon
\geq \sqrt{\frac{2 - \rho^2/8}{\frac{\rho^2}{8} + \frac{9 \rho^4}{32}}}
\cdot \left(1 - \frac{\rho}{4}\right) \geq \frac{4}{\rho}
\left(1 - \frac{5}{4} \rho^2\right) \left(1 - \frac{\rho}{4}\right),
\end{aligned}\]
whereas, for $\rho\leq 1$,
\[c_0(\rho) = \Upsilon(\rho)\cdot \cos \theta_0 + \sin \theta_0 \leq
1 + \sin \theta_0 \leq 1 + \rho/4\]
by (\ref{eq:synod}). Thus
\[
\frac{c_0-1/\sqrt{2}}{\sqrt{2} c_1 \rho} \leq 
\frac{1 + \frac{\rho}{4} - \frac{1}{\sqrt{2}}}{\sqrt{2} \cdot 4 
\left(1 - \frac{5}{4} \rho^2\right) \left(1 - \frac{\rho}{4}\right)}
\leq 0.0584 
\]
for $\rho\leq 0.1$.
 We check the remaining interval $\lbrack 0.1,25\rbrack$ (or
$\lbrack 0.1,8\rbrack$, if we aim at the bound $\leq 1/\sqrt{8 \rho}$)
by the bisection method (with $24$ iterations, including $12$ initial
iterations -- or $15$ iterations and $10$ initial iterations, in the case of
$\lbrack 0.1,8\rbrack$) and obtain that
\begin{equation}\label{eq:anka}\begin{aligned}
0.0763895 \leq
\max_{\rho\geq 0} \frac{c_0-1/\sqrt{2}}{\sqrt{2} c_1 \rho} &\leq
0.0763896\\
\sup_{\rho\geq 0} \frac{c_0-1/\sqrt{2}}{c_1 \sqrt{\rho}}
&\leq \frac{1}{2}.\end{aligned}\end{equation}

In the same way, we see that
\[\frac{c_0}{c_1 \rho} \leq \frac{1}{\sqrt{\rho}} \frac{1}{1 -
  \frac{3}{4\rho}} \leq 0.171\]
for $\rho\geq 36$ and
\[\frac{c_0}{c_1 \rho} \leq 
\frac{1 + \frac{\rho}{4}}{4 
\left(1 - \frac{5}{4} \rho^2\right) \left(1 - \frac{\rho}{4}\right)}
\leq 0.267\]
for $\rho\leq 0.1$. The bisection method applied to $\lbrack 0.1,36\rbrack$
with $24$ iterations (including $12$ initial iterations) now gives
\begin{equation}\label{eq:temek}
0.29887 \leq \max_{\rho>0} \frac{c_0}{c_1 \rho} \leq 0.29888 .
\end{equation}

We would also like a lower bound for $c_0/c_1$. For $c_0$, we can use the lower
bound $c_0\geq 1$ given by (\ref{eq:pekka}). By (\ref{eq:emo}), 
(\ref{eq:caliph}) and
(\ref{eq:autom}),
\[\begin{aligned}
c_1(\rho) &= \sqrt{\frac{1 + 1/\upsilon}{\upsilon^2-\upsilon}} \cdot 
\Upsilon \leq \sqrt{\frac{2 - \frac{\rho^2}{8} + \frac{7 \rho^4}{128}}{\rho^2/8 - 5 \rho^4/64}}
\cdot \left(1 - \frac{\rho}{4} + \frac{5 \rho^2}{64}\right)\\
&\leq \frac{4}{\rho} 
\left(1 + \frac{5 \rho^2}{16}\right) \left(1 - \frac{\rho}{4} +
\frac{5 \rho^2}{64}\right) < \frac{4}{\rho}
\end{aligned}\]
for $\rho \leq 1/4$. Thus, $c_0/(c_1 \rho) \geq 1/4$ for $\rho \in \lbrack 0,1/4
\rbrack$.
The bisection method (with $20$ iterations, including $10$ initial iterations) 
gives us that 
$c_0/(c_1 \rho) \geq 1/4$ also holds for $\rho\in \lbrack 1/4,6.2\rbrack$.
Hence
\[\frac{c_0}{c_1} \geq \frac{\rho}{4}\]
for $\rho\leq 6.2$.

Now consider the case of large $\rho$. By 
and $\Upsilon\leq 1$,
\begin{equation}\label{eq:meloma}
\frac{c_0}{c_1 \sqrt{\rho}} \geq \frac{1/\Upsilon}{\sqrt{\frac{1+1/\upsilon}{
\upsilon^2-\upsilon}}\cdot \sqrt{\rho}} \geq \frac{\sqrt{(\upsilon^2 - \upsilon)/
\rho}}{\sqrt{1+1/\upsilon}} \geq 
\frac{1}{\sqrt{2}} \frac{1 - 1/\sqrt{2\rho}}{\sqrt{1+1/\upsilon}}.\end{equation}
(This is off from optimal by a factor of about $\sqrt{2}$.) 
For $\rho \geq 200$, (\ref{eq:meloma}) implies that
$c_0/(c_1 \sqrt{\rho}) \geq 0.6405$. The bisection method
(with $20$ iterations, including $5$ initial iterations) gives us
$c_0/(c_1 \sqrt{\rho}) \geq 5/8 = 0.625$ for $\rho \in \lbrack 6.2,200\rbrack$.
We conclude that
\begin{equation}\label{eq:ruwo}
\frac{c_0}{c_1} \geq \min\left(\frac{\rho}{4}, \frac{5}{8} \sqrt{\rho}\right).
\end{equation}

Finally, we verify an inequality that will be useful for the estimation
a crucial exponent in Thm.~\ref{thm:princo}.
We wish to show that, for all $\alpha\in \lbrack 0,\pi/2\rbrack$,
\begin{equation}\label{eq:antiman}
\alpha - \frac{\sin 2 \alpha}{4 \cos^2 \frac{\alpha}{2}} \geq
\frac{\sin \alpha}{2 \cos^2 \alpha} 
- \frac{5 \sin^3 \alpha}{ 24 \cos^6 \alpha}
\end{equation}
The left side is positive for all $\alpha\in (0,\pi/2\rbrack$, since
$\cos^2 \alpha/2 \geq 1/\sqrt{2}$ and $(\sin 2 \alpha)/2$ is less than
$2 \alpha/2 = \alpha$.
The right side is negative for $\alpha>1$ (since it is negative for $\alpha=1$,
and $(\sin \alpha)/(\cos \alpha)^2$ is increasing on $\alpha$). 
Hence, it is enough to check (\ref{eq:antiman}) for $\alpha \in 
\lbrack 0,1\rbrack$. The two sides of (\ref{eq:antiman}) are equal for $\alpha=0$;
moreover, the first four derivatives also match at $\alpha=0$.
We take the fifth derivatives of both sides; the bisection method (running
on $\lbrack 0,1\rbrack$ with $20$ iterations, including $10$ initial iterations)
gives us that the fifth derivative of the left side minus the fifth derivative
of the right side is always positive on $\lbrack 0,1\rbrack$ (and minimal
at $0$, where it equals $30.5+O^*\left(10^{-9}\right)$).

\section{Norms of smoothing functions}\label{app:norsmo}

Our aim here is to give bounds on the norms of some smoothing functions --
and, in particular, on several norms of a smoothing function $\eta_+:\lbrack 0,\infty)\to \mathbb{R}$ based on the Gaussian $\eta_\heartsuit(t) = e^{-t^2/2}$.

As before, we write
\begin{equation}\label{eq:clager}
h:t\mapsto \begin{cases}
t^2 (2-t)^3 e^{t-1/2} &\text{if $t\in \lbrack 0,2\rbrack$,}\\ 0 &\text{otherwise}\end{cases}
\end{equation}
We recall that we will work with an approximation $\eta_+$ to the function
$\eta_\circ:\lbrack 0,\infty)\to \mathbb{R}$ defined by
\begin{equation}\label{eq:cleo}
\eta_\circ(t) = h(t) \eta_\heartsuit(t) = \begin{cases}
t^3 (2-t)^3 e^{-(t-1)^2/2} &\text{for $t\in \lbrack 0,2\rbrack$,}\\
0 &\text{otherwise.}\end{cases}\end{equation}
The approximation $\eta_+$ is defined by 
\begin{equation}\label{eq:patra}\eta_+(t) = h_H(t) t e^{-t^2/2},\end{equation} 
where
\begin{equation}\label{eq:dirich}\begin{aligned}
F_H(t) &= \frac{\sin(H \log y)}{\pi \log y},\\
h_H(t) &= (h \ast_M F_H)(y) = \int_0^\infty h(t y^{-1}) F_H(y)
\frac{dy}{y}
\end{aligned}\end{equation}
and $H$ is a positive constant to be set later. By (\ref{eq:zorbag}),
$M h_H = M h \cdot M F_H$. Now $F_H$ is just a Dirichlet kernel
under a change of variables; using this, we get that, for $\tau$ real, 
\begin{equation}M F_H(i \tau) = \begin{cases}
1 &\text{if $|\tau|< H$,}\\
1/2 &\text{if $|\tau|=H$,}\\
0 &\text{if $|\tau|>H$.}
\end{cases}\end{equation}
Thus,
\begin{equation}\label{eq:karmor}
M h_H(i \tau) = \begin{cases}
M h(i \tau) &\text{if $|\tau|< H$,}\\
\frac{1}{2} M h(i \tau) &\text{if $|\tau|=H$,}\\
0 &\text{if $|\tau|>H$.}
\end{cases}\end{equation}

As it turns out, $h$, $\eta_\circ$ and $M h$ (and hence $M h_H$)
 are relatively easy to
work with, whereas we can already see that
$h_H$ and $\eta_+$ have more complicated definitions. 
Part of our work will consist in expressing norms of $h_H$ and $\eta_+$
in terms of norms of $h$, $\eta_\circ$ and $M h$.

\subsection{The decay of $M h(i \tau)$}

Now, consider any
$\phi:\lbrack 0,\infty)\to \mathbb{C}$ that (a) has compact support
(or fast decay), (b)
 satisfies $\phi^{(k)}(t) t^{k-1} = O(1)$ for $t\to 0^+$ and
$0\leq k\leq 3$, and (c)
 is $C^2$ everywhere and quadruply differentiable outside a 
finite set of points.

By definition,
\[M\phi(s) = \int_0^\infty \phi(x) x^s \frac{dx}{x}.\]
Thus, by integration by parts, for
$\Re(s)>-1$ and $s\ne 0$,
\begin{equation}\label{eq:dotorwho}
\begin{aligned} &M\phi(s) = \int_0^\infty \phi(x) x^s \frac{dx}{x} = 
\lim_{t\to 0^+}
 \int_t^\infty \phi(x) x^s \frac{dx}{x} = 
- \lim_{t\to 0^+} \int_t^\infty \phi'(x) \frac{x^s}{s} dx \\ 
&= \lim_{t\to 0^+} \int_t^\infty \phi''(x)
\frac{x^{s+1}}{s (s+1)} dx = \lim_{t\to 0^+} 
- \int_t^\infty \phi^{(3)}(x) \frac{x^{s+2}}{s (s+1) (s+2)} dx \\ &= 
\lim_{t\to 0^+} \int_t^\infty \phi^{(4)}(x)
\frac{x^{s+3}}{s (s+1) (s+2) (s+3)} dx 
,\end{aligned}\end{equation}
where $\phi^{(4)}(x)$ is understood in the sense of distributions at the
finitely many points where it is not well-defined as a function.

Let $s = it$, $\phi = h$.
Let $C_k = \lim_{t\to 0^+} \int_t^\infty |h^{(k)}(x)| x^{k-1} dx$ for $0\leq k\leq 4$.
Then (\ref{eq:dotorwho}) gives us that
\begin{equation}\label{eq:kust}
Mh(i t) \leq \min\left(C_0, \frac{C_1}{|t|}, \frac{C_2}{|t| |t+i|},
\frac{C_3}{|t| |t+i| |t+2i|},\frac{C_4}{|t| |t+i| |t+2 i| |t+3i|}\right) .
\end{equation}
We must
 estimate the constants $C_j$, $0\leq j\leq 4$.

Clearly, $h(t) t^{-1} = O(1)$ as $t\to 0^+$, $h^{k}(t) = O(1)$ as $t\to
0^+$
for all $k\geq 1$, $h(2)=h'(2)=h''(2)=0$, and 
$h(x)$, $h'(x)$
and $h''(x)$ are all continuous. The function $h'''$ has 
a discontinuity at $t=2$. As we said, we understand $h^{(4)}$ in
the sense of distributions at $t=2$; for example,
$\lim_{\epsilon\to 0} \int_{2 -\epsilon}^{2+\epsilon} h^{(4)}(t) dt = 
\lim_{\epsilon\to 0} (h^{(3)}(2+\epsilon) - h^{(3)}(2-\epsilon))$.

Symbolic integration easily gives that
\begin{equation}\label{eq:nessu0}
C_0 = \int_0^2 t (2-t)^3 e^{t-1/2} dt = 92 e^{-1/2} - 12 e^{3/2}
= 2.02055184\dotsc\end{equation}
We will have to compute $C_{k}$, $1\leq k\leq 4$, with some care, due to the
absolute value involved in the definition.

The function $(x^2 (2-x)^3 e^{x-1/2})'= ((x^2 (2-x)^3)' + x^2 (2-x)^3) e^{x-1/2}$ 
has the same zeros as $H_1(x) = (x^2 (2-x)^3)' + x^2 (2-x)^3$, namely, 
$-4$, $0$, $1$ and $2$.
The sign of $H_1(x)$ (and hence
of $h'(x)$) is $+$ within
$(0,1)$ and $-$ within $(1,2)$. Hence
\begin{equation}\label{eq:nessu1}
\begin{aligned}C_{1} &= \int_0^\infty |h'(x)| dx = |h(1) - h(0)| + 
|h(2) - h(1)| =
 2 h(1) = 2 \sqrt{e}.
\end{aligned}\end{equation}

The situation with $(x^2 (2-x)^3 e^{x-1/2})''$ is similar: it has
zeros at the roots of $H_2(x)=0$, where $H_2(x) = H_1(x) + H_1'(x)$
(and, in general, $H_{k+1}(x) = H_k(x) + H_k'(x)$). This time, we will
prefer to find the roots numerically.
It is enough to find (candidates for)
the roots using any available tool\footnote{Routine \texttt{find\_root} in SAGE was used
here.}
and then check rigorously that the sign does change around the purported roots.
In this way, we check that $H_2(x)=0$ has two roots $\alpha_{2,1}$,
$\alpha_{2,2}$ in the interval
$(0,2)$, another root at $2$, and two more roots outside $\lbrack 0,2\rbrack$;
moreover,
\begin{equation}\label{eq:depard}\begin{aligned}
\alpha_{2,1} &= 0.48756597185712\dotsc,\\
\alpha_{2,2} &= 1.48777169309489\dotsc, \end{aligned}\end{equation}
where we verify the root using interval arithmetic.
The sign of $H_2(x)$ (and hence of $h''(x)$) is first $+$, then $-$, then $+$.
Write $\alpha_{2,0}=0$, $\alpha_{2,3}=2$. By integration by parts,
\begin{equation}\label{eq:nessu2}
\begin{aligned}C_{2} &= \int_0^\infty |h''(x)| x\; dx = 
\int_0^{\alpha_{2,1}} h''(x) x\; dx -
 \int_{\alpha_{2,1}}^{\alpha_{2,2}} h''(x) x\; dx + 
\int_{\alpha_{2,2}}^2 h''(x) x\; dx\\
&= \sum_{j=1}^3 (-1)^{j+1} \left( h'(x) x|_{\alpha_{2,j-1}}^{\alpha_{2,j}} - 
\int_{\alpha_{2,j-1}}^{\alpha_{2,j}} h'(x)\; dx\right)\\
&= 2 \sum_{j=1}^2 (-1)^{j+1} 
\left(h'(\alpha_{2,j}) \alpha_{2,j} - h(\alpha_{2,j})\right)
= 10.79195821037\dotsc
.\end{aligned}\end{equation}

To compute $C_{3}$, we proceed in the same way,
finding two roots of $H_3(x)=0$ (numerically) within the
interval $(0,2)$, viz.,
\[\begin{aligned}
\alpha_{3,1} &= 1.04294565694978\dotsc\\
\alpha_{3,2} &= 1.80999654602916\dotsc
\end{aligned}\]
The sign of $H_3(x)$ on the interval $\lbrack 0,2\rbrack$
is first $-$, then $+$, then $-$. Write $\alpha_{3,0}=0$, $\alpha_{3,3}=2$.
Proceeding as before -- with the only difference that the integration
by parts is iterated once now -- we obtain that
\begin{equation}
\begin{aligned}C_{3} &= \int_0^\infty |h'''(x)| x^2 dx
= \sum_{j=1}^3 (-1)^j \int_{\alpha_{3,j-1}}^{\alpha_{3,j}} h'''(x) x^2 dx\\
&= \sum_{j=1}^3 (-1)^j \left(h''(x) x^2 |_{\alpha_{3,j-1}}^{\alpha_{3,j}} -
\int_{\alpha_{3,j-1}}^{\alpha_{3,j}} h''(x) \cdot 2 x\right) dx\\
&= \sum_{j=1}^3 (-1)^j \left(h''(x) x^2 - h'(x) \cdot 2 x + 2 h(x)
\right)|_{\alpha_{3,j-1}}^{\alpha_{3,j}}\\
&=
2 \sum_{j=1}^2 (-1)^{j} (h''(\alpha_{3,j}) \alpha_{3,j}^2 - 
2 h'(\alpha_{3,j}) \alpha_{3,j} +
2 h(\alpha_{3,j}))\end{aligned}\end{equation}
and so interval arithmetic gives us
\begin{equation}\label{eq:nessu3}
C_{3} = 75.1295251672\dotsc
\end{equation}

The treatment of the integral in $C_{4}$ is very similar, at least
as first. There are two roots of $H_4(x)=0$ in the interval
$(0,2)$, namely,
\[\begin{aligned}
\alpha_{4,1} &= 0.45839599852663\dotsc\\
\alpha_{4,2} &= 1.54626346975533\dotsc\end{aligned}\]
The sign of $H_4(x)$ on the interval $\lbrack 0,2\rbrack$ is first 
$-$, $+$, then $-$. Using integration by parts as before, we obtain
\[\begin{aligned}
\int_{0^+}^{2^-} &\left| h^{(4)}(x) \right| x^3 dx 
\\ &= 
- \int_{0^+}^{\alpha_{4,1}} h^{(4)}(x) x^3 dx 
+ \int_{\alpha_{4,1}}^{\alpha_{4,2}} h^{(4)}(x) x^3 dx 
- \int_{\alpha_{4,1}}^{2^-} h^{(4)}(x) x^3 dx \\
&= 2 \sum_{j=1}^2 (-1)^j \left(h^{(3)}(\alpha_{4,j}) \alpha_{4,j}^3 -
3 h^{(2)}(\alpha_{4,j}) \alpha_{4,j}^2 + 6 h'(\alpha_{4,j}) \alpha_{4,j} -
6 h(\alpha_{4,j})\right)\\ &- \lim_{t\to 2^-} h^{(3)}(t) t^3
= 1152.69754862\dotsc,\end{aligned}\]
since $\lim_{t\to 0^+} h^{(k)}(t) t^k = 0$ for $0\leq k\leq 3$,
 $\lim_{t\to 2^-} h^{(k)}(t) = 0$ for $0\leq k\leq 2$ and
$\lim_{t\to 2^-} h^{(3)}(t) = - 24 e^{3/2}$.
Now
\[\int_{2^-}^\infty |h^{(4)}(x) x^3| dx = \lim_{\epsilon\to 0^+}
|h^{(3)}(2+\epsilon) - h^{(3)}(2-\epsilon)| \cdot 2^3 = 2^3\cdot 24 e^{3/2},\]
Hence
\begin{equation}\label{eq:nessu4}
C_{4} = \int_{0^+}^{2^-} \left| h^{(4)}(x) \right| x^3 dx 
+ 24 e^{3/2} \cdot 2^3 = 2013.18185012\dotsc
\end{equation}

We finish by remarking that can write down $M h$ explicitly:
\begin{equation}\label{eq:consti}Mh =
 - e^{-1/2} (-1)^{-s} 
(8 \gamma(s+2,-2) + 12 \gamma(s+3,-2) + 6 \gamma(s+4,-2) + \gamma(s+5,-2)),
\end{equation}
 where $\gamma(s,x)$ is the {\em (lower) 
incomplete Gamma function} 
\[\gamma(s,x) = \int_0^x e^{-t} t^{s-1} dt.\]
We will, however, find it easier to deal with $M h$ by means of the
bound (\ref{eq:kust}), in part because (\ref{eq:consti}) amounts to
an invitation to numerical instability.

For instance, it is easy to use (\ref{eq:kust}) to give a bound for
the $\ell_1$-norm of $Mh(i t)$. Since $C_4/C_3 > C_3/C_2 > C_2/C_1 > C_1/C_0$,
\[\begin{aligned}
&|M h(i t)|_1 = 2 \int_0^\infty Mh(it) dt \\ &\leq
2\left(C_0\cdot \frac{C_1}{C_0} + C_1 \int_{C_1/C_0}^{C_2/C_1} \frac{dt}{t} +
C_2 \int_{C_2/C_1}^{C_3/C_2} \frac{dt}{t^2} + 
C_3 \int_{C_3/C_2}^{C_4/C_3} \frac{dt}{t^3} + 
C_4 \int_{C_4/C_3}^\infty \frac{dt}{t^4}\right)\\
&= 2\left(C_1 + C_1 \log \frac{C_2 C_0}{C_1^2} + C_2 \left(\frac{C_1}{C_2} - 
\frac{C_2}{C_3}\right) + 
\frac{C_3}{2} \left(\frac{C_2^2}{C_3^2} - \frac{C_3^2}{C_4^2}\right)
+ \frac{C_4}{3} \cdot \frac{C_3^3}{C_4^3}\right),\end{aligned}\]
and so
\begin{equation}\label{eq:marpales}
|M h(i t)|_1\leq 16.1939176.\end{equation}
This bound is far from tight, but it will certainly be useful.

Similarly, $|(t +i) M h(it)|_1$ is at most two times
\[\begin{aligned}
&C_0\int_0^{\frac{C_1}{C_0}} |t+i| \; dt +
C_1 \int_{\frac{C_1}{C_0}}^{\frac{C_2}{C_1}} \left|1+ \frac{i}{t}\right| dt +
C_2 \int_{\frac{C_2}{C_1}}^{\frac{C_3}{C_2}} \frac{dt}{t} + 
C_3 \int_{\frac{C_3}{C_2}}^{\frac{C_4}{C_3}} \frac{dt}{t^2} + 
C_4 \int_{\frac{C_4}{C_3}}^\infty \frac{dt}{t^3}\\
&= \frac{C_0}{2} \left(\sqrt{\frac{C_1^4}{C_0^4} + \frac{C_1^2}{C_0^2}} + 
\sinh^{-1} \frac{C_1}{C_0}\right) + C_1
\left(\sqrt{t^2+1} + \log\left(\frac{\sqrt{t^2+1}-1}{t}\right)\right)|_{\frac{C_1}{C_0}}^{\frac{C_2}{C_1}}\\ &+
 C_2 \log \frac{C_3 C_1}{C_2^2} + C_3 \left(\frac{C_2}{C_3} - 
\frac{C_3}{C_4}\right) + 
\frac{C_4}{2} \frac{C_3^2}{C_4^2},\end{aligned}\]
and so
\begin{equation}\label{eq:marplat}
|(t+i) Mh(i t)|_1 \leq 27.8622803.
\end{equation}
\subsection{The difference $\eta_+-\eta_\circ$ in $\ell_2$ norm.}\label{subs:yosabi}
We wish to estimate the distance in $\ell_2$ norm between
$\eta_\circ$ and its approximation $\eta_+$. This will be an easy affair,
since, on the imaginary axis, the Mellin transform of $\eta_+$ is
just a truncation of the Mellin transform of $\eta_\circ$.

By (\ref{eq:cleo}) and (\ref{eq:patra}),
\begin{equation}\label{eq:estor}\begin{aligned}
|\eta_+-\eta_\circ|_2^2 &=
\int_0^\infty \left|h_H(t) t e^{-t^2/2} - h(t) t e^{-t^2/2}\right|^2 dt \\
&\leq
\left(\max_{t\geq 0} e^{-t^2} t^3 \right) \cdot
 \int_0^\infty |h_H(t)-h(t)|^2 \frac{dt}{t}.\end{aligned} \end{equation}
The maximum $\max_{t\geq 0} t^3 e^{-t^2}$ is $(3/2)^{3/2} e^{-3/2}$.
Since the Mellin transform is an isometry (i.e., (\ref{eq:victi}) holds),
\begin{equation}\label{eq:sev}\int_0^\infty |h_H(t)-h(t)|^2 \frac{dt}{t} = 
\frac{1}{2\pi} \int_{-\infty}^\infty |Mh_H(it)-Mh(it)|^2 dt =
\frac{1}{\pi} \int_{H}^\infty |Mh(it)|^2 dt.\end{equation}
 By (\ref{eq:kust}),
\begin{equation}\label{eq:shei}
\int_H^\infty
|Mh(it)|^2 dt \leq \int_H^\infty \frac{C_4^2}{t^8} dt
\leq \frac{C_4^2}{7 H^7}.
\end{equation}
Hence
\begin{equation}\label{eq:sevshei}
\int_0^\infty |h_H(t) - h(t)|^2 \frac{dt}{t} \leq \frac{C_4^2}{7 \pi H^7}.
\end{equation}
Using the bound (\ref{eq:nessu4}) for $C_4$, we conclude that 
\begin{equation}\label{eq:impath}
|\eta_+-\eta_\circ|_2 \leq \frac{C_4}{\sqrt{7 \pi}}
\left(\frac{3}{2e}\right)^{3/4} \cdot \frac{1}{H^{7/2}} \leq
\frac{274.856893}{H^{7/2}} .\end{equation}

It will also be useful to bound 
\[\left|\int_0^\infty (\eta_+(t) - \eta_\circ(t))^2 \log t\; dt\right|.\] 
This is at most
\[\left(\max_{t\geq 0} e^{-t^2} t^3 |\log t|\right)\cdot 
\int_0^\infty |h_H(t)- h(t)|^2 \frac{dt}{t}.
\]
Now
\[\begin{aligned}\max_{t\geq 0} e^{-t^2} t^3 |\log t| &= 
\max\left(\max_{t\in \lbrack 0,1\rbrack} e^{-t^2} t^3 (-\log t),
\max_{t\in \lbrack 1,5\rbrack} e^{-t^2} t^3 \log t\right)\\ &=
0.14882234545\dotsc 
\end{aligned}\]
where we find the maximum by the bisection method with $40$ iterations.\footnote{The bisection method (as described in, e.g., \cite[\S 5.2]{MR2807595}) can be used
to show that the minimum (or maximum) of $f$ on a compact interval $I$
lies within an interval (usually a very small one). Here, the bisection method
was carried rigorously, using interval arithmetic. The method was implemented 
by the author from the description in \cite[p. 87--88]{MR2807595}, using
D. Platt's interval arithmetic package.}
Hence, by (\ref{eq:sevshei}),
\begin{equation}\label{eq:lozhka}\begin{aligned}
\int_0^\infty (\eta_+(t)-\eta_\circ(t))^2 |\log t| dt &\leq
0.148822346 \frac{C_4^2}{7 \pi} \\ &\leq
\frac{27427.502}{H^7} \leq \left(\frac{165.61251}{H^{7/2}}\right)^2.
\end{aligned}\end{equation}

\subsection{Norms involving $\eta_+$}\label{subs:daysold}

Let us now bound some $\ell_1$- and $\ell_2$-norms involving $\eta_+$. Relatively crude bounds
will suffice in most cases. 

First, by (\ref{eq:impath}),
\begin{equation}\label{eq:mastodon}
|\eta_+|_2 \leq |\eta_\circ|_2 + |\eta_+ - \eta_\circ|_2 \leq
0.800129 + \frac{274.8569}{H^{7/2}},
\end{equation}
where we obtain 
\begin{equation}\label{eq:lamia}
|\eta_\circ|_2 = \sqrt{0.640205997\dotsc} = 0.8001287\dotsc
\end{equation} by symbolic integration.

Let us now bound $|\eta_+\cdot \log|_2^2$.
By isometry and (\ref{eq:harva}),
\[|\eta_+\cdot \log|_2^2 
= \frac{1}{2 \pi i}
\int_{\frac{1}{2} -i \infty}^{\frac{1}{2} +i \infty}
|M(\eta_+\cdot \log)(s)|^2 ds = 
\frac{1}{2 \pi i}
\int_{\frac{1}{2} -i \infty}^{\frac{1}{2} +i \infty}
|(M \eta_+)'(s)|^2 ds .
\]
Now, $(M \eta_+)'(1/2+it)$ equals $1/2\pi$ times the additive convolution of 
$M h_H(it)$ and $(M \eta_\diamondsuit)'(1/2+i t)$, where $\eta_\diamondsuit(t) = 
t e^{-t^2/2}$. Hence, by Young's
inequality, $|(M \eta_+)'(1/2+it)|_2\leq (1/2\pi) |M h_H(it)|_1 
|(M \eta_\diamondsuit)'(1/2+it)|_2$. 

Again by isometry and (\ref{eq:harva}),
\[|(M \eta_\diamondsuit)'(1/2+it)|_2 = \sqrt{2\pi} |\eta_\diamondsuit\cdot \log|_2.
\]
Hence, by (\ref{eq:marpales}),
\[|\eta_+\cdot \log|_2 \leq
\frac{1}{2 \pi} |M h_H(it)|_1 |\eta_\diamondsuit \cdot \log|_2 \leq
2.5773421 \cdot |\eta_\diamondsuit \cdot \log |_2.\]
Since, by symbolic integration,
\begin{equation}\label{eq:sonamo}\begin{aligned}
|\eta_\diamondsuit\cdot \log|_2 &\leq 
\sqrt{\frac{\sqrt{\pi}}{32} \left(8 (\log 2)^2 + 2\gamma^2 + \pi^2 + 
8(\gamma-2) \log 2 - 8 \gamma\right)}\\
&\leq 0.3220301,
\end{aligned}\end{equation}
we get that
\begin{equation}\label{eq:pamiatka}
|\eta_+\cdot \log|_2 \leq 0.8299818.\end{equation}

Let us bound $|\eta_+(t) t^\sigma|_1$ 
for $\sigma \in (-2,\infty)$.
By Cauchy-Schwarz and
Plancherel,
\begin{equation}\label{eq:lesalpes}\begin{aligned}
&|\eta_+(t) t^\sigma|_1 
= \left|h_H(t) t^{1+\sigma} e^{-t^2/2}\right|_1 \leq
\left|t^{\sigma+3/2} e^{-t^2/2}\right|_2 |h_H(t)/\sqrt{t}|_2 \\ &=
\left|t^{\sigma+3/2} e^{-t^2/2}\right|_2 
\sqrt{\int_0^\infty |h_H(t)|^2 \frac{dt}{t}}
=
\left|t^{\sigma+3/2} e^{-t^2/2}\right|_2 \cdot
 \sqrt{\frac{1}{2\pi} \int_{- H}^{H}
 |Mh(ir)|^2 dr}\\&\leq 
\left|t^{\sigma+3/2} e^{-t^2/2}\right|_2 
\cdot \sqrt{\frac{1}{2\pi} \int_{- \infty}^{\infty}
 |Mh(ir)|^2 dr}
= 
\left|t^{\sigma+3/2} e^{-t^2/2}\right|_2 \cdot 
|h(t)/\sqrt{t}|_2.
\end{aligned}\end{equation}
Since
\[\begin{aligned}
\left|t^{\sigma+3/2} e^{-t^2/2}\right|_2 &= 
\sqrt{\int_0^\infty e^{-t^2} t^{2\sigma+3} dt}
= \sqrt{\frac{\Gamma(\sigma+2)}{2}},\\
|h(t)/\sqrt{t}|_2 &=
\sqrt{\frac{31989}{8 e} - \frac{585 e^3}{8}} \leq
1.5023459
,\end{aligned}\]
we conclude that \begin{equation}\label{eq:paytoplay}
|\eta_+(t) t^{\sigma}|_1 \leq 1.062319
\cdot \sqrt{\Gamma(\sigma+2)}
\end{equation}
for $\sigma>-2$.
\subsection{Norms involving $\eta_+'$}\label{subs:weeksold}

By one of the standard transformation rules (see (\ref{eq:harva})),
the Mellin transform of $\eta_+'$ equals $- (s-1)\cdot  M\eta_+(s-1)$.
Since the Mellin transform is an isometry in the sense of (\ref{eq:victi}),
\[
|\eta_+'|_2^2 = \frac{1}{2\pi i } \int_{\frac{1}{2} - i\infty }^{
\frac{1}{2} + i \infty} \left|M(\eta_+')(s)\right|^2 ds =
 \frac{1}{2\pi i} \int_{-\frac{1}{2} - i \infty}^{- \frac{1}{2} + i \infty} 
\left|s\cdot M \eta_+(s)\right|^2 ds. 
\]
Recall that $\eta_+(t) = h_H(t) \eta_\diamondsuit(t)$, where
$\eta_\diamondsuit(t) = t e^{-t^2/2}$. Thus,
by (\ref{eq:mouv}), the function 
$M\eta_+(-1/2+it)$ equals $1/2\pi$ times the (additive)
convolution of $M h_H(it)$ and $M \eta_{\diamondsuit}(-1/2+it)$. Therefore,
for $s=-1/2+it$,
\begin{equation}\label{eq:grabai}
\begin{aligned} |s| \left|M \eta_+(s) \right| &= 
\frac{|s|}{2 \pi} \int_{-H}^H Mh(ir) M\eta_{\diamondsuit}(s-ir) dr\\ &\leq
\frac{3}{2\pi} \int_{-H}^H |ir - 1| |Mh(ir)| \cdot |s- ir| |M\eta_{\heartsuit}(s-ir)| 
dr\\
&= \frac{3}{2 \pi} (f\ast g)(t),
\end{aligned}\end{equation}
where $f(t) = |it - 1| |Mh(i t)|$ and 
$g(t) = |-1/2+it| |M\eta_{\diamondsuit}(-1/2+it)|$. 
(Since $|(-1/2+i(t-r))+(1+ir)| = |1/2+it| = |s|$, 
either $|-1/2+i(t-r)|\geq |s|/3$ or
$|1+ir|\geq 2 |s|/3$; hence $|s-ir| 
|ir - 1| = |-1/2+i(t-r)| |1+ir| 
\geq |s|/3$.)
By Young's inequality (in a special case that follows from Cauchy-Schwarz),
$|f\ast g|_2 \leq |f|_1 |g|_2$. 
By (\ref{eq:marplat}),
\[\begin{aligned}
|f|_1 = 
|(r+i) Mh(ir)|_1 \leq 27.8622803.
\end{aligned}\]
Yet again by Plancherel,
\[\begin{aligned}
|g|_2^2 &= \int_{-\frac{1}{2} - i\infty}^{-\frac{1}{2} + i \infty}
|s|^2 |M \eta_{\diamondsuit}(s)|^2 ds =  
\int_{\frac{1}{2} - i\infty}^{\frac{1}{2} + i \infty}
|(M(\eta_{\diamondsuit}')) (s)|^2 ds = 2 \pi |\eta_{\diamondsuit}'|_2^2 = 
\frac{3 \pi^{\frac{3}{2}}}{4}.
\end{aligned}\]
Hence
\begin{equation}\label{eq:miran}
|\eta_+'|_2 \leq \frac{1}{\sqrt{2\pi}} \cdot \frac{3}{2 \pi} |f\ast g|_2
 \leq \frac{1}{\sqrt{2 \pi}} \frac{3}{2 \pi} \cdot 27.8622803
\sqrt{\frac{3 \pi^{\frac{3}{2}}}{4}} \leq 10.845789.\end{equation}

Let us now bound $|\eta_+'(t) t^\sigma|_1$ for $\sigma\in (-1,\infty)$.
First of all,
\[\begin{aligned}
|\eta_+'(t) t^\sigma|_1
 &= \left|\left(h_H(t) t e^{-t^2/2}\right)'
t^\sigma\right|_1 \leq 
\left|\left(h_H'(t) t e^{-t^2/2} +
h_H(t) (1-t^2) e^{-t^2/2}\right) \cdot t^\sigma\right|_1\\
&\leq \left|h_H'(t) t^{\sigma+1} e^{-t^2/2}\right|_1 +
|\eta_+(t) t^{\sigma-1}|_1 + |\eta_+(t) t^{\sigma+1}|_1
.\end{aligned}\]
We can bound the last two terms by (\ref{eq:paytoplay}).
Much as in (\ref{eq:lesalpes}), we note that
\[
\left|h_H'(t) t^{\sigma+1} e^{-t^2/2}\right|_1 \leq
\left|t^{\sigma+1/2} e^{-t^2/2}\right|_2 |h_H'(t) \sqrt{t}|_2,\]
and then see that
\[\begin{aligned}
&|h_H'(t) \sqrt{t}|_2
 =
\sqrt{\int_0^\infty |h_H'(t)|^2 t\; dt}
=
 \sqrt{\frac{1}{2\pi} \int_{- \infty}^{\infty}
 |M(h_H')(1 + ir)|^2 dr}\\&=
 \sqrt{\frac{1}{2\pi} \int_{- \infty}^{\infty}
 |(-ir) Mh_H(ir)|^2 dr}
=
 \sqrt{\frac{1}{2\pi} \int_{- H}^{H}
 |(-ir) Mh(ir)|^2 dr}\\ &= 
 \sqrt{\frac{1}{2\pi} \int_{- H}^{H}
 |M(h')(1 + ir)|^2 dr} \leq
 \sqrt{\frac{1}{2\pi} \int_{-\infty}^{\infty}
 |M(h')(1 + ir)|^2 dr} = |h'(t) \sqrt{t}|_2,\end{aligned}\]
where we use the first rule in (\ref{eq:harva}) twice.
Since
\[\left|t^{\sigma+1/2} e^{-t^2/2}\right|_2 =
\sqrt{\frac{\Gamma(\sigma+1)}{2}},
\;\;\;\;
|h'(t) \sqrt{t}|_2 = \sqrt{\frac{103983}{16 e} - \frac{1899 e^3}{16}}
= 2.6312226,
\]
we conclude that
\begin{equation}\label{eq:uzsu}\begin{aligned}
|\eta_+'(t) t^\sigma|_1 &\leq 1.062319\cdot (\sqrt{\Gamma(\sigma+1)}
 + \sqrt{\Gamma(\sigma+3)}) + 
\sqrt{\frac{\Gamma(\sigma+1)}{2}} \cdot 2.6312226\\
&\leq 2.922875 \sqrt{\Gamma(\sigma+1)} + 1.062319 \sqrt{\Gamma(\sigma+3)}
\end{aligned}\end{equation}
for $\sigma>-1$.
\subsection{The $\ell_\infty$-norm of $\eta_+$}\label{subs:byron}

Let us now get a bound for $|\eta_+|_\infty$. 
Recall that $\eta_+(t) = h_H(t) \eta_\diamondsuit(t)$, where 
$\eta_\diamondsuit(t) = t e^{-t^2/2}$.
Clearly
 \begin{equation}\label{eq:jorat}\begin{aligned}
|\eta_+|_\infty &= |h_H(t) \eta_\diamondsuit(t)|_\infty
\leq |\eta_\circ|_\infty + |(h(t)-h_H(t)) \eta_\diamondsuit(t)|_\infty\\
&\leq |\eta_\circ|_\infty + \left|\frac{h(t)-h_H(t)}{t}\right|_\infty 
|\eta_\diamondsuit(t) t|_\infty.
\end{aligned}\end{equation}
Taking derivatives, we easily see that
\[|\eta_\circ|_\infty = \eta_\circ(1) = 1, 
\;\;\;\;\;\;\;\;\;\;
|\eta_\diamondsuit(t) t|_\infty = 2/e. 
\]
It remains to bound $|(h(t)-h_H(t))/t|_\infty$. 
By (\ref{eq:dirich2}),
\begin{equation}\label{eq:narco}
h_H(t) = \int_{\frac{t}{2}}^\infty h(t y^{-1}) \frac{\sin(H \log y)}{\pi
\log y} \frac{dy}{y}
= \int_{- H \log \frac{2}{t}}^\infty h\left(\frac{t}{e^{w/H}}\right)
\frac{\sin w}{\pi w} dw.\end{equation}
The {\em sine integral} 
\[\Si(x) = \int_0^x \frac{\sin t}{t} dt\]
is defined for all $x$; it tends to $\pi/2$ as $x\to +\infty$ and
to $-\pi/2$ as $x\to -\infty$ (see \cite[(5.2.25)]{MR0167642}).  We apply integration by parts
to the second integral in (\ref{eq:narco}), and obtain
\[\begin{aligned}
h_H(t) - h(t) &= - \frac{1}{\pi} \int_{-H \log \frac{2}{t}}^\infty
\left(\frac{d}{dw} h\left(\frac{t}{e^{w/H}}\right)\right) \Si(w) dw - h(t)\\
&= -\frac{1}{\pi} \int_{0}^\infty
\left(\frac{d}{dw} h\left(\frac{t}{e^{w/H}}\right)\right) \left(
\Si(w) - \frac{\pi}{2}\right) dw \\ &-
\frac{1}{\pi} \int_{-H \log \frac{2}{t}}^0
\left(\frac{d}{dw} h\left(\frac{t}{e^{w/H}}\right)\right) \left(
\Si(w) + \frac{\pi}{2}\right) dw.
\end{aligned}\]
Now
\[\left|\frac{d}{dw} h\left(\frac{t}{e^{w/H}}\right)\right| = 
\frac{t e^{-w/H}}{H} \left|h'\left(\frac{t}{e^{w/H}}\right)\right| \leq
\frac{t |h'|_\infty}{H e^{w/H}}.
\]
Integration by parts easily yields the bounds $|\Si(x)-\pi/2|<2/x$ for
$x>0$ and $|\Si(x)+\pi/2|<2/|x|$ for $x<0$; we also know that $0\leq \Si(x)
\leq x < \pi/2$
for $x\in \lbrack 0,1\rbrack$ and 
$-\pi/2 < x \leq \Si(x)\leq 0$ for $x\in \lbrack -1,0\rbrack$. Hence
\[\begin{aligned}
|h_H(t)-h(t)| &\leq \frac{2 t |h'|_\infty}{\pi H} \left(
\int_0^1 \frac{\pi}{2} e^{-w/H} dw 
+ \int_1^\infty \frac{2 e^{-w/H}}{w} dw\right)\\
&= t |h'|_\infty \cdot \left(
(1 - e^{-1/H})
 + \frac{4}{\pi} \frac{E_1(1/H)}{H}\right),
\end{aligned}\]
where $E_1$ is the {\em exponential integral}
\[E_1(z) = \int_z^\infty \frac{e^{-t}}{t} dt.\] 
By \cite[(5.1.20)]{MR0167642},
\[0 < E_1(1/H) < \frac{\log(H+1)}{e^{1/H}},\]
and, since $\log(H+1) = \log H + \log(1+1/H) < \log H + 1/H <
(\log H) (1+1/H) < (\log H) e^{1/H}$ for $H\geq e$, we see that
this gives us that $E_1(1/H)<\log H$ (again for $H\geq e$, as is the case).
Hence
\begin{equation}\label{eq:havana}
\frac{|h_H(t)-h(t)|}{t} < |h'|_\infty \cdot
\left(1 - e^{-\frac{1}{H}} + \frac{4}{\pi} \frac{\log H}{H}\right) <
|h'|_\infty \cdot
\frac{1 + \frac{4}{\pi} \log H}{H},
\end{equation}
and so, by (\ref{eq:jorat}),
\[
|\eta_+|_\infty
 \leq 1 + \frac{2}{e} \left|\frac{h(t)-h_H(t)}{t}\right|_\infty 
< 1 + \frac{2}{e} |h'|_\infty \cdot \frac{1 + \frac{4}{\pi} \log H}{H}.\]
By (\ref{eq:depard}) and interval arithmetic, we determine that
\begin{equation}\label{eq:morno}
|h'|_\infty = |h'(\alpha_{2,2})| \leq 2.805820379671,
\end{equation}
where $\alpha_{2,2}$ is a root of $h''(x)=0$
as in (\ref{eq:depard}). We have proven
\begin{equation}\label{eq:malgache}
|\eta_+|_\infty < 1 + \frac{2}{e}  \cdot 2.80582038\cdot 
\frac{1 + \frac{4}{\pi} \log H}{H}
< 1 + 2.06440727\cdot  \frac{1 + \frac{4}{\pi} \log H}{H}.
\end{equation}

We will need three other bounds of this kind, namely, for $\eta_+(t) \log t$,
$\eta_+(t)/t$ and $\eta_+(t) t$.
We start as in (\ref{eq:jorat}):
\begin{equation}\label{eq:rasal}\begin{aligned}
|\eta_+ \log t|_\infty &\leq |\eta_\circ \log t|_\infty + 
|(h(t)-h_H(t)) \eta_\diamondsuit(t) \log t|_\infty\\ &\leq |\eta_\circ \log t|_\infty
+ |(h-h_H(t))/t|_\infty |\eta_\diamondsuit(t) t \log t|_\infty,\\
|\eta_+(t)/t|_\infty &\leq |\eta_\circ(t)/t|_\infty 
+ |(h-h_H(t))/t|_\infty |\eta_\diamondsuit(t)|_\infty\\
|\eta_+(t) t|_\infty &\leq |\eta_\circ(t) t|_\infty 
+ |(h-h_H(t))/t|_\infty |\eta_\diamondsuit(t) t^2|_\infty
.\end{aligned}\end{equation}
By the bisection method with $30$ iterations, implemented with interval
arithmetic,
\[
|\eta_\circ(t) \log t|_\infty \leq 0.279491, \;\;\;\;\;\;\;\;\;\;
|\eta_\diamondsuit(t) t \log t|_\infty \leq 0.3811561. 
\]
Hence, by (\ref{eq:havana}) and (\ref{eq:morno}),
\begin{equation}\label{eq:dalida}
|\eta_+ \log t|_\infty \leq 0.279491 + 1.069456
 \cdot \frac{1 + \frac{4}{\pi} \log H}{H}.\end{equation}
By the bisection method with $32$ iterations,
\[|\eta_\circ(t)/t|_\infty \leq 1.08754396
.\]
(We can also obtain this by solving $(\eta_\circ(t)/t)'=0$ symbolically.)
It is easy to show that $|\eta_\diamondsuit|_\infty = 1/\sqrt{e}$. 
Hence, again by (\ref{eq:havana})
and (\ref{eq:morno}),
\begin{equation}\label{eq:gobmark}
|\eta_+(t)/t|_\infty \leq 1.08754396 + 1.70181609 \cdot \frac{1 + \frac{4}{\pi}
\log H}{H} .\end{equation}
By the bisection method with $32$ iterations,
\[|\eta_\circ(t) t|_\infty \leq 1.06473476.\]
Taking derivatives, we see that
$|\eta_\diamondsuit(t) t^2|_\infty = 3^{3/2} e^{-3/2}$. Hence, yet again by (\ref{eq:havana}) and (\ref{eq:morno}),
\begin{equation}\label{eq:shchedrin}
\left|\eta_+(t) t\right|_\infty \leq  1.06473476 +
3.25312 \cdot \frac{1 + \frac{4}{\pi} \log H}{H} .
\end{equation}
\bibliographystyle{alpha}
\bibliography{arcs}

\def\cprime{$'$} \def\cprime{$'$} \def\cprime{$'$}
\begin{thebibliography}{OLBC10}

\bibitem[AS64]{MR0167642}
M.~Abramowitz and I.~A. Stegun.
\newblock {\em Handbook of mathematical functions with formulas, graphs, and
  mathematical tables}, volume~55 of {\em National Bureau of Standards Applied
  Mathematics Series}.
\newblock For sale by the Superintendent of Documents, U.S. Government Printing
  Office, Washington, D.C., 1964.

\bibitem[BBO10]{Mellin}
J.~Bertrand, P.~Bertrand, and J.-P. Ovarlez.
\newblock Mellin transform.
\newblock In A.~D. Poularikas, editor, {\em Transforms and applications
  handbook}. CRC Press, Boca Raton, FL, 2010.

\bibitem[BM98]{MR1652147}
Martin Berz and Kyoko Makino.
\newblock Verified integration of {ODE}s and flows using differential algebraic
  methods on high-order {T}aylor models.
\newblock {\em Reliab. Comput.}, 4(4):361--369, 1998.

\bibitem[Boo06]{MR2263990}
A.~R. Booker.
\newblock Turing and the {R}iemann hypothesis.
\newblock {\em Notices Amer. Math. Soc.}, 53(10):1208--1211, 2006.

\bibitem[dB81]{MR671583}
N.~G. de~Bruijn.
\newblock {\em Asymptotic methods in analysis}.
\newblock Dover Publications Inc., New York, third edition, 1981.

\bibitem[GR00]{MR1773820}
I.~S. Gradshteyn and I.~M. Ryzhik.
\newblock {\em Table of integrals, series, and products}.
\newblock Academic Press Inc., San Diego, CA, sixth edition, 2000.
\newblock Translated from the Russian, Translation edited and with a preface by
  Alan Jeffrey and Daniel Zwillinger.

\bibitem[Har66]{MR0201267}
G.~H. Hardy.
\newblock {\em Collected papers of {G}. {H}. {H}ardy ({I}ncluding {J}oint
  papers with {J}. {E}. {L}ittlewood and others). {V}ol. {I}}.
\newblock Edited by a committee appointed by the London Mathematical Society.
  Clarendon Press, Oxford, 1966.

\bibitem[HB79]{MR532980}
D.~R. Heath-Brown.
\newblock The fourth power moment of the {R}iemann zeta function.
\newblock {\em Proc. London Math. Soc. (3)}, 38(3):385--422, 1979.

\bibitem[Hela]{HelfMaj}
H.~A. Helfgott.
\newblock Major arcs for {G}oldbach's problem.
\newblock Preprint. Available at {\texttt{arXiv:1203.5712}}.

\bibitem[Helb]{Helf}
H.~A. Helfgott.
\newblock Minor arcs for {G}oldbach's problem.
\newblock Preprint. Available as {\texttt{arXiv:1205.5252}}.

\bibitem[Helc]{HelfTern}
H.~A. Helfgott.
\newblock The {T}ernary {G}oldbach {Conjecture} is true.
\newblock Preprint.

\bibitem[HL23]{MR1555183}
G.~H. Hardy and J.~E. Littlewood.
\newblock Some problems of `{P}artitio numerorum'; {III}: {O}n the expression
  of a number as a sum of primes.
\newblock {\em Acta Math.}, 44(1):1--70, 1923.

\bibitem[IK04]{MR2061214}
H.~Iwaniec and E.~Kowalski.
\newblock {\em Analytic number theory}, volume~53 of {\em American Mathematical
  Society Colloquium Publications}.
\newblock American Mathematical Society, Providence, RI, 2004.

\bibitem[Kad05]{MR2140161}
H.~Kadiri.
\newblock Une r\'egion explicite sans z\'eros pour la fonction {$\zeta$} de
  {R}iemann.
\newblock {\em Acta Arith.}, 117(4):303--339, 2005.

\bibitem[Kn{\"u}99]{Profbis}
O.~Kn{\"u}ppel.
\newblock {PROFIL}/{BIAS}, February 1999.
\newblock version 2.

\bibitem[Leh66]{MR0202686}
R.~Sherman Lehman.
\newblock On the difference {$\pi (x)-{\rm li}(x)$}.
\newblock {\em Acta Arith.}, 11:397--410, 1966.

\bibitem[McC84]{MR726004}
K.~S. McCurley.
\newblock Explicit estimates for the error term in the prime number theorem for
  arithmetic progressions.
\newblock {\em Math. Comp.}, 42(165):265--285, 1984.

\bibitem[MV07]{MR2378655}
H.~L. Montgomery and R.~C. Vaughan.
\newblock {\em Multiplicative number theory. {I}. {C}lassical theory},
  volume~97 of {\em Cambridge Studies in Advanced Mathematics}.
\newblock Cambridge University Press, Cambridge, 2007.

\bibitem[Ned06]{VNODELP}
N.~S. Nedialkov.
\newblock {VNODE}-{LP}: a validated solver for initial value problems in
  ordinary differential equations, July 2006.
\newblock version 0.3.

\bibitem[OLBC10]{MR2723248}
F.~W.~J. Olver, D.~W. Lozier, R.~F. Boisvert, and Ch.~W. Clark, editors.
\newblock {\em N{IST} handbook of mathematical functions}.
\newblock U.S. Department of Commerce National Institute of Standards and
  Technology, Washington, DC, 2010.
\newblock With 1 CD-ROM (Windows, Macintosh and UNIX).

\bibitem[Olv58]{MR0094496}
F.~W.~J. Olver.
\newblock Uniform asymptotic expansions of solutions of linear second-order
  differential equations for large values of a parameter.
\newblock {\em Philos. Trans. Roy. Soc. London. Ser. A}, 250:479--517, 1958.

\bibitem[Olv59]{MR0109898}
F.~W.~J. Olver.
\newblock Uniform asymptotic expansions for {W}eber parabolic cylinder
  functions of large orders.
\newblock {\em J. Res. Nat. Bur. Standards Sect. B}, 63B:131--169, 1959.

\bibitem[Olv61]{MR0131580}
F.~W.~J. Olver.
\newblock Two inequalities for parabolic cylinder functions.
\newblock {\em Proc. Cambridge Philos. Soc.}, 57:811--822, 1961.

\bibitem[Olv65]{MR0185350}
F.~W.~J. Olver.
\newblock On the asymptotic solution of second-order differential equations
  having an irregular singularity of rank one, with an application to
  {W}hittaker functions.
\newblock {\em J. Soc. Indust. Appl. Math. Ser. B Numer. Anal.}, 2:225--243,
  1965.

\bibitem[Olv74]{MR0435697}
F.~W.~J. Olver.
\newblock {\em Asymptotics and special functions}.
\newblock Academic Press [A subsidiary of Harcourt Brace Jovanovich,
  Publishers], New York-London, 1974.
\newblock Computer Science and Applied Mathematics.

\bibitem[Plaa]{PlattPi}
D.~Platt.
\newblock Computing $\pi(x)$ analytically.
\newblock To appear in {\em Math. Comp.}. Available as
  {\texttt{arXiv:1203.5712}}.

\bibitem[Plab]{Plattfresh}
D.~Platt.
\newblock Numerical computations concerning {G}{R}{H}.
\newblock Preprint. Available at {\texttt{arXiv:1305.3087}}.

\bibitem[Pla11]{Platt}
D.~Platt.
\newblock {\em Computing degree $1$ L-functions rigorously}.
\newblock PhD thesis, Bristol University, 2011.

\bibitem[Ros41]{MR0003018}
B.~Rosser.
\newblock Explicit bounds for some functions of prime numbers.
\newblock {\em Amer. J. Math.}, 63:211--232, 1941.

\bibitem[Tem10]{MR2655352}
N.~M. Temme.
\newblock Parabolic cylinder functions.
\newblock In {\em N{IST} handbook of mathematical functions}, pages 303--319.
  U.S. Dept. Commerce, Washington, DC, 2010.

\bibitem[Tru]{Trudgian}
T.~S. Trudgian.
\newblock An improved upper bound for the error in the zero-counting formulae
  for {D}irichlet {$L$}-functions and {D}edekind zeta-functions.
\newblock Preprint.

\bibitem[Tuc11]{MR2807595}
W.~Tucker.
\newblock {\em Validated numerics: A short introduction to rigorous
  computations}.
\newblock Princeton University Press, Princeton, NJ, 2011.

\bibitem[Tur53]{MR0055785}
A.~M. Turing.
\newblock Some calculations of the {R}iemann zeta-function.
\newblock {\em Proc. London Math. Soc. (3)}, 3:99--117, 1953.

\bibitem[TV03]{MR1993339}
N.~M. Temme and R.~Vidunas.
\newblock Parabolic cylinder functions: examples of error bounds for asymptotic
  expansions.
\newblock {\em Anal. Appl. (Singap.)}, 1(3):265--288, 2003.

\bibitem[Vin37]{Vin}
I.~M. Vinogradov.
\newblock Representation of an odd number as a sum of three primes.
\newblock {\em Dokl. Akad. Nauk. SSR}, 15:291--294, 1937.

\bibitem[Wed03]{Wed}
S.~Wedeniwski.
\newblock {Z}eta{G}rid - {C}omputational verification of the {R}iemann
  hypothesis.
\newblock Conference in {N}umber {T}heory in honour of {P}rofessor {H}. {C}.
  {W}illiams, Banff, Alberta, Canada, May 2003.

\bibitem[Whi03]{Whi}
E.~T. Whittaker.
\newblock On the functions associated with the parabolic cylinder in harmonic
  analysis.
\newblock {\em Proc. London Math. Soc.}, 35:417--427, 1903.

\bibitem[Wig20]{Wigert}
S.~Wigert.
\newblock Sur la th\'eorie de la fonction $\zeta(s)$ de {R}iemann.
\newblock {\em Ark. Mat.}, 14:1--17, 1920.

\bibitem[Won01]{MR1851050}
R.~Wong.
\newblock {\em Asymptotic approximations of integrals}, volume~34 of {\em
  Classics in Applied Mathematics}.
\newblock Society for Industrial and Applied Mathematics (SIAM), Philadelphia,
  PA, 2001.
\newblock Corrected reprint of the 1989 original.

\end{thebibliography}
\end{document}